\documentclass[10pt]{amsart}
\usepackage{amssymb}
\usepackage{amsbsy}
\usepackage{amscd}
\usepackage[mathscr]{eucal}
\usepackage{verbatim}
\usepackage{version}
\usepackage{color}
\usepackage[matrix,arrow]{xy}

\usepackage{xr}

\def\cal{\mathcal}
\def\Bbb{\mathbb}
\def\frak{\mathfrak}

\newenvironment{NB}{
\color{red}{\bf NB}. \footnotesize 
}{}
\excludeversion{NB}
\newenvironment{NB2}{
\color{blue}{\bf NB}. \footnotesize
}{}

\newcommand{\wt}{\widetilde}

\newcommand{\pH}{{^p H}}

\newcommand{\baM}{{ M}}
\newcommand{\Pic}{\operatorname{Pic}}
\newcommand{\Quot}{\operatorname{Quot}}
\newcommand{\Supp}{\operatorname{Supp}}
\newcommand{\ch}{\operatorname{ch}}

\newcommand{\Coh}{\operatorname{Coh}}

\newcommand{\Ext}{\operatorname{Ext}}
\newcommand{\Hom}{\operatorname{Hom}}

\newcommand{\im}{\operatorname{im}}

\newcommand{\rk}{\operatorname{rk}}

\newcommand{\chr}{\operatorname{char}}

\newcommand{\NS}{\operatorname{NS}}
\newcommand{\coker}{\operatorname{coker}}
\newcommand{\td}{\operatorname{td}}
\newcommand{\Alb}{\operatorname{Alb}}

\newcommand{\Amp}{\operatorname{Amp}}

\newcommand{\Per}{\operatorname{Per}}

\newcommand{\alg}{\operatorname{alg}}

\newcommand{\Stab}{\operatorname{Stab}}

\newcommand{\Sym}{\operatorname{Sym}}

\font\b=cmr10 scaled \magstep5
\def\bigzerou{\smash{\lower1.7ex\hbox{\b 0}}}

\numberwithin{equation}{section}

\theoremstyle{plain}
 \newtheorem{thm}{Theorem}[section]
 \newtheorem{lem}[thm]{Lemma}
 \newtheorem{prop}[thm]{Proposition}
 \newtheorem{cor}[thm]{Corollary}
 
\theoremstyle{definition}
 \newtheorem{defn}[thm]{Definition}
 
\theoremstyle{remark}
 \newtheorem{rem}[thm]{Remark}

\begin{document}

\title{Wall crossing
of the moduli spaces of 
perverse coherent sheaves on a blow-up.}
\author{K\={o}ta Yoshioka}
\address{Department of Mathematics, Faculty of Science,
Kobe University,
Kobe, 657, Japan
}
\email{yoshioka@math.kobe-u.ac.jp}

\thanks{
The author is supported by the Grant-in-aid for 
Scientific Research (No.\ 26287007, \ 24224001), JSPS}


\begin{abstract}
We give a remark on the wall crossing behavior of perverse coherent 
sheaves on a blow-up \cite{perv}, \cite{perv2} 
and stability condition \cite{Toda}
\end{abstract}

\maketitle

\section{Introduction}

Let $X$ be a smooth projective surface over an algebraically 
closed field $k$ of characteristic 0.
For $(\beta, \omega) \in \NS(X)_{\Bbb R} \times \Amp(X)_{\Bbb R}$,
Arcara and Bertram \cite{AB} constructed stability conditions
$\sigma_{(\beta,\omega)}$
such that
the structure sheaves of points $k_x$ $(x \in X)$ are stable of phase 1.
In \cite{Toda}, \cite{Toda2}, Toda constructed new examples of stability 
conditions $\sigma_{(\beta,\omega)}$ which is regarded to an extension
of Arcara and Bertram's examples to non-ample $\omega$.
Moreover Toda showed that new examples are related to stability
conditions on blow-downs of $(-1)$-curves on $X$.
Assume that $\omega$ is close to ample in $\NS(X)_{\Bbb R}$.
Then $\sigma_{(\beta,\omega)}$ is related to stability condition
on a blow-down $\pi:X \to Y$ of a $(-1)$-curve $C$ of $X$.  
In this case, we can set $\omega:=\pi^*(L)+t C$, where 
$L \in \Amp(Y)$ and $|t|$ is sufficiently small. 
If $t<0$, then $\sigma_{(\beta,\omega)}$ is an example of
Arcara and Bertram \cite{AB}.
If $t=0$, then  
instead of using a torsion pair of $\Coh(X)$, 
$\sigma_{(0,\omega)}$ is constructed by
using a similar torsion pair of       
a category of perverse coherent sheaves
${^{-1}\Per}(X/Y)$ \cite{Br:4}. 
Since $k_x$ $(x \in C)$ becomes reducible in $^{-1}\Per(X/Y)$,
$k_x$ is properly $\sigma_{(0,\omega)}$-semi-stable.
If $t>0$, then $k_x$ $(x \in C)$ is not $\sigma_{(0,\omega)}$-semi-stable.
In this case, ${\bf L}\pi^*(k_y)$ is $\sigma_{(0,\omega)}$-stable and 
$\sigma_{(0,\omega)}$-semi-stable objects 
are parameterized by $Y$.
  
In this note, we shall give exmaples of Bridgeland semi-stable objects
on $X$.
In \cite{perv} and \cite{perv2},
Nakajima and the author studied 
the relation of 
moduli spaces of Gieseker semi-stable sheaves on
$Y$  and $X$   
by looking at the wall crossing 
behavior in the category of perverse coherent sheaves.
As expected from Toda's papers and also the relation with the Gieseker
semi-stability in the large volume
limit, we shall show that Gieseker type semi-stability
of perverse coherent sheaves corresponds to
Bridgeland semi-stability in the large volume limit.
Then we shall explain our wall crossing behavior
of the moduli of perverse coherent sheaves 
in terms of Bridgeland stability condition 
(subsection \ref{subsect:wall-crossing:blow-up}). 


For the proof of our results, we also need to study the large volume limit
where $\omega$ is ample.
As is proved by Bridgeland \cite{Br:3} and Lo and Qin \cite{Lo-Qin}, 
Bridgeland's stability $\sigma_{(\beta,\omega)}$ is
related to (twisted) Gieseker stability 
if $(\omega^2)$ is sufficiently large (Proposition \ref{prop:large}). 
On the other hand if an object $E \in {\bf D}(X)$ satisfies
$(c_1(E)-\rk E \beta_0,\omega_0)=0$,
then 
$\sigma_{(\beta_0,\omega_0)}$-stability of $E$ is related to $\mu$-stability
and the moduli space is related to the Uhlenbeck compactification
of the moduli of $\mu$-stable locally free sheaves.
By looking at the wall and chamber structure near 
$\sigma_{(\beta_0,\omega_0)}$,
we shall show that each adjacent chamber of $(\beta_0,\omega_0)$
contains a point at the large volume limit, which
implies that each chamber corresponds to Gieseker semi-stabilty.
Thus we may say that
Bridgeland stability unifies (twisted) Gieseker stabilities and
the $\mu$-stability.
We call these chambers {\it Gieseker chambers}.
We explain the usual wall crossing of moduli spaces of
Gieseker semi-stable sheaves \cite{EG}, \cite{FQ}, \cite{MW}
and also those of Uhlenbeck compactifications \cite{HL}
in terms of Bridgeland stability conditions.
The blow-up case is a slight generalization of this consideration.

Let us explain the organization of this article.
In section \ref{sect:back}, 
we introduce a bilinear form on the algebraic cohomology
group $H^*(X,{\Bbb Q})_{\alg}$, which is useful to state the Bogomolov
inequality of $\mu$-semi-stable sheaves.
If $X$ is a K3 surface or an abelian surface, 
this bilinear form is nothing but the Mukai's bilinear form.  
We next recall categories ${\frak C}$ of perverse coherent sheaves
associated to a birational map
$\pi:X \to Y$ of surfaces \cite{PerverseI}, 
and then introduce a stability condition
by tilting categories of perverse coherent sheaves.
In section \ref{sect:blowup}, we consider the case of 
a blow-up $\pi:X \to Y$ of a smooth point of $Y$.
We explain local projective generators of ${\frak C}$ and
the moduli of semi-stable perverse coherent sheaves.
We also explain some results in \cite{perv} and \cite{perv2}.

In section \ref{sect:chamber},
we shall study wall and chamber structure of the moduli of
$\sigma$-semi-stable objects 
${\cal M}_\sigma(v)$, where $v$ is the Chern character
of the objects.
We construct a map $\xi$ from the subspace of stability conditions
parameterized by $\NS(X)_{\Bbb R} \times \Amp(X)_{\Bbb R}$
to the positive cone of $v^\perp$
with respect to
the bilinear form on $H^*(X,{\Bbb Q})_{\alg}$.
Then we show that the wall and chamber structure
on $\NS(X)_{\Bbb R} \times \Amp(X)_{\Bbb R}$ 
is the pull-back of a wall and chamber structure on $v^\perp$.
In particular, all semi-stabilities on a fiber are the same.
The wall and chamber structure in $v^\perp$ is used 
to compare Gieseker chambers.
In section \ref{sect:stability-blowup},
we shall partially generalize
results in the previous section
to the blow-up case.
We also construct an analytic neighborhood $U$ of the
origin of ${\Bbb C}^2$ which parameterizes stability conditions.
It is used to study the wall crossing for the blow-up case.

In section \ref{sect:wall-crossing}, we first recall a homomorphism
from $v^\perp$ to $\NS(M_H(v))_{\Bbb Q}$, where
$M_H(v)$ is the moduli spaces of Gieseker semi-stable
sheaves. 
It is expected to be surjective modulo $\NS(\Alb(M_H(v)))_{\Bbb Q}$ 
under suitable conditions.
Composing this homomorphism with  
the map constructed in section \ref{sect:chamber},
we have a map $\NS(X)_{\Bbb Q} \times \Amp(X)_{\Bbb Q} 
\to \NS(M_H(v))_{\Bbb Q}$.
By the work of Bayer and Macri \cite{BM},
the image of Gieseker chamber is contained in the nef cone
of the moduli space.
Then we shall explain
Matsuki-Wentworth wall crossing of Gieseker semi-stability
in terms of Bridgeland stability conditions.
Finally we explain the wall crossing of the moduli spaces of perverse
coherent sheaves as a wall crossing of Bridgeland stability conditions.
In Appendix, we give some technical results.
In particular, we explain the moduli of semi-stable objects 
in the large volume limit, which is a generalization of 
\cite{MYY:2011:1}. 
We also give a trivial example of Gieseker chamber 
in section \ref{subsect:ex-Gieseker}.
 

During preparation of this note,
Bertram and Martinez informed us they also studied the wall
crossing of twisted Gieseker stability by using Bridgeland stability, and
get the same result of section \ref{subsect:MW} (\cite{BeM}).   

\begin{NB}
$E$ is a
$\sigma_{(\beta,\omega)}$-stable object of phase 0 with
$(\beta,\omega)=(c_1(E),\omega)/\rk E$, then
$E$ is a $\mu$-stable locally free sheaf.
Thus the moduli space can be regarded as the 
Uhlenbeck compactification of the moduli of $\mu$-stable
locally free sheaves
\end{NB}

\section{Background materials}\label{sect:back}

\subsection{
Basic notations}
Let $X$ be a smooth projective surface over an algebraically closed
field $k$ of characteristic 0.
As in Mukai lattice on abelian surfaces, let us introduce a bilinear
form on $H^*(X,{\Bbb Q})_{\alg}:={\Bbb Q} \oplus \NS(X)_{\Bbb Q}
\oplus {\Bbb Q}$. 
For $x=(x_0,x_1,x_2),y=(y_0,y_1,y_2) 
\in H^*(X,{\Bbb Q})_{\alg}$,
we set
\begin{equation}
(x,y):=(x_1,y_1)-x_0 y_2-x_2 y_0 \in {\Bbb Q}.
\end{equation}
Let $\varrho_X:=(0,0,1)$ be the fundamental class of $X$.
We also use the notation
$x=x_0+x_1+x_0 \varrho_X$ to denote 
$x=(x_0,x_1,x_2)$.
For $x=(x_0,x_1,x_2)$, we set
\begin{equation}
\begin{split}
\rk x:= & x_0,\\
c_1(x):=& x_1.
\end{split}
\end{equation}
For $E \in {\bf D}(X)$, we set
$v(E):=\ch(E) \in H^*(X,{\Bbb Q})_{\alg}$.
Since 
\begin{equation}
\begin{split}
\Delta(E):= & c_2(E)-\frac{\rk E-1}{2\rk E}(c_1(E)^2)\\
=&
\frac{c_2(E^{\vee} \otimes E)}{2 \rk E}=\frac{(\ch(E)^2)}{2 \rk E},
\end{split}
\end{equation} 
the Bogomolov inequality is the following.
\begin{lem}
Assume that $X$ is defined over a field of characteristic 0.
Let $E$ be a $\mu$-semi-stable torsion free sheaf. Then
$(v(E)^2)=2 (\rk E) \Delta(E) \geq 0$.
\end{lem}

Let $L$ be an ample divisor on $X$.
$$
P^+(X)_{\Bbb R}:=\{x \in  \NS(X)_{\Bbb R} \mid (x^2)>0, (x,L)>0 \}
$$ 
denotes the positive cone of $X$ and
$C^+(X)$ denotes $P^+(X)_{\Bbb R}/{\Bbb R}_{>0}$. 

For a stability condition $\sigma$,
$Z_\sigma:{\bf D}(X) \to {\Bbb C}$ is the central charge
and ${\cal A}_\sigma$ is the abelian category
generated by $\sigma$-stable objects $E$
with the phase $\phi_\sigma(E) \in (0,1]$.
Then $\sigma$ consists of the pair 
$(Z_\sigma,{\cal A}_\sigma)$
of a central charge 
$Z_\sigma:{\bf D}(X) \to {\Bbb C}$
and an abelian category
${\cal A}_\sigma$.
Let $\Stab(X)$ be the space of stability conditions.
We have a map
$\Pi:\Stab(X) \to H^*(X,{\Bbb C})_{\alg}$
such that 
$(\Pi(\sigma),*)=Z_\sigma(*)$.
If $\sigma$ satisfies the support property,
then $\Pi$ is locally isomorphic.

\begin{NB}
Let $X$ be a smooth projective surface.
Arcara and Bertram \cite{AB} constructed stability condition
$\sigma_{(\beta,\omega)}=({\cal A}_{(\beta,\omega)},Z_{(\beta,\omega)})$
where $Z_{(\beta,\omega)}(E)=(e^{\beta+\sqrt{-1} \omega},\ch(E))$.
More precisely, they did not prove the support property, which is
important for the deformation of 
$\sigma_{(\beta,\omega)}$.
By showing Bogomolov type inequality for $\sigma_{(\beta,\omega)}$-stable
objects, Toda \cite{Toda} proved the support property
of $\sigma_{(\beta,\omega)}$.
Moreover Toda constructed similar stability condition
if $\omega$ is not ample. 
For example if $\pi:X \to Y$ is the blow-up of 
a smooth surface $Y$ at a point and $\omega$ is the pull-back of an ample
class of $Y$, then instead of using a torsion pair of $\Coh(X)$, 
${\cal A}_{(\beta,\omega)}$ is constructed as a tilting of 
a category of perverse coherent sheaves \cite{Br:4}.
On the other hand, in \cite{perv} and \cite{perv2},
Nakajima and the author studied 
\end{NB}

\subsection{Perverse coherent sheaves}\label{subsect:perverse}

Let $\pi:X \to Y$ be a birational morphism of projective surfaces
such that $X$ is smooth, $Y$ is normal and $R^1 \pi_*({\cal O}_X)=0$.
The notion of perverse coherent sheaves was introduced by Bridgeland
\cite{Br:4}. 
Let us briefly recall a slightly different formulation of
perverse coherent sheaves in \cite{PerverseI}.
Let $G$ be a locally free sheaf on $X$ which is a local projective
generator of a category of perverse coherent sheaves
${\frak C}$ (\cite{PerverseI}).
Thus 
\begin{equation}
\begin{split}
T:= & \{ E \in \Coh(X) \mid R^1 \pi_*(G^{\vee} \otimes E)=0 \}\\
S:= & \{ E \in \Coh(X) \mid \pi_*(G^{\vee} \otimes E)=0 \},
\end{split}
\end{equation}
is a torsion pair $(T,S)$ of $\Coh(X)$, 
$G \in T$ 
and
\begin{equation}
{\frak C}=\{E \in {\bf D}(X) \mid H^i(E)=0, 
i \ne -1,0, H^{-1}(E) \in S, H^0(E) \in T \}.
\end{equation}
${\frak C}$ is the heart of a bounded $t$-structure of ${\bf D}(X)$.
For $E \in {\bf D}(X)$,
$\pH^i(E) \in {\frak C}$ denotes the $i$-th cohomology
of $E$ with respect to the
$t$-structure.
We take a divisor $H$ on $X$ which is the pull-back
of an ample divisor on $Y$.
By using a twisted Hilbert polynomial $\chi(G,E(nH))$ of $E \in {\frak C}$,
we define the dimension and the torsion freeness of $E$, which depend only on 
the category ${\frak C}$.
We also define
a $G$-twisted semi-stability of $E \in {\frak C}$ 
as in the Gieseker semi-stability \cite[Prop. 1.4.3]{PerverseI}.
Then we have the following.
\begin{prop}[{\cite{perv2}, \cite[Prop. 1.4.3]{PerverseI}}]
\label{prop:moduli}
There is a coarse moduli scheme $\baM_H^G(v)$ 
of $S$-equivalence classes of $G$-twisted 
semi-stable objects $E$ with $v(E)=v$.
\end{prop}
\begin{defn}
For $\gamma \in \NS(X)_{\Bbb Q}$ such that there is a local
projective generator $G$ of ${\frak C}$
with $\gamma=c_1(G)/\rk G$,
we also define $\gamma$-twisted semi-stability
as $G$-twisted semi-stability. 
\begin{enumerate}
\item[(1)]
We denote the
moduli stack of
$\gamma$-twisted semi-stable objects $E$ with $v(E)=v$
by ${\cal M}_H^\gamma(v)$. 
It is a quotient stack of a scheme by a group action
(see the construction in \cite[Prop. 1.4.3]{PerverseI}).
${\cal M}_H^\gamma(v)^s$ denotes the open substack
of ${\cal M}_H^\gamma(v)$ consisting of $\gamma$-twisted stable objects.
\item[(2)]
$\baM_H^\gamma(v)$ denotes the coarse
moduli scheme of $S$-equivalence classes of
$\gamma$-twisted semi-stable objects $E$ with $v(E)=v$. 
\end{enumerate}
\end{defn}

\begin{rem}
If ${\frak C}=\Coh(X)$, then
every locally free sheaf is a local projective
generator. Hence any $\gamma \in \NS(X)_{\Bbb Q}$ is expressed
as $\frac{c_1(G)}{\rk G}=\gamma$ for a locally free sheaf $G$.
$\gamma$-twisted semi-stability depends on the equivalence
class $\gamma \mod {\Bbb Q}H$ and coincides with the twisted semi-stability
of Matsuki and Wentworth \cite{MW}. 
\end{rem}

\begin{defn}
By using the slope function $(c_1(E),H)/\rk E$,
we also define the $\mu$-semi-stability of a torsion free object
$E$ of ${\frak C}$. 
${\cal M}_H(v)^{\mu\text{-}ss}$ denotes the moduli stack of
$\mu$-semi-stable objects $E$ with $v(E)=v$.
\end{defn}

\begin{rem}\label{rem:indep-of-G}
\begin{enumerate}
\item[(1)]
$\mu$-semi-stability depends only on
the category ${\frak C}$ and $H$.
\item[(2)]
 ${\cal M}_H(v)^{\mu\text{-}ss}$ is bounded.
\end{enumerate}
\end{rem}

\begin{lem}\label{lem:G/beta}
Assume that $G$ satisfies
\begin{equation}\label{eq:ch(G)-1}
\frac{\ch(G)}{\rk G}= e^{\frac{c_1(G)}{\rk G}}
+\left(\frac{1}{8}(K_X^2)-\chi({\cal O}_X)\right)\varrho_X.
\end{equation}
Then
\begin{equation}
\frac{\chi(G,E)}{\rk G}=-(e^\beta,v(E)),
\end{equation}
where 
\begin{equation}
\beta-\frac{1}{2}K_X=\frac{c_1(G)}{\rk G}.
\end{equation}
\end{lem}

\begin{proof}
For $E \in {\bf D}(X)$, we have
\begin{equation}\label{eq:chi}
\begin{split}
\chi(e^{\beta-\frac{1}{2}K_X},E)=&
\int_X \left\{
e^{-\beta}e^{\frac{1}{2}K_X}\ch(E)
\left(1-\frac{1}{2}K_X+\chi({\cal O}_X)\varrho_X \right)
\right\}\\
=& \int_X e^{-\beta}\ch(E)
\left(1+(\chi({\cal O}_X)-\frac{1}{8}(K_X^2))\varrho_X \right)\\
=& \int_X (e^{-\beta}\ch(E))
-\rk E \left(\frac{1}{8}(K_X^2)-\chi({\cal O}_X) \right)\\
=& -( e^\beta,v(E))-
\rk E \left(\frac{1}{8}(K_X^2)-\chi({\cal O}_X) \right).
\end{split}
\end{equation}
Hence the claim holds.
\end{proof}

\subsection{Stability condition associated to ${\frak C}$}
\label{subsect:stab}

For the birational map $\pi:X \to Y$ in \ref{subsect:perverse},  
let $H$ be the pull-back of an ample divisor on $Y$.
For $\beta \in \NS(X)_{\Bbb Q}$ in Lemma \ref{lem:G/beta} and
$\omega \in {\Bbb R}_{>0}H$,
we set
\begin{equation}
Z_{(\beta,\omega)}(E):=(e^{\beta+\omega \sqrt{-1}},v(E)),\; E \in {\bf D}(X).
\end{equation}

For $E \in {\bf D}(X)$, we can write $v(E)$ as 
\begin{equation}
\begin{split}
v(E)= & e^\beta(r(E)+a_\beta(E) \varrho_X+d_\beta(E) H+D_\beta(E)),\; 
D_\beta(E) \in H^\perp\\
= & r(E) e^\beta+a_\beta(E) \varrho_X+d_\beta(E) H+D_\beta(E)+
(d_\beta(E) H+D_\beta(E),\beta) \varrho_X,
\end{split}
\end{equation}
where $r(E)=\rk (E)$ is the rank of $E$ and 
\begin{equation}
d_\beta(E)=\frac{(c_1(E)-r(E) \beta,H)}{(H^2)},\;
a_\beta(E)=-(e^\beta,v(E)).
\end{equation}
Then we have
\begin{equation}
Z_{(\beta,\omega)}(E)=-a_\beta(E)+r(E)\frac{(\omega^2)}{2}
+d_\beta(E)(H,\omega)\sqrt{-1}.
\end{equation}

In appendix, we shall prove the following inequality. 

\begin{lem}[{Lemma \ref{lem:weak-Bogomolov2}}]\label{lem:weak-Bogomolov}
For a $\mu$-semi-stable object $E$ of ${\frak C}$, 
$$
(v(E)^2)-(D_\beta(E)^2)=d_\beta(E)^2 (H^2)-2\rk E a_\beta(E) \geq 0.
$$
\end{lem}

\begin{defn}
\begin{enumerate}
\item[(1)]
${\cal T}_\beta$ is the subcategory of ${\frak C}$
generated by torsion objects and 
$\mu$-stable objects $E$ with $d_\beta(E)>0$.
\item[(2)]
${\cal F}_\beta$ is the subcategory of ${\frak C}$
generated by $\mu$-stable objects $E$ with $d_\beta(E) \leq 0$.
\end{enumerate}
\end{defn}
Then $({\cal T}_\beta,{\cal F}_\beta)$ is a torsion pair
of ${\frak C}$. 
\begin{defn}
Let ${\cal A}_{(\beta,\omega)}$ be the tilting
of the torsion pair $({\cal T}_\beta,{\cal F}_\beta)$
of $\Coh(X)$:
\begin{equation}
{\cal A}_{(\beta,\omega)}=
\{E \in {\bf D}(X) \mid H^i(E)=0, 
i \ne -1,0, H^{-1}(E) \in {\cal F}_\beta, H^0(E) \in 
{\cal T}_\beta \}.
\end{equation}
\end{defn}
By Lemma \ref{lem:weak-Bogomolov},
$\sigma_{(\beta,\omega)}:=({\cal A}_{(\beta,\omega)},Z_{(\beta,\omega)})$
is an example of stability condition.
However we do no know whether 
$\sigma_{(\beta,\omega)}$ satisfies the support property in general. 
If $\pi$ is an isomorphism, a blow-up of a smooth point
or the minimal resolution of rational double points, then
the usual Bogomolov inequality
holds and the support property holds 
(see subsection \ref{subsect:support-property}).

\begin{rem}
By Lemma \ref{lem:G/beta}, a torsion free object 
$E \in {\frak C}$ is $(\beta-\frac{1}{2}K_X)$-twisted semi-stable
with respect to $\omega$
if and only if 
\begin{equation}
\text{(i) } \frac{d_\beta(F)}{\rk F} < \frac{d_\beta(E)}{\rk E} \text{ or }
\text{(ii) }
\frac{d_\beta(F)}{\rk F}=\frac{d_\beta(E)}{\rk E} 
\text{ and }
\frac{a_\beta(F)}{\rk F} \leq \frac{a_\beta(E)}{\rk E}
\end{equation}
for all non-zero subobject $F$ of $E$.
\end{rem}

\begin{defn}\label{defn:Bridgeland-moduli}
For a stability condition $\sigma$,
${\cal M}_\sigma(v)$ denotes the moduli stack of 
$\sigma$-semi-stable objects $E$ with $v(E)=v$. 
We also denote ${\cal M}_{\sigma_{(\beta,\omega)}}(v)$ by
${\cal M}_{(\beta,\omega)}(v)$.
${\cal M}_{(\beta,\omega)}(v)^s$
denotes the substack of ${\cal M}_{(\beta,\omega)}(v)$
consisting of $\sigma_{(\beta,\omega)}$-stable objects.
\end{defn}

\begin{NB}
\begin{rem}
In the proof of Proposition \ref{prop:large} (ii),
we used the $\pi$-nefness of $-K_X$.
Indeed we need the vanishing of
$\Hom(A(-K_X),E)=\Hom(E,A[2])^{\vee}$ for every irreducible object $A$.
If $\pi:X \to Y$ is the blow-up of a smooth point of $Y$,
then $A(-K_X)={\cal O}_C(l)[1],{\cal O}_C(l+1)$.
Hence $A(-K_X) \in {\frak C}[1] \cup {\frak C}$, which implies
that $\Hom(A(-K_X),E)=0$, since $\phi(A(-K_X))=1,2$.

\end{rem}
\end{NB}

\section{Perverse coherent sheaves on a blow-up}\label{sect:blowup}

\subsection{A local projective generator}
Let $\pi:X \to Y$ be the blow-up of a point of $Y$
and $C$ the exceptional divisor on $X$.
For $\beta \in \NS(X)_{\Bbb Q}$,
there is an element $G \in K(X)$ such that 
\begin{equation}\label{eq:ch(G)}
\frac{\ch(G)}{\rk G}= e^{\beta-\frac{1}{2}K_X}
+\left(\frac{1}{8}(K_X^2)-\chi({\cal O}_X)\right)\varrho_X.
\end{equation}
Then
\begin{equation}
\frac{\chi(G,E)}{\rk G}=-(e^\beta,v(E))
\end{equation}
by Lemma \ref{lem:G/beta}.
Assume that $(\beta,C) \not \in \frac{1}{2}+{\Bbb Z}$.
We take an integer $l$ satisfying 
$l-\frac{1}{2}<(\beta,C)<l+\frac{1}{2}$.
Then there is a locally free sheaf $G$ satisfying
\eqref{eq:ch(G)} and
\begin{equation}\label{eq:generator-condition}
\Ext^1(G,{\cal O}_C(l))=\Hom(G,{\cal O}_C(l-1))=0
\end{equation}
by \cite[sect. 2.4]{PerverseI} (see also the next paragraph).
We set
\begin{equation}
\begin{split}
T:= & \{ F \in \Coh(X) \mid R^1 \pi_*(G^{\vee} \otimes F)=0 \}\\
S:= & \{ F \in \Coh(X) \mid \pi_*(G^{\vee} \otimes F)=0 \}.
\end{split}
\end{equation}
Then $(T,S)$ is a torsion pair of $\Coh(X)$.
Let ${\frak C}^\beta$ be a category of
perverse coherent sheaves associated to $(T,S)$.
Then $G$ is a local projective 
generator of ${\frak C}^\beta$.
It is easy to see that
${\frak C}^\beta(lC)$ is the category of perverse coherent sheaves
${^{-1} \Per(X/Y)}$ defined by Bridgeland \cite{Br:4} and studied in
\cite{perv}, \cite{perv2}.
\begin{NB}
Let $E$ be a stable sheaf on $X$ such that
$\Hom(E,E(K_X+C))_0=0$ and $(c_1(E),C)=\rk E(\beta,C)$.
Then $E$ deforms to a locally free sheaf such that
$\Ext^1(E',E')=0$.
For any $\Delta(E) \gg 0$, such a sheaf exists.
Then $G:=E'(kH) \oplus E'(-kH)$ becomes a local projective
generator with a desired invariant $\ch(G)/\rk G$.   
\end{NB}

We set $G_1:=G(lC)$.
By \eqref{eq:generator-condition},
$G^{\vee}_{1|C} \cong {\cal O}_C^{\oplus k} \oplus 
{\cal O}_C(-1)^{\oplus (r-k)}$, where $(c_1(G_1),C)=r-k$.
Then we have an exact sequence
\begin{equation}\label{eq:const-G}
0 \to E \to G_1^{\vee} \to {\cal O}_C(-1)^{\oplus (r-k)} \to 0
\end{equation}
such that $E_{|C} \cong {\cal O}_C^{\oplus r}$.
Since $E$ is the pull-back of a locally free sheaf on $Y$,
$\pi_*(G_1^{\vee}) \cong \pi_*(E)$ is a locally free sheaf on $Y$.
By taking the dual of \eqref{eq:const-G},
we have an exact sequence
\begin{equation}
0 \to G_1 \to \pi^*(G_2) \to {\cal O}_C^{\oplus (r-k)} \to 0,
\end{equation}
where $\pi^*(G_2)=E^{\vee}$.
Thus $G_1$ is an elementary transform of the pull-back of a
locally free sheaf $G_2$ on $Y$.
Conversely starting from a locally free sheaf $G_2$ on $Y$,  
we can construct a local projective generator $G_1$.

\subsection{Some properties of perverse coherent sheaves} 

We recall some results on stable perverse coherent sheaves
in \cite{perv2}.

\begin{lem}\label{lem:reflexive-hull2}
Assume that $l-\frac{1}{2}<(\beta,C)<l+\frac{1}{2}$, that is, 
${\frak C}^\beta(lC)={^{-1}\Per(X/Y)}$.
For a torsion free object $E$ of ${\frak C}^\beta$,
there is an exact sequence 
\begin{equation}
0 \to E \to E' \to T \to 0
\end{equation}
in ${\frak C}^\beta$
such that $T$ is 0-dimensional, $E'(lC)$ is the pull-back
of a locally free sheaf on $Y$.
\end{lem}

\begin{proof}
By Lemma \ref{lem:reflexive-hull} in Appendix,
there is an exact sequence 
\begin{equation}
0 \to E \to E' \to T \to 0
\end{equation}
such that $T$ is 0-dimensional, $E'$ is 
torsion free and
$\Ext^1(A,E')=0$ for all 0-dimensional
object $A$ of ${\frak C}^\beta$. 
In particular, $\Ext^1(E',k_x)=\Ext^1(k_x,E')^{\vee}=0$
for all $x \in X$, which implies $E'$ is locally free by 
\cite[Lem. 1.1.31]{PerverseI}.
Let $A$ be an irreducible object of
${\frak C}^\beta$.
Then $A={\cal O}_C(l)$, $A={\cal O}_C(l-1)[1]$ or
$k_x$ ($x \in X \setminus C$) 
(cf. \cite[Lem. 1.2.16, Prop. 1.2.23]{PerverseI}).
Hence $\Ext^1(E',{\cal O}_C(l-1))=\Ext^1({\cal O}_C(l),E')^{\vee}=0$.
Since $E' \in {\frak C}^\beta$, we also have
$\Hom(E',{\cal O}_C(l-1))=0$.
Hence 
$E'_{|C} \cong {\cal O}_C(l)^{\oplus r}$, which implies that
$E'(lC)$ is the pull-back of a locally free sheaf on $Y$.
\begin{NB}
We have an exact sequence 
\begin{equation}
0 \to {\cal O}_C(l-1)^{\oplus 2} \to 
{\cal O}_C(l) \to {\cal O}_C(l-2)[1] \to 0
\end{equation}
 in ${\frak C}^\beta(K_X)$.
Hence $\Ext^1(E',{\cal O}_C(l))=0$.
Since
$$
\Ext^1(E',{\cal O}_C(l-1)[1])=\Hom({\cal O}_C(l-1)(-K_X),E')^{\vee}
=\Hom({\cal O}_C(l),E')^{\vee}
$$
 and
$E'$ is torsion free,
we also have $\Ext^1(E',{\cal O}_C(l-1)[1])=0$.
Therefore $E'$ is a local projective object of ${\frak C}^\beta$ 
(\cite[Rem. 1.1.35]{PerverseI}).
In particular $E'$  is locally free.
Then $\Ext^1(E',{\cal O}_C(l-1))=0$ implies that
$E'_{|C} \cong {\cal O}_C(l)^{\oplus r}$, which implies that
$E'(lC)$ is the pull-back of a locally free sheaf on $Y$.
\end{NB}
\end{proof}


\begin{defn}
Assume that $l-\frac{1}{2}<(\beta,C)<l+\frac{1}{2}$.
Since ${\frak C}^\beta$ depends only on $l$,
we denote this category by ${\frak C}^l$.
We also denote the moduli stack of $\mu$-semi-stable objects
by  
${\cal M}_H^{\beta-\frac{1}{2}K_X}(v)^{\mu\text{-}ss}$
or ${\cal M}_H^l(v)^{\mu\text{-}ss}$ in order to
indicate the category where the stability is defined
(cf. Remark \ref{rem:indep-of-G}). 
\end{defn}

\begin{rem}
Assume that $\gcd(r,(c_1(v),H))=1$.
Then $E \in {\frak C}$ is $\beta$-twisted semi-stable if and only if
$\pi_*(E)^{\vee \vee}$ is stable.
In particular it is independent of the choice of 
$\beta$ satisfying 
$-\frac{1}{2}+l<(\beta,C)<\frac{1}{2}+l$.
\end{rem}

We shall consider the relation with the $\mu$-semi-stability
of torsion free sheaves.
\begin{defn}
A torsion free sheaf $E$ is $\mu$-semi-stable
with respect to $H$, if
\begin{equation}
\frac{(c_1(E_1),H)}{\rk E_1} \leq
\frac{(c_1(E),H)}{\rk E}
\end{equation}
for any subsheaf $E_1 \ne 0$ of $E$.
\end{defn}
The set of $\mu$-semi-stable sheaves of a fixed Mukai 
vector is bounded (use \cite[Lem. 2.2, sect. 2.3]{chamber}).
If a $\mu$-semi-stable perverse coherent sheaf $E$
is torsion free in $\Coh(X)$, then
it is a $\mu$-semi-stable sheaf. 
By the proof of \cite[Prop. 3.37]{perv2}, we have the following
lemma.
\begin{lem}\label{lem:l=infty}
For $m \gg 0$,
${\cal M}_H^{\beta+mC-\frac{1}{2}K_X}(v)^{\mu\text{-}ss}$ 
consists of $\mu$-semi-stable
torsion free sheaves with respect to
$H$.
\end{lem}

\begin{prop}\label{prop:big-m}
For $m \gg 0$,
${\cal M}_H^{\beta+mC-\frac{1}{2}K_X}(v)=
{\cal M}_{H-q C}^{\beta-\frac{1}{2}K_X}(v)$,
where $q>0$ is sufficiently small.
\begin{NB}
$E$ is a $(\beta+mC-\frac{1}{2}K_X)$-semi-stable object of 
${\frak C}^{\beta+mC}$ if and only if
$E$ is a $(\beta-\frac{1}{2}K_X)$-semi-stable sheaf with respect to
$H-q C$ ($1 \gg q >0$).
\end{NB}
\end{prop}

\begin{proof}
We set $\beta_m:=\beta+mC$.
Let ${\cal M}_H(v)^{\mu\text{-}ss}$
be the moduli stack of $\mu$-semi-stable torsion free 
sheaves with respect to $H$.
By Lemma \ref{lem:l=infty}, 
we may assume that
${\cal M}_H^{\beta_m-\frac{1}{2}K_X}(v) \subset 
{\cal M}_H(v)^{\mu\text{-}ss}$.
\begin{NB}
we may assume that all $(\beta_m-\frac{1}{2}K_X)$-semi-stable
objects $E$ of ${\frak C}^{\beta_m}$ with $v(E)=v$
are torsion free $\mu$-semi-stable sheaves with respect to $H$. 
\end{NB}
We choose a sufficiently small $q>0$ such that
${\cal M}_{H-q C}^{\beta-\frac{1}{2}K_X}(v) \subset
{\cal M}_H(v)^{\mu\text{-}ss}$.
We consider the set $T$ of pairs
$(E_1,E)$ such that  
\begin{enumerate}
\item $E_1$ is a subsheaf of a $\mu$-semi-stable sheaf $E$ of $v(E)=v$
with respect to $H$,  
\item $\frac{(c_1(E_1),H)}{\rk E_1}=\frac{(c_1(E),H)}{\rk E}$ and
\item $E/E_1$ is torsion free.
\end{enumerate}
\begin{NB}
Since $E_1, E/E_1$ are $\mu$-semi-stable sheaves,
Bogomolov inequality hold.
\end{NB}
Since $(v(E_1)^2) \geq 0$ and $(v(E/E_1)^2) \geq 0$,
we see that $\{v(E_1) \mid (E_1,E) \in T \}$
 is a finite set.
Since
\begin{equation}
\begin{split}
& 
\frac{\chi(E(-\beta_m+\frac{1}{2}K_X))}{\rk E}- 
\frac{\chi(E_1(-\beta_m+\frac{1}{2}K_X))}{\rk E_1}\\
=& m\left(\frac{c_1(E)}{\rk E}-\frac{c_1(E_1)}{\rk E_1},-C\right)+
\frac{\chi(E(-\beta+\frac{1}{2}K_X))}{\rk E}- 
\frac{\chi(E_1(-\beta+\frac{1}{2}K_X))}{\rk E_1},
\end{split}
\end{equation}
we can take a sufficiently large $m$ such that
\begin{equation}
\frac{\chi(E_1(p(H-q C)-(\beta-\frac{1}{2}K_X)))}{\rk E_1} 
\underset{(\geq)}{\leq}
\frac{\chi(E(p(H-q C)-(\beta-\frac{1}{2}K_X)))}{\rk E},\;p \gg 0 
\end{equation}  
if and only if
\begin{equation}
\frac{\chi(E_1(-\beta_m+\frac{1}{2}K_X))}{\rk E_1} \underset{(\geq)}{\leq}
\frac{\chi(E(-\beta_m+\frac{1}{2}K_X))}{\rk E} 
\end{equation}  
for $(E_1,E) \in T$.
If $E$ is a $(\beta_m-\frac{1}{2}K_X)$-twisted semi-stable object
of ${\frak C}^{\beta_m}$, then
for $(E_1,E) \in T$,
we take a subsheaf $E_1'$ of $E_1$ such that
$E_1' \in {\frak C}^{\beta_m}$ and
$E_1/E_1' \in {\frak C}^{\beta_m}[-1]$.
Then $E_1'$ is a subobject
of $E$ such that
$\chi(E_1'(-\beta_m+\frac{1}{2}K_X)) \geq 
\chi(E_1(-\beta_m+\frac{1}{2}K_X))$.
Hence
\begin{equation}
\frac{\chi(E_1(-\beta_m+\frac{1}{2}K_X))}{\rk E_1}
\leq \frac{\chi(E_1'(-\beta_m+\frac{1}{2}K_X))}{\rk E_1'}
\leq
\frac{\chi(E(-\beta_m+\frac{1}{2}K_X))}{\rk E},
\end{equation}
which shows
\begin{equation}\label{eq:H-C}
\frac{\chi(E_1(p(H-q C)-(\beta-\frac{1}{2}K_X)))}{\rk E_1} 
\leq
\frac{\chi(E(p(H-q C)-(\beta-\frac{1}{2}K_X)))}{\rk E},\;p \gg 0. 
\end{equation}
Moreover if the equality holds in \eqref{eq:H-C},
then 
we see that $E_1=E_1'$.
Hence $E$ is $(\beta-\frac{1}{2}K_X)$-semi-stable with respect to
$H-q C$.

Conversely for $E \in {\cal M}_{H-q C}^{\beta-\frac{1}{2}K_X}(v)$,
we may assume that 
$\Hom(E(-(m+l)C),{\cal O}_C(-1))=0$.
Hence $E$ is an object of ${\frak C}^{\beta_m}$.
Then we see that $E$ is a $(\beta_m-\frac{1}{2}K_X)$-semi-stable
object. 
\end{proof}

Bogomolov inequality for perverse coherent sheaves
is a consequence of \cite{perv2}. There is another proof in
\cite{Toda}.
\begin{lem}[Bogomolov inequality]\label{lem:perverse-Bogomolov}
For a $\mu$-semi-stable object $E$ of ${\frak C}^\beta$,
$(v(E)^2) \geq 0$.
\end{lem}

\begin{proof}
Assume that $E \in {\cal M}_H^{\beta-\frac{1}{2}K_X}(v)^{\mu\text{-}ss}$.
Take an integer such that $-m-\frac{1}{2}<(\beta,C)<-m+\frac{1}{2}$.
By the proof of \cite[Prop. 3.15]{perv2},
we have an exact sequence  
\begin{equation}
0 \to {\cal O}_C(-m-1)^{\oplus n} \to E \to E' \to 0
\end{equation}
such that $E' \in {\cal M}_H^{\beta-\frac{1}{2}K_X}(v')^{\mu\text{-}ss} \cap 
{\cal M}_H^{\beta-\frac{1}{2}K_X+C}(v')^{\mu\text{-}ss}$,
$v'+n v({\cal O}_C(-m-1))=v$ and
\begin{equation*}
n \leq \dim \Ext^1(E',{\cal O}_C(-m-1)).
\end{equation*}
Since $\chi(E',{\cal O}_C(-m-1))=-\dim \Ext^1(E',{\cal O}_C(-m-1))$,
we have $(c_1(E'),C)+rm \geq n$.
Hence 
\begin{equation}
\begin{split}
(v^2)=& ({v'}^2)+2n((c_1(v'),C)+(\tfrac{1}{2}+m)r)-n^2\\
\geq & ({v'}^2)+n^2+nr>({v'}^2).
\end{split}
\end{equation}
By the induction on $m$, the claim follows from Lemma \ref{lem:l=infty}.
\end{proof}

\begin{NB}
Assume that 
$\phi_{(\beta,\omega)}({\cal O}_C(-m-1))=1$.
Then $\phi_{(\beta,\omega)}({\cal O}_C(-m-2))=0$
and $\phi_{(\beta,\omega)}({\cal O}_C(-m))=1$.
For a $\sigma_{(\beta,\omega)}$-semi-stable object
$E$ with $v(E)=v$ and $\phi_{(\beta,\omega)}(E) \in (0,1)$,
$\Hom({\cal O}_C(-m-1),E)=0$ implies that
$\Ext^2(E,{\cal O}_C(-m-2))=0$.
Since $\Hom(E,{\cal O}_C(-m-2))=0$,
$\chi(E,{\cal O}_C(-m-2))=$.......

\end{NB}

\begin{NB}
We may assume that $l=0$.
Then we have an exact sequence
\begin{equation}
0 \to {\cal O}_C(-1)^{\oplus k} \to \pi^*(\pi_*(E)) \to E \to 0.
\end{equation}
Hence $(v(E)^2)=(v(\pi_*(E))^2)+kr-k^2$.
\end{NB}

\begin{lem}[{\cite[Lem. 3.2]{perv2}}]\label{lem:perv2:non-empty}
\begin{NB}
Assume that 
$0<(\beta-\tfrac{1}{2}K_X,C)<1$.
Then 
$-(c_1(v),C) \leq 0$ if 
${\cal M}_H^{\beta-\frac{1}{2}K_X}(v)^{\mu\text{-}ss} \ne \emptyset$.
\end{NB}
Let $E$ be a torsion free object of ${\frak C}^0$.
Then $(c_1(E),C) \geq 0$.
\end{lem}

\begin{NB}
\begin{proof}
We note that $\Hom({\cal O}_C,E)=0$ and
$\Ext^2({\cal O}_C,E)=\Hom(E,{\cal O}_C(-1))^{\vee}=0$.
Hence $0 \geq \chi({\cal O}_C,E)=-(c_1(E),C)$.
\end{proof}
\end{NB}

\begin{prop}[{cf. \cite[Prop. 3.3]{perv2}}]
Assume that $(c_1(v),C)=0$. Then
${\cal M}_H^{\beta-\frac{1}{2}K_X}(v) \cong 
{\cal M}_{H'}^{\beta'-\frac{1}{2}K_Y}(v')$,
where $\pi^*(v')=v$,
$0<(\beta-\tfrac{1}{2}K_X,C)<1$,
$H' \in \Amp(Y)$ satisfies $\pi^*(H')=H$ and
$\beta'$ is the $\NS(Y)$-component of $\beta$. 
\end{prop}

\begin{proof}
Under the assumption,
$\pi^*(\pi_*(E)) \to E$ is an isomorphism by
the proof of \cite[Prop. 3.3]{perv2}.
Let $G$ be a local projective generator of ${\frak C}^0$
with $\frac{c_1(G)}{\rk G}=\beta-\frac{1}{2}K_X$.
Then there is a locally free sheaf $G_0$ on $Y$ and an exact sequence
$$
0 \to \pi^*(G_0^{\vee}) \to G^{\vee} \to {\cal O}_C(-1)^{\oplus p} \to 0,
$$
where $p=(c_1(G),C)$.
Then $\frac{c_1(G_0)}{\rk G_0}=\beta'-\frac{1}{2}K_Y$
and $G_0^{\vee}={\bf R}\pi_*(\pi^*(G_0^{\vee}))={\bf R}\pi_*(G^{\vee})$.
By the projection formula,
$\pi_*(G^{\vee} \otimes E)=G_0^{\vee} \otimes \pi_*(E)$.
For an exact sequence
\begin{equation}
0 \to E_1 \to E \to E_2 \to 0
\end{equation}
of torsion free objects of ${\frak C}^0$,
we have $(c_1(E_i),C) \geq 0$ for $i=1,2$ by Lemma \ref{lem:perv2:non-empty}.
Since $(c_1(E),C)=0$,
we have $(c_1(E_i),C)=0$ $(i=1,2)$ which implies
$E_i \cong \pi^*(\pi_*(E_i))$. 
Hence the stabilty coincides. 
\end{proof}

\begin{NB}
\begin{prop}[{\cite{perv2},\cite[Prop. 1.4.3]{PerverseI}}]
Assume that $(\beta,C) \not \in \frac{1}{2}+{\Bbb Z}$.
We set $\gamma:=\beta-\frac{1}{2}K_X$.
Then 
there is a coarse moduli scheme $\baM_H^{\gamma}(v)$
of $S$-equivalence classes of 
$\gamma$-twisted semi-stable perverse coherent sheaves
$E \in {\frak C}^\beta$ with $v(E)=v$.  
\end{prop}
\end{NB}

\subsection{Moduli spaces of perverse coherent sheaves
on a categorical wall}\label{subsect:moduli-on-wall}

We shall study the case where $\beta$ satisfies
$(\beta,C)=l-\frac{1}{2}$.
In this case,
$(e^\beta,v({\cal O}_C(l-1)))=0$ and we cannot consider 
a category of perverse coherent sheaves.
We shall use the same idea in \cite[sect. 3.7]{perv2} 
to construct a moduli space
of semi-stable perverse coherent sheaves.

We take a locally free sheaf $G$ such that
$\chi(G,E)/\rk G=-(e^\beta,v(E))$ and
$\Hom(G,{\cal O}_C(l-1)[i])=0$ for all $i \in {\Bbb Z}$.
\begin{NB}
$G^{\vee}(-lC)$ is the pull-back of a locally free sheaf on $Y$.
\end{NB} 

\begin{defn}\label{defn:exceptional-case1}
$E \in {\frak C}^{l-1}$ with $\rk E >0$ is 
$(\beta-\frac{1}{2}K_X)$-twisted semi-stable, if
\begin{equation}\label{eq:beta-ss}
\chi(G,E_1(nH)) \leq \rk E_1 \frac{\chi(G,E(nH))}{\rk E}
\end{equation}
for all subobject $E_1$ of $E$ in ${\frak C}^{l-1}$.
If the inequality is strict for any non-trivial subobject
$E_1$ of $E$, then $E$ is $(\beta-\frac{1}{2}K_X)$-twisted stable.

Let ${\cal M}_H^{\beta-\frac{1}{2}K_X}(v)$ 
be the moduli stack of $\beta$-twisted semi-stable
objects $E$ with $v(E)=v$.
\end{defn}

\begin{rem}
If $E_1 \in {\frak C}^{l-1}$ is a subobject of 
$E \in {\cal M}_H^{\beta-\frac{1}{2}K_X}(v)$ with
$\rk E_1=0$, then
$\chi(G,E_1(nH))=0$ for all $n$ by \eqref{eq:beta-ss}.
Then $E_1 \cong {\cal O}_C(l-1)^{\oplus k}$.
In particular we see $E \in \Coh(X)$ by
$\chi(G,(H^{-1}(E)[1])(nH))=0$. 
\end{rem}

 \begin{rem}\label{rem:beta-stable}
If $E$ is $\beta$-twisted stable, then
$\Hom(E,{\cal O}_C(l-1))=\Hom({\cal O}_C(l-1),E)=0$.
In particular $E \in {\frak C}^{l}$.
Moreover $E$ is a torsion free object of ${\frak C}^{l}$ 
and ${\frak C}^{l-1}$.
We also have an exact sequence
\begin{equation}
0 \to {\cal O}_C(-1)^{\oplus N} \to \pi^*(\pi_*(E(lC)))
\to E(lC) \to 0.
\end{equation}
Moreover $\Hom({\cal O}_C(-1),\pi^*(\pi_*(E(lC)))) \cong 
{\Bbb C}^{\oplus N}$.
Hence $E$ is determined by $\pi_*(E(lC))$.
\end{rem}

We take $\beta_- \in \NS(X)_{\Bbb Q}$ which is sufficiently close to
$\beta$ and $l-\frac{3}{2}<(\beta_-,C)<l-\frac{1}{2}$, and
let $G_-$ be a local projective generator of ${\frak C}^{l-1}$
with $\beta_--\frac{1}{2}K_X=\frac{c_1(G_-)}{\rk G_-}$.
For $E \in {\cal M}_H^{\beta-\frac{1}{2}K_X}(v)$,
we have the Harder-Narasimhan filtration with respect to
$G_-$:
\begin{equation}
0 \subset F_0 \subset F_1 \subset \cdots \subset F_s=E.
\end{equation}
Thus $E_i:=F_i/F_{i-1}$ $(i \geq 0)$ satisfy the following.
\begin{enumerate}
\item
$F_0={\cal O}_C(l-1)^{\oplus k}$ and $E_i$ $(i \geq 1)$
are $(\beta_--\frac{1}{2}K_X)$-semi-stable objects,
\item
$$
\frac{\chi(G_-,E_1(nH))}{\rk E_1}>\frac{\chi(G_-,E_2(nH))}{\rk E_2}
>\cdots >\frac{\chi(G_-,E_s(nH))}{\rk E_s},\; n \gg 0
$$
and
\item
$$
\frac{\chi(G,E_1(nH))}{\rk E_1}=\frac{\chi(G,E_2(nH))}{\rk E_2}
=\cdots =\frac{\chi(G,E_s(nH))}{\rk E_s},\;n \gg 0.
$$
\end{enumerate}

If $E$ is $\beta$-twisted stable, then
$F_0=0$ and $s=1$.
Moreover $E$ is $(\beta_--\frac{1}{2}K_X)$-twisted stable.

We take $\beta_+$ which is sufficiently close to
$\beta$ and $\beta_++\beta_-=2\beta$.
Then 
$l-\frac{1}{2}<(\beta_+,C)<l+\frac{1}{2}$.
Let $G_+$ be a local projective generator of ${\frak C}^l$
with $\beta_+-\frac{1}{2}K_X=\frac{c_1(G_+)}{\rk G_+}$.
Then we have a Harder-Narasimhan type filtration
with respect to $G_+$:
\begin{equation}
0 \subset F_1 \subset F_2 \subset \cdots \subset F_t=E.
\end{equation}
Thus $E_i:=F_i/F_{i-1}$ $(i=1,2,...,t)$ satisfy
\begin{enumerate}
\item
$E_i$ $(i=1,2,...,t-1)$
are $(\beta_+-\frac{1}{2}K_X)$-twisted semi-stable objects of 
${\frak C}^l$ and
$E_t={\cal O}_C(l-1)^{\oplus k}$,
\item
$$
\frac{\chi(G_+,E_1(nH))}{\rk E_1}>\frac{\chi(G_+,E_2(nH))}{\rk E_2}
>\cdots >\frac{\chi(G_+,E_{t-1}(nH))}{\rk E_{t-1}},\; n \gg 0
$$
and
\item
$$
\frac{\chi(G,E_1(nH))}{\rk E_1}=\frac{\chi(G,E_2(nH))}{\rk E_2}
=\cdots =\frac{\chi(G,E_{t-1}(nH))}{\rk E_{t-1}},\; n\gg 0.
$$
\end{enumerate}

\begin{rem}
$$
E_t=\im(E \to {\cal O}_C(l-1) \otimes \Hom(E,{\cal O}_C(l-1))^{\vee}).
$$
\end{rem}

If $E$ is $(\beta-\frac{1}{2}K_X)$-twisted stable, then
$F_t/F_{t-1}=0$ and $t=2$.
Moreover $E$ is $(\beta_+-\frac{1}{2}K_X)$-twisted stable.

\begin{rem}
Since ${\cal O}_C(l-1) \in {\frak C}^{\beta_-}$ is
an irreducible object and
${\cal O}_C(l-1)[1]$ is an irreducible object
of ${\frak C}^{\beta_+}$,  
we also set 
${\cal M}_H^{\beta-\frac{1}{2}K_X}(v)=
{\cal M}_H^{\beta_\pm-\frac{1}{2}K_X}(v)=
\{{\cal O}_C(l-1)^{\oplus n}\}$ for
$v=nv({\cal O}_C(l-1))$ $(n>0)$.
\end{rem}

\begin{lem}\label{lem:stable-on-wall}
Assume that $\rk v>0$. Then
${\cal M}_H^{\beta-\frac{1}{2}K_X}(v)^s=
{\cal M}_H^{\beta_--\frac{1}{2}K_X}(v)^s \cap
 {\cal M}_H^{\beta_+-\frac{1}{2}K_X}(v)^s$.
The same claim also holds for $v=nv({\cal O}_C(l-1))$. 
\end{lem}

\begin{proof}
We already know that
$$
{\cal M}_H^{\beta-\frac{1}{2}K_X}(v)^s \subset 
{\cal M}_H^{\beta_--\frac{1}{2}K_X}(v)^s \cap
 {\cal M}_H^{\beta_+-\frac{1}{2}K_X}(v)^s.
$$
Assume that $E \in {\cal M}_H^{\beta_--\frac{1}{2}K_X}(v)^s \cap
 {\cal M}_H^{\beta_+-\frac{1}{2}K_X}(v)^s$.
Then $E \in {\frak C}^l$ and $\Hom({\cal O}_C(l-1),E)=0$.
If $E$ is not $(\beta-\frac{1}{2}K_X)$-twisted stable, then
there is a subobject $E_1$ of $E$ in ${\frak C}^{l-1}$ with
\begin{equation}\label{eq:E_1}
\chi(G,E_1(nH)) \geq \rk E_1 \frac{\chi(G,E(nH))}{\rk E}.
\end{equation}
We have an exact sequence
$$
0 \to E_1' \to E_1 \to {\cal O}_C(l-1)^{\oplus q} \to 0  
$$
such that $E_1' \in {\frak C}^l$.
Then $E_1'$ is a subobject of $E$ in ${\frak C}^l$ and ${\frak C}^{l-1}$
with $\rk E_1'=\rk E_1$ and
$\chi(G,E_1'(nH))=\chi(G,E_1(nH))$.
By $\Hom({\cal O}_C(l-1),E)=0$,
$E_1' \ne 0$.
Then the $G_\pm$-twisted stability
implies 
\begin{equation}
\chi(G_\pm,E_1'(nH))
< \rk E_1' \frac{\chi(G_\pm,E(nH))}{\rk E},
\end{equation}
which shows 
\begin{equation}
\chi(G,E_1(nH))< \rk E_1 \frac{\chi(G,E(nH))}{\rk E}.
\end{equation}
Therefore $E$ is $(\beta-\frac{1}{2}K_X)$-twisted stable.
\end{proof}

\begin{NB}
If we introduce a suitable stability for 0-dimensional objects,
then we can expect the following remark. 
\begin{rem}\label{rem:stable-on-wall:rk0}
Assume that $v=mv({\cal O}_C(l-1))$.
Then ${\cal M}_H^{\beta-\frac{1}{2}K_X}(v)=\{{\cal O}_C(l-1)^{\oplus m}\}$.
\begin{NB2}
${\cal O}_C(l-1)$ is an irreducible object of ${\frak C}^{l-1}$
and ${\cal O}_C(l-1)[1]$ is an irreducible object of ${\frak C}^{l}$.
\end{NB2}
\end{rem}
\end{NB}

\begin{NB}
$E$ is $(\beta-\frac{1}{2}K_X)$-twisted stable if and only if
$E$ is $(\beta_--\frac{1}{2}K_X)$-twisted stable 
and $(\beta_+-\frac{1}{2}K_X)$-twisted stable.
\end{NB}

We shall explain a construction of
the moduli space.
There is a locally free sheaf $G_0$ on $Y$ such that
$\pi^*(G_0)(-lC) \cong G$.
Then $E \in {\frak C}^{l-1}$ is 
$(\beta-\frac{1}{2}K_X)$-twisted semi-stable if and only if
${\bf R}\pi_*(E(lC))$ is $G_0$-twisted semi-stable.
Moreover if $E$ is $(\beta-\frac{1}{2}K_X)$-twisted stable, then
$\pi_*(E(lC))$ is $(\beta'-\frac{1}{2}K_Y)$-twisted stable,
where $\beta'-\frac{1}{2}K_Y=\frac{c_1(G_0)}{\rk G_0}$.
We have a morphism 
\begin{equation}
\phi:{\cal M}_H^{\beta-\frac{1}{2}K_X}(v) \to
M_{H'}^{\beta'-\frac{1}{2}K_Y}(w),
\end{equation}
where $w=v(\pi_*(E(lC)))$.
\begin{NB}
For $E \in {\cal M}_H^{\beta-\frac{1}{2}K_X}(v)$,
we have an exact sequence
$
0 \to E'(lC) \to E(lC) \to {\cal O}_C(-1)^{\oplus q} \to 0$,
where $E' \in {\frak C}^l$.
Then $\pi_*(E'(lC)) \cong \pi_*( E(lC))$.
For a subsheaf $F_1$ of $\pi_*(E'(lC))$,
we set $E_1:=\im(\pi^*(F_1) \to \pi^*(\pi_*(E'(lC))) \to E'(lC))$
in ${\frak C}^{-1}$.
Then $\pi_*(\pi^*(F_1)) \to \pi_*(E_1) \to \pi_*(E'(lC))$ 
coincides with the inclusion $F_1 \to \pi_*(E'(lC))$. 
Hence $\pi_*(E_1)\cong F_1$.
Then we have an exact sequence
$0 \to {\cal O}_C(-1)^{\oplus p} \to \pi^*(F_1) \to E_1 \to 0$.
Since $\chi(\pi^*(G_0),\pi^*(F_1)(nH))=\chi(\pi^*(G_0),E_1)$,
by the $\beta-\frac{1}{2}K_X$-semi-stability of $E$,
$\pi_*(E(lC))$ is $\beta'-\frac{1}{2}K_Y$-semi-stable   
\end{NB}
If $E$ is $S$-equivalent to
$\oplus_{i=0}^s E_i$ such that
$E_0={\cal O}_C(l-1)^{\oplus p}$ and
$E_i$ $(i>0)$ are $(\beta-\frac{1}{2}K_X)$-twisted stable objects with
$\rk E_i >0$.
Then $\pi_*(E(lC))$ is $S$-equivalent to 
$\oplus_{i=1}^s \pi_*(E_i(lC))$ and
$\pi_*(E_i(lC))$ are $(\beta'-\frac{1}{2}K_Y)$-twisted
stable.
By Remark \ref{rem:beta-stable},
$\oplus_{i=1}^s E_i$ is uniquely determined by $\pi_*(E(lC))$.
Since $pC=c_1(E)-\sum_{i=1}^s c_1(E_i)$,
$E_0$ is also uniquely determined.
 
We set $k:=r((\delta,C)-l)$ and
\begin{equation}
M_{H'}^{\beta'-\frac{1}{2}K_Y}(w,k):=
\{F \in M_{H'}^{\beta'-\frac{1}{2}K_Y}(w) \mid 
\dim \Hom({\cal O}_C(-1),\pi^*(F)) \geq k \}.
\end{equation}
We shall introduce a scheme structure on this Brill-Noether locus
as in \cite[Prop. 3.31]{perv2}.

\begin{lem}\label{lem:onto}
The image of $\phi$ is $M_{H'}^{\beta'-\frac{1}{2}K_Y}(w,k)$.
\end{lem}

\begin{proof}
For $E \in M_H^{\beta-\frac{1}{2}K_X}(v)$,
we have an exact sequence
\begin{equation}
0 \to E'(lC) \to E(lC) \to {\cal O}_C(-1)^{\oplus q} \to 0 
\end{equation}
where $E' \in {\frak C}^l$.
Then $\pi_*(E(lC))=\pi_*(E'(lC))$
and 
\begin{equation}
\begin{split}
\dim \Hom({\cal O}_C(-1),\pi^*(\pi_*(E'(lC))))
\geq& \dim \Ext^1(E'(lC),{\cal O}_C(-1))\\
=&-\chi(E'(lC),{\cal O}_C(-1))=
q+k.
\end{split}
\end{equation}
Hence $\phi(E) \in M_{H'}^{\beta'-\frac{1}{2}K_Y}(w,k)$.
Conversely for $F \in M_{H'}^{\beta'-\frac{1}{2}K_Y}(w,k)$,
we take a $k$-dimensional subspace $V$ of 
$\Hom({\cal O}_C(-1),\pi^*(F))$.
Then $E(lC):=\coker({\cal O}_C(-1) \otimes V \to \pi^*(F)) \in 
{\cal M }_H^{\beta-\frac{1}{2}K_X}(v)$ and
$\phi(E)=F$.
\begin{NB}
We set $E'(lC):=\coker({\cal O}_C(-1)\otimes \Hom({\cal O}_C(-1),\pi^*(F))
\to \pi^*(F))$.
Then $E:=E' \oplus {\cal O}_C(l-1)^{\oplus l}$
also satisfies $f(E)=F$, where
$l=\dim \Hom({\cal O}_C(-1),\pi^*(F))-k$. 
\end{NB}
\begin{NB}
Since $\chi(E(lC)-q{\cal O}_C(-1),E(lC)-q{\cal O}_C(-1))=
\chi(E,E)+q(2(c_1(E(lC)),C)+r)+q^2$,
$q$ is bounded.
\end{NB}
\end{proof}

\begin{defn}\label{defn:moduli-on-wall}
We define the moduli scheme
of the $S$-equivalence classes of
$(\beta-\frac{1}{2}K_X)$-twisted
semi-stable objects 
by $M_H^{\beta-\frac{1}{2}K_X}(v):=M_{H'}^{\beta'-\frac{1}{2}K_Y}(w,k)$.
\end{defn}

Since ${\cal M}_H^{\beta_\pm-\frac{1}{2}K_X}(v) 
\subset {\cal M}_H^{\beta-\frac{1}{2}K_X}(v)$,
we have the following.
\begin{prop}
There are projective morphisms 
$M_H^{\beta_\pm-\frac{1}{2}K_X}(v) \to M_H^{\beta-\frac{1}{2}K_X}(v)$.
\end{prop}

\subsection{Support property}\label{subsect:support-property}

Let us recall Bogomolov type inequality and the support
property in \cite{Toda}, \cite{Toda2}.
In these paper, Toda states results under the assumption $\beta=0$.
Obviously the same results hold for general cases.

\begin{NB}
Let $Q_A(x):=\sum_{i,j}a_{ij}x_i x_j$ be a positive definite quadratic
form associated to a symmetric matrix $A=(a_{ij})$ over ${\Bbb R}$.
Let $\lambda(A)$ be the minimal eigen value of $A$.
Then $Q_A(x) \geq \lambda(A)||x||^2$, where $||x||^2=\sum_i x_i^2$.
Let $B$ be a bounded subset of $\Sym_n({\Bbb R})$.
If $A \in B$, then $\lambda(A)$ is bounded.
For $(\beta,\omega)$, we have a decomposition
$x=e^\beta(r+d \omega+\eta+a \varrho_X)$.
$(r,d,\eta,a)$ are linear forms of $(x_1,...,x_n)$
whose coefficients are 
rational function of $\beta,\omega$.
We set $Q_A(x):=r^2+d^2-(\eta^2)+a^2$. 
Then it is a quadratic form such that each component
of $A$ are rational function of $\beta,\omega$.

\begin{NB2}
$\sqrt{x^2+y^2} \leq |x|+|y| \leq \sqrt{2}\sqrt{x^2+y^2}$.
$\sqrt{\sum_i x_i^2} \leq \sum_i |x_i| \leq \sqrt{n}\sqrt{\sum_i x_i^2}$
by Schwarz inequality.
\end{NB2}

For $x=e^\beta(r+d \omega+\eta+a \varrho_X)$,
\begin{equation}
\frac{||x||_{(\beta,\omega)}}{|Z_{(\beta,\omega)}(x)|}=
\frac{\max \{|r|,|a|,|d(\omega^2)|,\sqrt{-(\eta^2)}\}}
{\sqrt{(-a+r(\omega^2)/2)^2+(d(\omega^2))^2}}
\leq
\max \left\{
\frac{2}{(\omega^2)},\sqrt{2},
\frac{\sqrt{4+(2C_\omega+3)(\omega^2)}}{(\omega^2)},
\frac{1}{a_\beta({\cal O}_C(l))},
\frac{1}{1-a_\beta({\cal O}_C(l))} 
\right\}
\end{equation}

\begin{proof}

\begin{equation}
\begin{split}
(-a+r(\omega^2)/2)^2+(d(\omega^2))^2=&
a^2+(r(\omega^2)/2)^2-ra(\omega^2)+(d(\omega^2))^2 \\
\geq &  
a^2+(r(\omega^2)/2)^2+d^2(\omega^2)^2/2
\end{split}
\end{equation}
 
Hence
\begin{equation}
\begin{split}
\frac{|r|}{|Z_{(\beta,\omega)}(x)|}
\leq & \frac{|r|}{|r|(\omega^2)/2}=\frac{2}{(\omega^2)}\\
\frac{|a|}{|Z_{(\beta,\omega)}(x)|}
\leq & 1\\
\frac{|d(\omega^2)|}{|Z_{(\beta,\omega)}(x)|}
\leq & \sqrt{2}
\end{split}
\end{equation}

Assume that 
\begin{equation}
d^2(\omega^2)+(\eta^2)-2ra+\frac{C_\omega (d(\omega^2))^2}{(\omega^2)}
\geq -1.
\end{equation}
If $\rk x=0$, then
Lemma \ref{lem:Bogomolov(1-dim)} implies
LHS $\geq 0$.
Then
\begin{equation}
\frac{-(\eta^2)}{|Z_{(\beta,\omega)}(x)|^2}
\leq \frac{t(r)+(c_\omega+1)d^2(\omega^2)-2ra}
{(-a+r(\omega^2)/2)^2+d^2(\omega^2)^2}
\end{equation}
where $t(r)=0$ for $r=0$ and $t(r)=1$ for $r \ne 0$.

Assume that $r \ne 0$. Then
\begin{equation}
\begin{split}
\frac{1+(c_\omega+1)d^2(\omega^2)}
{(-a+r(\omega^2)/2)^2+d^2(\omega^2)^2}
\leq & \frac{1+(c_\omega+1)d^2(\omega^2)}
{a^2+(r(\omega^2)/2)^2+d^2(\omega^2)^2/2}\\
\leq & \frac{1}
{a^2+(r(\omega^2)/2)^2+d^2(\omega^2)^2/2}
+\frac{(c_\omega+1)d^2(\omega^2)}
{a^2+(r(\omega^2)/2)^2+d^2(\omega^2)^2/2}\\
\leq & \frac{1}{r^2((\omega^2)/2)^2}+
 \frac{(c_\omega+1)}
{(\omega^2)/2} \leq
 \frac{1}{((\omega^2)/2)^2}+
 \frac{(c_\omega+1)}
{(\omega^2)/2}.
\end{split}
\end{equation}

If $ra<0$, then
\begin{equation}
\begin{split}
\frac{-2ra}
{(-a+r(\omega^2)/2)^2+d^2(\omega^2)^2} \leq &
\frac{-2a/r}
{(-a/r+(\omega^2)/2)^2}\\
\leq & \frac{1}{(\omega^2)}.
\end{split}
\end{equation}

Assume that $r=0$.
If $d \ne 0$, then

Then
\begin{equation}
\begin{split}
\frac{t(r)+(c_\omega+1)d^2(\omega^2)}
{(-a+r(\omega^2)/2)^2+d^2(\omega^2)^2}
\leq & \frac{(c_\omega+1)d^2(\omega^2)}
{a^2+d^2(\omega^2)^2/2}
\leq  \frac{2(c_\omega+1)}{(\omega^2)}.
\end{split}
\end{equation}

If $r=d=0$, then
$x=v({\cal O}_C(l)), v({\cal O}_C(l-1)[1])$.
Then $-(\eta^2)=1$ and
$a_\beta({\cal O}_C(l)), 1-a_\beta({\cal O}_C(l))>0$.
Hence 
\begin{equation}
\frac{\sqrt{-(\eta^2)}}{|Z_{(\beta,\omega)}(x)|}
\leq \max \left\{
\frac{1}{a_\beta({\cal O}_C(l))},
\frac{1}{1-a_\beta({\cal O}_C(l))} 
\right\}.
\end{equation}

\end{proof}
\end{NB}

\begin{lem}[{\cite[Lem. 3.20]{Toda}}]\label{lem:Bogomolov(1-dim)}
There is a constant $C_\omega$ depending only on 
the class $\omega \in C^+(X)$ such that
\begin{equation}
(c_1(E)^2)(\omega^2)+C_\omega (c_1(E),\omega)^2 \geq 0
\end{equation}
for all purely 1-dimensional objects $E \in {\frak C}^\beta$. 
\end{lem}

\begin{proof}
If $\omega$ is ample, then it is nothing but
\cite[Lem. 3.20]{Toda}.
So we assume that $\omega \in \pi^*(\Amp(X))$.
If $l-\frac{1}{2}<(\beta,C)<l+\frac{1}{2}$, then
we have 
${\frak C}^\beta(lC)={^{-1}\Per}(X/Y)$ and
$c_1(E)=c_1(E(lC))$.
Hence the claim also follows from 
\cite[Lem. 3.20]{Toda}.
\end{proof}

Then the Bogomolov type inequality
\cite[Cor. 3.24]{Toda} holds.
\begin{lem}\label{lem:strong-Bogomolov}
For any $\sigma_{(\beta,\omega)}$-stable object
$E \in {\cal A}_{(\beta,\omega)}$,
\begin{equation}
(v(E)^2)+C_\omega \frac{(c_1(E(-\beta)),\omega)^2}{(\omega^2)} 
\geq -1.
\end{equation}
\end{lem}
As in \cite[sect. 3.7]{Toda},
there is a constant $A$
depending on $C_\omega$, $(\omega^2)$,
$\frac{1}{a_\beta({\cal O}_C(l))}$ and
$\frac{1}{1-a_\beta({\cal O}_C(l))}$ 
such that 
\begin{equation}
\frac{\max \{|r|,|a|,|d(\omega^2)|,\sqrt{-(\eta^2)}\}}
{|Z_{(\beta,\omega)}(E)|} \leq A
\end{equation}
for all $\sigma_{(\beta,\omega)}$-stable objects
$E$, where 
$v(E)=e^\beta(r+d \omega+\eta+a \varrho_X)$,
$\eta \in \omega^\perp$ and
$l-\frac{1}{2}<(\beta,C)<l+\frac{1}{2}$.
More precisely,
if $E \ne {\cal O}_C(l),{\cal O}_C(l-1)[1]$, then
we have a bound $A$ depending only on
$C_\omega$ and $(\omega^2)$ as in \cite[sect. 3.7]{Toda}.
If $E={\cal O}_C(l),{\cal O}_C(l-1)[1]$, then
$\frac{1}{a_\beta({\cal O}_C(l))}$ and
$\frac{1}{1-a_\beta({\cal O}_C(l))}$ appear.

In particular, we have the following support property.
\begin{lem}[{\cite[Prop. 3.13]{Toda}}]\label{lem:support}
Let $B$ be a compact subset of  
$\NS(X)_{\Bbb R} \times \Amp(X)_{\Bbb R}$.
Then there is a constant $C_B$ such that 
for $(\beta,\omega) \in B$ satisfying 
$\beta \in \NS(X)_{\Bbb Q}$, $\omega \in {\Bbb R}_{>0}H$,
$H \in \Amp(X)_{\Bbb Q}$ and
$\sigma_{(\beta,\omega)}$-stable object $E$,
\begin{NB}
Wrong statement:
Let $B$ be a compact subset of  
$\NS(X)_{\Bbb R} \times P^+(X)_{\Bbb R}$
such that $(\beta,\omega) \in B$ satisfies
$(\beta,C) \not \in \frac{1}{2}+{\Bbb Z}$
if $\omega \in \pi^*(\Amp(Y)_{\Bbb Q})$.
Then there is a constant $C_B$ such that 
for $(\beta,\omega) \in B$ satisfying 
$\beta \in \NS(X)_{\Bbb Q}$, $\omega \in {\Bbb R}_{>0}H$,
$H \in \Amp(X)_{\Bbb Q}
\cup \pi^*(\Amp(Y)_{\Bbb Q})$ and
$\sigma_{(\beta,\omega)}$-stable object $E$,
\end{NB}
 \begin{equation}\label{eq:C_B}
\frac{||v(E)||}{|Z_{(\beta,\omega)}(E)|} \leq C_B,
\end{equation}
where $||v(E)||$ is a fixed norm on $H^*(X,{\Bbb R})_{\alg}$.
Assume that $B$ is a compact subset of
$\NS(X)_{\Bbb R} \times \pi^*(\Amp(Y)_{\Bbb R})$
such that 
$(\beta,\omega) \in B$ satisfies
$(\beta,C) \not \in \frac{1}{2}+{\Bbb Z}$.
Then \eqref{eq:C_B} also holds.
 \end{lem}

\begin{NB}
\begin{rem}
Assume that $B$ contains
$(\beta,\omega)$ such that
$\omega \in \pi^*(\Amp(Y)_{\Bbb Q})$ and
$(\beta,C) \in \frac{1}{2}+{\Bbb Z}$.
For a stability condition $\sigma_{(\beta,\omega)}$ with
$(\beta,\omega) \in B$ and a 
$\sigma_{(\beta,\omega)}$-stable object $E$ with
$v(E) \ne v({\cal O}_C(a))$, $a \in {\Bbb Z}$,
the same inequality \eqref{eq:C_B} holds.
\end{rem}
\end{NB}

\begin{NB}
If $E$ is $S$-equivalent to
$\oplus_i E_i$, then
$Z_\sigma(E_i)=\lambda_i Z_\sigma(E)$, $\lambda_i>0$.
Since $\sum_i |Z_\sigma(E_i)|=\sum_i \lambda_i |Z_\sigma(E)|
=|\sum_i \lambda_i Z_\sigma(E)|=|\sum_i Z_\sigma(E_i)|=
|Z_\sigma(E)|$,
$||v(E)|| \leq \sum_i ||v(E_i)|| \leq \sum_i C_B|Z_\sigma(E_i)|
=C_B |Z_\sigma(E)|$.
\end{NB}

Let $(\beta,\omega)$ satisfy the conditions in 
Lemma \ref{lem:support}. 
By the support property, we have a wall/chamber structure
in a neighborhood $U$ of $\sigma_{(\beta,\omega)}$.
Thus 
\begin{equation}
\{ \sigma \in U \mid Z_\sigma(E)/Z_\sigma(E_1) \in {\Bbb R}_{>0},
\text{ $E_1 \subset E$ in ${\cal A}_\sigma$,
 $E, E_1$ are $\sigma$-semi-stable, $v(E)=v$ } 
\} 
\end{equation}
is the set of walls for $v$ and a connected component
of the complement is a chamber in $U$.

\begin{NB}
\begin{defn}
$\sigma$ is not general with respect to $v$
if and only if there are $\sigma_\sigma$-semi-stable
objects $E_1$ and $E_2$ such that $v(E_1) \not \in {\Bbb Q}v$,
$v=v(E_1)+v(E_2)$ and $\phi_\sigma(E_1)
=\phi_\sigma(E_2)$.
\end{defn}
\end{NB}

\begin{rem} 
Assume that $\NS(X)={\Bbb Z}H$.
Then the Bogomolov inequality holds for any
semi-stable object of ${\cal A}_{(\beta,\omega)}$.
As in \cite[Prop. 5.7]{MYY:2011:2}, 
if $v_1$ defines a wall, then
we have $(v_1^2) \geq 0$, $(v_2^2) \geq 0$ and
$(v_1,v_2)>0$, where $v_2=v-v_1$. 
Since we don't know the sufficient condition
for the existence of stable objects $E_i$ with $v(E_i)=v_i$, 
the condition is only a 
necessary condition.
\end{rem}

We set 
\begin{equation}
\begin{split}
{\cal K}(X):= &\{(\beta,\omega) \mid \beta \in \NS(X)_{\Bbb R},
\omega \in \Amp(X)_{\Bbb R} \}\\
\partial {\cal K}(X):= & \{(\beta,\omega) \mid 
\beta \in \NS(X)_{\Bbb R},
\omega \in \pi^*(\Amp(Y)_{\Bbb R}),
(\beta,C) \not \in \tfrac{1}{2}+{\Bbb Z} \}\\
\overline{\cal K}(X):=& {\cal K}(X) \cup \partial {\cal K}(X).
\end{split}
\end{equation}

We have a natural embedding 
$\overline{\cal K}(X) \to \NS(X)_{\Bbb R} \times P^+(X)_{\Bbb R}$.

\begin{prop}[{\cite{Toda2}}]
We have a map ${\frak s}:\overline{\cal K}(X) \to \Stab(X)$
such that $Z_{{\frak s}(\beta,\omega)}=e^{\beta+\sqrt{-1}\omega}$
and ${\frak s}(\beta,\omega)=\sigma_{(\beta,\omega)}$
if $\beta \in \NS(X)_{\Bbb Q}$ and
$\omega \in {\Bbb R}_{>0}H$, $H \in \NS(X)$. 
\end{prop}
For the proof of this result, we also need the following.
Then the same proof of \cite{Toda} works.

\begin{lem}[{\cite[Prop. 3.14]{Toda}}]
Assume that $(\beta_0,\omega_0) \in 
\partial {\cal K}(X)$.  
Let $U$ be an open neighborhood of $\sigma_0:=\sigma_{(\beta_0,\omega_0)}$
in $\Stab(X)$. 
We set
$V:=\{\sigma \mid \Pi(\sigma)=e^{\beta+\sqrt{-1}\omega}, (\omega,C) \geq 0\}$.
Then there is an open neighborhood $U'$ of $\sigma_0$ such that
${\cal M}_{\sigma}(\varrho_X)=\{ k_x \mid x \in X \}$
for all $\sigma \in U' \cap V$.
\end{lem}

\begin{proof}
For $E \in {\cal M}_{\sigma_0}(\varrho_X)$,
$\phi_{\sigma_0}(E)=1$. 
We first classify 
$\sigma_0$-stable objects $E$ with 
$\phi_{\sigma_0}(E)=1$.
Since $\phi_{\sigma_0}(\pH^{-1}(E)[1]),
\phi_{\sigma_0}(\pH^0(E)) \leq 1$,
$\phi_{\sigma_0}(\pH^{-1}(E)[1])=\phi_{\sigma_0}(\pH^0(E))=1$.
Hence $\pH^0(E)=0$ or $\pH^{-1}(E)=0$.
In the first case,
$0=-\rk E=\rk \pH^{-1}(E)>0$. Hence this case does not occur.
For the second case,
$E=\pH^0(E)$ satisfies $\rk \pH^0(E)=(c_1(E),\omega)=0$.
Hence $E$ is a 0-dimensional object of ${\frak C}^\beta$.
If $\Supp(E) \subset X \setminus C$, then
$E \cong k_x$, $x \in X \setminus C$.
Assume that $\Supp(E) \subset C$.
We may assume ${\frak C}^\beta(lC)={^{-1} \Per}(X/Y)$.
By the classification of 0-dimensional objects,
$E$ is generated by ${\cal O}_C(l)$ and ${\cal O}_C(l-1)[1]$.
Thus $E={\cal O}_C(l)$ or $E={\cal O}_C(l-1)[1]$.
We note that $\varrho_X=n_1 v({\cal O}_C(l))+n_2 v({\cal O}_C(l-1)[1])$
if and only if $n_1=n_2=1$.
It is easy to see that
${\cal M}_{\sigma_0}(\varrho_X)$ consists of 
$E$ such that
\begin{enumerate}
\item
$E$ is a non-trivial extensions
$$
0 \to {\cal O}_C(l) \to E \to {\cal O}_C(l-1)[1] \to 0
$$
or 
\item
$$
0 \to {\cal O}_C(l-1)[1] \to E \to {\cal O}_C(l) \to 0
$$
or 
\item
$E={\cal O}_C(l) \oplus {\cal O}_C(l-1)[1]$
or
\item
$E=k_x$ $(x \in X \setminus C)$.
\end{enumerate}
We may assume that ${\cal O}_C(l)$ and ${\cal O}_C(l-1)[1]$
are $\sigma$-stable for all $\sigma \in U'$.
If $\sigma \in V$, then
$\phi_\sigma({\cal O}_C(l))<1<\phi_\sigma({\cal O}_C(l-1)[1])$.
Hence if $E$ is $\sigma$-semi-stable for $\sigma \in U' \cap V$, then
$E$ is an extension of the first type.
Therefore $E \cong k_x$, $x \in C$.
\end{proof}

\section{Structures of walls and chambers}
\label{sect:chamber}

Let $X$ be a smooth projective surface.
In this section, we shall study the structure of walls and
chambers. 
Let $P^+(v)_{\Bbb R} \subset v^\perp$ be the positive cone in $v^\perp$ and 
set $C^+(v):=P^+(v)_{\Bbb R}/{\Bbb R}_{>0}$.
We shall construct a map
$$
\xi:\NS(X)_{\Bbb R} \times P^+(X)_{\Bbb R} \to C^+(v)
$$ 
and study its property (Proposition \ref{prop:fibration}, 
Corollary \ref{cor:fibration}).

Let $v=r+c_1(v)+a\varrho_X$ 
be an element of $v(K(X)) \subset H^*(X,{\Bbb Q})_{\alg}$
with $r>0$.
From now on, we set
\begin{equation}\label{eq:delta}
\delta:=\frac{c_1(v)}{r}.
\end{equation}
Then we have
$$
v=r e^\delta-\frac{(v^2)}{2r} \varrho_X.
$$
As in \cite{Movable}, 
we set
\begin{equation}\label{eq:def-xi}
\begin{split}
\xi(\beta,\omega)
:=& \xi(\beta,\omega,1)/r\\
=& 
\left(
\frac{(\omega^2)-((\beta-\delta)^2)}{2}+\frac{(v^2)}{2r^2} \right)
(\omega+(\omega,\delta)\varrho_X)\\
&+(\beta-\delta,\omega)(\beta-\delta+(\beta-\delta,\delta)\varrho_X)
+(\beta-\delta,\omega)\left(e^\delta+\frac{(v^2)}{2r^2}\varrho_X \right)
\in C^+(v)
\end{split}
\end{equation}
for $(\beta,\omega) \in \NS(X)_{\Bbb R} \times P^+(X)_{\Bbb R}$. 
We have
$$
\xi(\beta,\omega) =
\mathrm{Im}\frac{e^{\beta+\sqrt{-1}\omega}}
{Z_{(\beta,\omega)}(v)} \in C^+(v).
$$

For $v_1 \in H^*(X,{\Bbb Q})_{\alg}$,
$Z_{(\beta,\omega)}(v_1) \in {\Bbb R}Z_{(\beta,\omega)}(v)$
if and only if $\xi(\beta,\omega) \in v_1^\perp$.

\begin{rem}\label{rem:fiber-property}
For $u \in v^\perp$ with $(u^2)>0$,
$(\beta,\omega) \in \xi^{-1}(u)$ if and only if
$Z_{(\beta,\omega)}(x) \in {\Bbb R}Z_{(\beta,\omega)}(v)$
for all $x \in u^\perp$. 
\end{rem}

The main result of this subsection is the following.
\begin{prop}\label{prop:fibration}
$\xi:\NS(X)_{\Bbb R} \times P^+(X)_{\Bbb R} \to C^+(v)$
is a regular map whose fibers are connected.
In particular, $\xi^{-1}(U)$ is connected if $U$ is connected. 
Moreover the restriction of a fiber to ${\cal K}(X)$
is also connected.
\end{prop}

Before proving this proposition, 
we shall give consequences of the proposition.
We consider the chamber structure 
in $\NS(X)_{\Bbb R} \times \Amp(X)_{\Bbb R}$
and the map
\begin{equation}\label{eq:map1}
\xi:\NS(X)_{\Bbb R} \times \Amp(X)_{\Bbb R} \to C^+(v).
\end{equation}
 By the identification 
$\overline{\cal K}(X) \cong {\frak s}(\overline{\cal K}(X))$,
we have a map
$$
\tilde{\xi}:{\frak s}(\overline{\cal K}(X)) \to C^+(v).
$$
For $E \in {\bf D}(X)$
with $v(E)=v$,
assume that $E$ is $\sigma$-semi-stable and
is $S$-equivalent to
$\oplus_i E_i$, where $\sigma \in D:=\tilde{\xi}^{-1}(u)$ and
$E_i$ are $\sigma$-stable
with the same phase.
Then there is an open neighborhood $U$ of $\sigma$ in $D$
such that $E_i$ are $\sigma'$-stable for $\sigma' \in U$.
Since $\phi_{\sigma'}(E_i),\phi_{\sigma'}(E_j)$ are continuous over $U$
and $\phi_{\sigma'}(E_i)-\phi_{\sigma'}(E_j) \mod {\Bbb Z}$ is constant
over $D$,
\begin{NB}
$Z_{\sigma'}(E_i) \in {\Bbb R}Z_{\sigma'}(E)$
for $\sigma' \in D$.
\end{NB}
$\phi_{\sigma'}(E_i)=\phi_{\sigma'}(E_j)$ for all $i,j$.
Hence $E$ is a $\sigma'$-semi-stable object
which is $S$-equivalent to $\oplus_i E_i$.
Therefore semi-stability is an open condition.
If $E$ is not $\sigma'$-semi-stable for $\sigma' \in D$,
then in a neighborhood of $\sigma'$,
$E$ is not semi-stable.
Assume that $D$ is connected.
Then $E$ is $\sigma'$-semi-stable
for all $\sigma' \in D$.        
Assume that $E$ is $\sigma$-semi-stable and
is $S$-equivalent to
$\oplus_i E_i$, where $E_i$ are $\sigma$-stable
with the same phase.
Then we also see that $E_i$ are $\sigma'$-stable for all
$\sigma' \in D$. 
Thus the $S$-equivalence class of $E$ is independent of $\sigma \in D$.
In particular, we have the following.
\begin{lem}\label{lem:xi-wall}
Let $\sigma_0 \in {\frak s}(\overline{\cal K}(X))$ and assume that
$\tilde{\xi}^{-1}(\tilde{\xi}(\sigma_0))$ is connected.
If $\sigma_0$ is on a wall, then
every $\sigma \in \tilde{\xi}^{-1}(\tilde{\xi}(\sigma_0))$
is on the same wall.
In particular, if $\sigma_0 \in {\frak s}({\cal K}(X))$,
then the claim holds.
\begin{NB}
Assume that $\sigma_0 \in {\frak s}({\cal K}(X))$.
If $\rk \tilde{\xi}(\sigma_0) \ne 0$,
then by the proof of Proposition \ref{prop:fibration},
the fiber is connected.
If $\rk \tilde{\xi}(\sigma_0)=0$, then
the ampleness of $c_1(\tilde{\xi}(\sigma_0))$ implies
the fiber is contained in ${\frak s}({\cal K}(X))$.
Hence it is also connected by its description.   
\end{NB}
\end{lem}

\begin{NB}
Let $B$ be a compact subset of $\Stab(X)^*$.
Then 
$m_\tau(E)/m_\sigma(E)$ is uniformly bounded for all nonzero 
$E \in {\bf D}(X)$ by
$d(\sigma,\tau)>\left|\log \frac{m_\tau(E)}{m_\sigma(E)} \right|$
and 
$$
N(B):=\max_{\sigma,\tau \in B} d(\sigma,\tau)<\infty,
$$ 
where
$m_\tau(E)=\sum_i |Z_\tau(A_i)|$.
Hence 
$$
e^{-N(B)}m_\sigma(E)<m_\tau(E)<e^{N(B)}m_\sigma(E)
$$
for all $0 \ne E \in {\bf D}(X)$.
For $E \in {\bf D}(X)$, let $T(E)$ 
be the set of semi-stable factor $A$ of an object $E$ in some stability
condition $\tau \in B$.
By the definition of $m_\tau$, 
$m_\tau(E) \geq m_\tau(A)$. 
Since $e^{N(B)} m_\tau(A)>m_\sigma(A)$,
we have 
$$
e^{2N(B)}m_\sigma(E)>e^{N(B)}m_\tau(E)>e^{N(B)}m_\tau(A)>m_\sigma(A).
$$
Thus $\{m_\sigma(A) \mid A \in T(E) \}$ si bounded.
Now assume that the support property holds at $\sigma$, that is
$C|Z_\sigma(F)|>||F||$ for all $\sigma$-stable objects.
Then for $\sigma$-stable factor $F$ of $A$ satisfies
$Ce^{2N(B)}m_\sigma(E)>Cm_\sigma(A) \geq C|Z_\sigma(F)|>||F||$.
Therefore the choice of $v(F)$ is finite.
Then the choice of $v(A)$ is also finite, by
the bound of $m_\sigma(A)$.
In particular, for the set of $\sigma$-semi-stable 
objects $E$ with a fixed Mukai vector $v$,
the set of $v(A)$ is finite.

\begin{lem}
$\sigma$-stability of $E$ is an open condition.
\end{lem}

\begin{proof}
Let $\sigma=({\cal P},Z) \in B$ be a stability condition
such that $E \in {\cal P}(\phi)$ with $\phi \in {\Bbb R}$.
For $0<\eta<1/8$,
we set
\begin{equation}
U:=\{\tau \in \Stab(X) \mid d(\sigma,\tau)<\eta \}.
\end{equation}
If $A$ is a semi-stable factor of $E$ in a stability condition 
$\tau \in U$,
then $A$ lies in the abelian subcategory
\begin{equation}
{\cal A}:={\cal P}\left(\left(\phi-\frac{1}{2},\phi+\frac{1}{2}\right]
\right) \subset {\bf D}(X). 
\end{equation}
Hence the Harder-Narasimhan filtration of $E$ with respect to
$\tau$ gives a filtration  
in ${\cal A}$.
In particular there is a subobject $A \in {\cal A}$ of $E$
such that $A$ is $\tau$-stable with $\phi_\tau(A)=\phi_\tau^+(E)$. 
Let $C_t$ be a curve with $\xi(C_t)$ is a segment in $C^+(v)$.
We set $\sigma=C_0$ and $\tau=C_1$.
Assume that $E$ is $\sigma$-stable but
is not $\tau$-stable.
Then $\phi_1(A) \geq \phi_1(E)$. 
Since $E$ is a $\sigma$-stable object of ${\cal A}$,
$\phi_0(A)<\phi_0(E)$.
\begin{NB2}
We first note that 
since $A \in {\cal A}$,
$\phi_\sigma(A)$ is determined by $Z_\sigma(A)$.

Since $E$ is $\sigma$-stable,
$\phi^+_\sigma(A) \leq  \phi_\sigma(E)=\phi$.
Since $A \in {\cal A}$,
$\phi_\sigma^-(A)>\phi-\frac{1}{2}$.
Hence $\phi-\frac{1}{2}<\phi_\sigma(A) \leq \phi_\sigma(E)$.
If the equality holds, then
$A$ is $\sigma$-semi-stable with $\phi_\sigma(A)=\phi_\sigma(E)$.
Then $A \to E$ is surjective in ${\cal P}(\phi)$ and
$(E/A)[-1] \in {\cal P}(\phi)$.
Therefore $(E/A) \in {\cal P}(\phi+1)$, which implies
$E/A=0$.
Since $\phi_\tau(A)>\phi_\tau(E)$, this is impossible. 
Therefore $\phi_\sigma(A)< \phi_\sigma(E)$.
\end{NB2} 
Hence
$\phi_s(A)=\phi_s(E)$ at $s \in (0,1]$.
Then $\xi(C_t)$ intersect the hyperplane
$v(A)^\perp$ transversely at $\xi(C_s)$.
In particular $\sigma$ is not on the wall defined by
$A$.
Since the set of $v(A)$ is finite,
in a small neighborhood of $\sigma$,
this is impossible.
Therefore $E$ is stable in a neighborhood of $\sigma$. 
\end{proof}

\begin{lem}
If $E$ is not $\sigma$-semi-stable, 
then it is not semi-stable
in a neighborhood of $\sigma$.
\end{lem}

\begin{proof}Let $A$ be a $\sigma$-stable subobject of $E$ 
with $\phi_\sigma(A)=\phi_\sigma^+(E)$.
Then there is an open neighborhood $U$ of $\sigma$
such that $A$ is $\tau$-stable and
$\phi_\tau(A)>\phi_\tau(E)$ for all $\tau \in U$.
Then $\Hom(A,E) \ne 0$ implies $E$ is not $\tau$-stable.
\end{proof}

\end{NB}

\begin{lem}\label{lem:locally-finite}
Let $u \in P^+(v)_{\Bbb R}$ satisfy
$\rk u \ne 0$ or $u=e^{\delta}\lambda$ 
$(\lambda \in \Amp(X)_{\Bbb R})$.
Then 
the set of walls in $P^+(v)_{\Bbb R}$ is finite in a neighborhood of
$u$. 
\end{lem}

\begin{proof}
Let $V$ be a compact small neighborhood of $x$ of $C^+(v)$.
Then there is a compact subset $\widetilde{V}$ of 
$\NS(X)_{\Bbb R} \times \Amp(X)_{\Bbb R}$ such that
$\xi(\widetilde{V})=V$ by Proposition \ref{prop:fibration}.
Then there are finitely many walls intersecting $\widetilde{V}$.
By Lemma \ref{lem:xi-wall},
the set of walls in $P^+(v)_{\Bbb R}$ is locally finite. 
\end{proof}

\begin{cor}\label{cor:fibration}
Let $W$ be the set of vectors defining walls with respect to
$v$.
Let $U$ be a connected component of 
$C^+(v) \setminus \cup_{u \in W} u^\perp$.
Then $\xi^{-1}(U) \cap \NS(X)_{\Bbb R} \times \Amp(X)_{\Bbb R}$ 
is a chamber for $v$.
\end{cor}

\begin{proof}
A chamber is a connected component of
$\NS(X)_{\Bbb R} \times \Amp(X)_{\Bbb R} \setminus \cup_{u \in W} 
W_u$ and $(\beta,\omega) \in W_u$ if and only if 
$\xi(\beta,\omega) \in u^\perp$.
By Proposition \ref{prop:fibration},
$\xi^{-1}(U)$ is connected.
Hence it is a chamber. 
\end{proof}

Corollary \ref{cor:fibration}
implies that we can study the wall crossing behavior 
by looking at linear walls in $C^+(v)$.

\begin{NB}
A similar result does not hold for $\omega \in \pi^*(\Amp(Y)_{\Bbb Q})$
by Lemma \ref{lem:fiber2} (1).
\end{NB}

{\it Proof of Proposition \ref{prop:fibration}.}

We set 
\begin{equation}\label{eq:u}
u:=\zeta+(\zeta,\delta)\varrho_X+
y \left(e^\delta+\frac{(v^2)}{2r^2}\varrho_X \right).
\end{equation}
Then $u \in C^+(v)$ if and only if
$(u^2)>0$ and $(u,H+(H,\delta)\varrho_X)=(\zeta,H)>0$. 
Assume that $y \ne 0$.
For $H \in P^+(X)_{\Bbb R}$, we have a decomposition
\begin{equation}\label{eq:zeta}
\zeta=\frac{(\zeta,H)}{(H^2)}H+D,\; 
D=\zeta-\frac{(\zeta,H)}{(H^2)}H \in H^\perp.
\end{equation} 
Then 
\begin{equation}\label{eq:xi-explicit}
\frac{\xi(\delta+sH+D',tH)}{(sH+D',tH)}=
\frac{\frac{(H^2)}{2}(t^2+s^2)+\frac{(v^2)}{2r^2}-\frac{({D'}^2)}{2}}{s(H^2)}
(H+(H,\delta)\varrho_X)
+D'+(D',\delta)\varrho_X+\left(e^\delta+\frac{(v^2)}{2r^2}\varrho_X \right),
\end{equation}
where $D' \in H^\perp$.
Hence 
$\frac{\xi(\delta+sH+D',tH)}{(sH+D',tH)}=\frac{u}{y}$
if and only if
\begin{equation}
D'=\frac{D}{y},\; 
\frac{\frac{(H^2)}{2}(t^2+s^2)+\frac{(v^2)}{2r^2}-\frac{(D^2)}{2y^2}}{s(H^2)}
=\frac{(\zeta,H)}{y(H^2)}.
\end{equation}
Since 
\begin{equation}
\begin{split}
& \frac{(H^2)}{2}(t^2+s^2)+\frac{(v^2)}{2r^2}-\frac{(D^2)}{2y^2}
-s \frac{(\zeta,H)}{y}\\
=& \frac{(H^2)}{2}(t^2+s^2)+\frac{(v^2)}{2r^2}-
\frac{1}{2y^2}\left((\zeta^2)-\frac{(\zeta,H)^2}{(H^2)} \right)
-s \frac{(\zeta,H)}{y}\\
=& \frac{(H^2)}{2}t^2+
\frac{(H^2)}{2}\left(s-\frac{(\zeta,H)}{y(H^2)} \right)^2
+\frac{(v^2)}{2r^2}-
\frac{(\zeta^2)}{2y^2}\\
=& \frac{(H^2)}{2}t^2+
\frac{(H^2)}{2}\left(s-\frac{(\zeta,H)}{y(H^2)} \right)^2
-\frac{(u^2)}{2y^2},
\end{split}
\end{equation}
$\xi^{-1}(u)$ is parameterized by 
$H$ and the circle 
\begin{equation}\label{eq:circle}
\frac{(H^2)}{2}t^2+
\frac{(H^2)}{2}\left(s-\frac{(\zeta,H)}{y(H^2)} \right)^2=
\frac{(u^2)}{2y^2}.
\end{equation}

\begin{NB}
For $(H,s_0,t_0)$, let $e^\delta(\zeta_0+y_0(1+\frac{(v^2)}{2r^2}\varrho_X))$
be the image of $\xi$.
Then for a small neighborhood $U$ and a point 
$e^\delta(\zeta+y(1+\frac{(v^2)}{2r^2}\varrho_X))$ of $U$,
there is $(H,s_0,t)$ with 
$\xi(\delta+s_0 H+D',tH)= e^\delta(\zeta+y(1+\frac{(v^2)}{2r^2}\varrho_X))$.
\end{NB}

\begin{NB}
$(\zeta,H)^2 \geq (H^2)(\zeta^2)$ means that
$s \geq 0$ for $y>0$ and $s \leq 0$ for $y<0$. 

The number of parameters is
$(\rho(X)-1)+1=\rho(X)$.
\end{NB}

\begin{NB}
We set $\omega=tH$. Then
$\beta=\delta+\zeta+\left(x-\frac{(\zeta,\omega)}{(\omega^2)} \right)\omega$
with 
$$
(\omega^2)\left(1+\left(x-\frac{(\zeta,\omega)}{(\omega^2)} \right)^2 \right)
=\frac{(u^2)}{2y}.
$$
\end{NB}
For the general case, we apply Proposition \ref{prop:Phi-xi}
to isometries
$(r,\xi,a) \mapsto e^\eta(r,\xi,a)$ and 
$(r,\xi,a) \mapsto (a,-\xi,r)$ in order to reduce to the case
$y \ne 0$.
Thus $\xi$ is a regular map as a $C^\infty$-map.  
Since we are restricted to $\Amp(X)_{\Bbb R}$,
for the connectedness of the fiber of $\xi$,
we need to describe it directly.
For $u=H+(H,\delta)\varrho_X$,
$$
\xi^{-1}(u)=\{(\beta,\omega) \mid \beta=\delta+D, \omega \in {\Bbb R}_{>0}H,
D \in H^\perp \}.
$$
So it is connected.
\qed

\begin{rem}\label{rem:fiber-property2}
In order to clarify the dependence of $\xi(\beta,\omega)$ on 
$v$, we set 
$$
\xi_v(\beta,\omega):=\mathrm{Im}\frac{e^{\beta+\sqrt{-1}\omega}}
{Z_{(\beta,\omega)}(v)}.
$$
Then
$$
\xi_v^{-1}(u)=\xi_w^{-1}(u)
$$
for $u \in v^\perp \cap w^\perp$.

Indeed for $u \in v^\perp$ with $\rk u \ne 0$,
by using the expression of $u$ in \eqref{eq:u},
$c_1(u)=\zeta+y\delta$ and
\eqref{eq:circle}, we see that
$\xi_v(\beta,\omega)=u$ in $C^+(v)$ if and only if
$$
\frac{(\omega^2)}{2}+\left(\frac{(y\beta-c_1(u),\omega)}{y(\omega^2)}
\right)^2 \frac{(\omega^2)}{2}=\frac{(u^2)}{2y^2}.
$$
If $\rk u=0$, then 
$(c_1(v)-r\beta,\omega)=0$ and $\omega \in {\Bbb R}_{>0}c_1(u)$.
Hence it does not depend on the choice of $v$.
The same claim also holds if $r=0$.
\begin{NB}
For $v=(r,\xi,a)$,
\begin{equation*}
\begin{split}
\xi(\beta,\omega):=&
\left(r\frac{(\omega^2)}{2}+\langle e^\beta,v \rangle \right)
(\omega+(\beta,\omega)\varrho_X)
-(\xi-r \beta,\omega)\left(e^\beta-\frac{(\omega^2)}{2}\varrho_X \right).
\end{split}
\end{equation*} 
If $u=\xi(\beta,\omega)$ with $\rk u \ne 0$,
then 
\begin{equation}
\begin{split}
\frac{c_1(u)}{\rk u}=&\beta-A\omega,\\
\frac{(c_1(u)^2)-(u^2)}{2 \rk u^2}=& \frac{(\beta^2)-(\omega^2)}{2}
-A(\beta,\omega),\\
A=& \frac{r\frac{(\omega^2)}{2}+\langle e^\beta,v \rangle }
{(\xi-r \beta,\omega)}.
\end{split}
\end{equation}
By the first equation,
$A(\omega^2)=(\beta-c_1(u),\omega)$.
Hence  
$$
\frac{(\omega^2)}{2}+
\left(\frac{(\beta \rk u -c_1(u),\omega)}{\rk u(\omega^2)}
\right)^2 \frac{(\omega^2)}{2}=\frac{(u^2)}{2 (\rk u)^2}.
$$
\end{NB}

\end{rem}

\begin{NB}

Direct proof of the regularity at
$(\beta_0,\omega_0)$ with
$\xi(\beta_0,\omega_0) \in \varrho_X^\perp$.
We set $(\beta,\omega)=(\beta_0+\eta,\omega_0+\lambda)$,
$\lambda \in \omega_0^\perp$ and $\eta=x \omega_0$.
We also set 
$\beta_0-\delta=pe$, $(e^2)=-1$,
$e \in \omega_0^\perp$.
Then $\beta-\delta=pe+\eta=pe+x \omega_0$.
\begin{equation}
\begin{split}
\xi(\beta,\omega)-\xi(\beta_0,\omega_0)
\sim & 
\left(\frac{(\omega_0^2)-((\beta_0-\delta)^2)}{2}+
\frac{(v^2)}{2r^2} \right)\lambda
+(pe+\eta,\omega_0+\lambda)(pe+\eta)+
(pe+\eta,\omega_0+\lambda)\left(1+\frac{(v^2)}{2r^2}\varrho_X \right)\\
\sim & 
\left(\frac{(\omega_0^2)-((\beta_0-\delta)^2)}{2}+
\frac{(v^2)}{2r^2} \right)\lambda
+
((\eta,\omega_0)+p(\lambda,e))pe+
((\eta,\omega_0)+p(\lambda,e))\left(1+\frac{(v^2)}{2r^2}\varrho_X \right)\\
=& \left(\frac{(\omega_0^2)-((\beta_0-\delta)^2)}{2}+
\frac{(v^2)}{2r^2} \right)(\lambda,e)e
+
((\eta,\omega_0)+p(\lambda,e))pe+
((\eta,\omega_0)+p(\lambda,e))\left(1+\frac{(v^2)}{2r^2}\varrho_X \right)\\
\mod e^\perp \cap \left(1+\frac{(v^2)}{2r^2}\varrho_X \right)^\perp.
\end{split}
\end{equation}
Since 
\begin{equation}
\begin{pmatrix}
\left(\frac{(\omega_0^2)-((\beta_0-\delta)^2)}{2}+
\frac{(v^2)}{2r^2} \right) & p\\
0 & 1
\end{pmatrix}
\begin{pmatrix}
(\lambda,e)\\
(\eta,\omega_0)+p(\lambda,e)
\end{pmatrix}
\ne
0
\end{equation}
for any $((\lambda,e),(\eta,\omega_0)) \ne (0,0)$,
$\xi$ is a regular map.
By this proof, $(x,\lambda)$ gives a local parameter
in a neighborhood of $\xi(\beta_0,\omega_0)$.
\end{NB}

\begin{rem}
By \cite{Ma}, the walls form nested circles, 
if we fixed ${\Bbb R}_{>0}\omega$.
By our proof of Proposition \ref{prop:fibration},
the circles are the fibers of $\xi$.
\end{rem}

We set
\begin{equation}
\begin{split}
x_0:= & e^\delta H=H+(H,\delta)\varrho_X,\\
x_1:= & -e^\delta 
\left(D+\left(1+\frac{(v^2)}{2r^2}\varrho_X \right)\right)=
-\left(D+(D,\delta)\varrho_X+
\left(e^\delta+\frac{(v^2)}{2r^2}\varrho_X \right)\right).
\end{split}
\end{equation}
Let $L$ be the line in ${\Bbb P}(v^\perp)$ passing through
$x_0$ and $x_1$.
Then $\xi(\delta+sH+D,tH) \in L$ for all $(s,t)$.
The image of the unbounded chamber in $(s,t)$ with $s<0$ and $t>0$
is the interior of a segment connecting $x_0$
and $x_0+\epsilon x_1$ ($0<\epsilon \ll 1$).
\begin{NB}
${\cal M}_H^{\delta+D}(v)=
{\cal M}_{(\delta+sH+D,tH)}(v)$
for $(s,t)$ in the unbounded chamber.
\end{NB}

\begin{lem}\label{lem:isom}
Let $H$ be an ample divisor on $X$.
\begin{enumerate}
\item[(1)]
${\cal M}_{(\delta+sH+D,tH)}(v)={\cal M}_H^{\delta+D-\frac{1}{2}K_X}(v)$
if $-1 \ll s <0$.
\item[(2)]
If $(\beta,\omega)$ is general, then
there is $\beta' \in \NS(X)_{\Bbb Q}$ with 
${\cal M}_{(\beta,\omega)}(v) \cong {\cal M}_{(\beta',tH)}(v)$.
In particular, ${\cal M}_{\omega}^{\beta-\frac{1}{2}K_X}(v) \cong
{\cal M}_{(\beta',tH)}(v)$.
\end{enumerate}
\end{lem}

\begin{proof}
(1) is a consequence of Proposition \ref{prop:large}.
(2)
For a chamber in $C^+(v)$ containing $\xi(\beta,\omega)$,
we take a vector $u$ and write it
as in \eqref{eq:u} with \eqref{eq:zeta}.
Then there is $(s,t)$ such that
$\xi(\delta+sH+D/y,tH) =u$.
For $\beta':=\delta+sH+D/y$,
${\cal M}_{(\beta,\omega)}(v) \cong {\cal M}_{(\beta',tH)}(v)$.
\end{proof}


We shall remark the behavior of $\xi(\beta,tH)$ under an isometry 
of $H^*(X,{\Bbb Q})_{\alg}$.
Let 
$$
\Phi:H^*(X,{\Bbb Q})_{\alg} \to
H^*(X,{\Bbb Q})_{\alg}$$ 
be an isomery.
Then $\Phi$ induces an isomorphism
$\Phi^+:H^*(X,{\Bbb Q})_{\alg}^+ \to H^*(X,{\Bbb Q})_{\alg}^+$
of positive 2-plane 
$H^*(X,{\Bbb Q})_{\alg}^+={\Bbb Q}(1-\varrho_X)+{\Bbb Q}H$.
Assume that $\Phi$ is an isometry such that
$\Phi(r_1 e^\gamma)=\varrho_{X_1}$,
$\Phi(\varrho_X)=r_1 e^{\gamma'}$ and $\Phi^+$ preserves
the orientation of $H^*(X,{\Bbb Q})_{\alg}^+$.
Then we can describe the action as
\begin{equation*}
\Phi(r e^\gamma+a \varrho_X+\xi+(\xi,\gamma)\varrho_X)
=\frac{r}{r_1} \varrho_{X}+r_1 a e^{\gamma'}-
\frac{r_1}{|r_1|}
( \widehat{\xi}+(\widehat{\xi},\gamma')\varrho_{X}),
\end{equation*}
where $\xi \in \NS(X)_{\Bbb Q}$ and $
\widehat{\xi}:=
\frac{r_1}{|r_1|} 
c_1(\Phi(\xi+(\xi,\gamma)\varrho_X)) \in \NS(X)_{\Bbb Q}$.
We note that $\xi$ belongs to the positive cone if and only
if $\widehat{\xi}$ belongs to the positive cone.
For $(\beta,\omega) \in \NS(X)_{\Bbb R} \times P^+(X)_{\Bbb R}$,
we set 
\begin{equation}\label{eq:tilde(beta)}
\begin{split}
\widetilde{\omega}:= & -\frac{1}{|r_1|}
\frac{\frac{((\beta-\gamma)^2)-(\omega^2)}{2}\widehat{\omega}-
(\beta-\gamma,\omega)(\widehat{\beta}-\widehat{\gamma})}
{\left(\frac{((\beta-\gamma)^2)-(\omega^2)}{2} \right)^2
+(\beta-\gamma,\omega)^2},\\
\widetilde{\beta}:= & \gamma'-\frac{1}{|r_1|}
\frac{\frac{((\beta-\gamma)^2)-(\omega^2)}{2}(\widehat{\beta}-\widehat{\gamma})
-(\beta-\gamma,\omega) \widehat{\omega}}
{\left(\frac{((\beta-\gamma)^2)-(\omega^2)}{2} \right)^2
+(\beta-\gamma,\omega)^2}.
\end{split}
\end{equation}
Then $(\widetilde{\beta},\widetilde{\omega}) 
\in \NS(X)_{\Bbb R} \times P^+(X)_{\Bbb R}$.

\begin{proof}
By our assumption, $\widehat{\omega} \in P^+(X)_{\Bbb R}$.
It is sufficient to prove $(\widetilde{\omega},\widetilde{\omega})>0$
and $(\widetilde{\omega},\widehat{\omega})>0$,
which follows from the following equations:
\begin{equation}\label{eq:tilde(beta)2}
\begin{split}
(\widetilde{\omega}^2)= & \frac{1}{|r_1|^2}
\frac{(\omega^2)}
{\left(\frac{((\beta-\gamma)^2)-(\omega^2)}{2} \right)^2
+(\beta-\gamma,\omega)^2},\\
(\widetilde{\omega},\widehat{\omega})=&
 \frac{1}{|r_1|}
\frac{(\omega^2)^2+(\beta-\gamma,\omega)^2-(D^2)(\omega^2)}
{2\left(\left(\frac{((\beta-\gamma)^2)-(\omega^2)}{2} \right)^2
+(\beta-\gamma,\omega)^2 \right)},
\end{split}
\end{equation}
where $\beta-\gamma=\lambda \omega+D$ ($\lambda \in {\Bbb R}$,
$D \in \omega^\perp$).
\end{proof}

By \cite[sect. A.1]{MYY:2011:2},
we get the following commutative diagram:
\begin{equation}
\xymatrix{
   H^*(X,{\Bbb Q})_{\alg} \ar[r] \ar[d]_{Z_{(\beta,\omega)}}
 & H^*(X,{\Bbb Q})_{\alg} \ar[d]^{Z_{(\wt{\beta},\wt{\omega})}} \\
   {\Bbb C}  \ar[r]_{\zeta^{-1}} 
 & {\Bbb C}
}
\end{equation}
where 
$$
\zeta=r_1 \left(
\frac{((\gamma-\beta)^2)-(\omega^2)}{2}
+\sqrt{-1}(\beta-\gamma,\omega) \right).
$$

\begin{NB}
\begin{prop}
Assume that $Z_{(\beta,tH)}(r_1 e^\gamma) \in {\Bbb R}_{>0}
Z_{(\beta,tH)}(v)$.
Then $\Phi(\xi(\beta,H,t)) \in {\Bbb R}_{>0}
(\widetilde{tH}+(\widetilde{tH},\widetilde{\beta})\varrho_{X_1})$.
\end{prop}

\begin{proof}
\begin{equation}
\Phi( e^{\beta+\sqrt{-1}tH})=
\zeta e^{\widetilde{\beta}+\sqrt{-1}\widetilde{tH}}.
\end{equation}

Since 
\begin{equation}
\begin{split}
Z_{(\beta,tH)}(v)=& \frac{r}{2}t^2 (H^2)-a+(c_1-r \beta,tH)\sqrt{-1}\\
=& r \frac{t^2 (H^2) -((c_1/r-\beta)^2)}{2}
+\frac{\langle v^2 \rangle}{2r} +(c_1-r \beta,tH)\sqrt{-1}
\end{split}
\end{equation}
and $Z_{(\beta,tH)}(r_1 e^\gamma)=\zeta$,
we have
\begin{equation}
{\Bbb R}_{>0}\xi(\beta,H,t)=
{\Bbb R}_{>0}\mathrm{Im}(Z_{(\beta,tH)}(v)^{-1} e^{\beta+\sqrt{-1}tH})=
{\Bbb R}_{>0}\mathrm{Im}(\zeta^{-1} e^{\beta+\sqrt{-1}tH}).
\end{equation}
\end{proof}
\end{NB}

\begin{prop}\label{prop:Phi-xi}
For $(\beta,\omega) \in \NS(X)_{\Bbb R} \times P^+(X)_{\Bbb R}$, we have
$\Phi(\xi(\beta,\omega))=
\xi(\widetilde{\beta},\widetilde{\omega})$.
\end{prop}

\begin{proof}
The proof is completely the same as of \cite[Prop. 3.7]{Movable}.
\end{proof}
\begin{NB}
\begin{proof}
\begin{equation}
\begin{split}
{\Bbb R}_{>0}\Phi(\xi(\beta,\omega))=&
{\Bbb R}_{>0}\Phi(\mathrm{Im}
(Z_{(\beta,\omega)}(v)^{-1} e^{\beta+\sqrt{-1}\omega}))\\
=& {\Bbb R}_{>0} \mathrm{Im}(Z_{(\beta,\omega)}(v)^{-1} 
\Phi(e^{\beta+\sqrt{-1}\omega}))\\
=& {\Bbb R}_{>0} 
\mathrm{Im}(Z_{(\widetilde{\beta},\widetilde{\omega})}(\Phi(v))^{-1} \zeta^{-1}
\zeta e^{\widetilde{\beta}+\sqrt{-1}\widetilde{\omega}})\\
=& {\Bbb R}_{>0} \xi(\widetilde{\beta},\widehat{\omega}).
\end{split}
\end{equation}
\end{proof}
\end{NB}

\section{Stability conditions on a blow-up}\label{sect:stability-blowup}

\subsection{Stability conditions for $(\beta,tH)$}
Let $\pi:X \to Y$ be the blow-up of a point as 
in section \ref{sect:blowup}.
In this subsection, we shall study the map
$\xi$ in a neighborhood of 
$H+(H,\delta)\varrho_X$, where
$H \in \pi^*(\Amp(Y)_{\Bbb Q})$.
We start with the following easy fact.
\begin{lem}\label{lem:lambda}
Assume that $\omega \in \pi^*(\Amp(Y)_{\Bbb R})$.
\begin{enumerate}
\item[(1)]
$Z_{(\beta,\omega)}({\cal O}_C(-a))=0$ if and only if
$(\beta,C)=-a+\frac{1}{2}$.
\item[(2)]
If $(\beta-\delta,\omega) \ne 0$,
then $Z_{(\beta,\omega)}({\cal O}_C(-a))=0$ if and only if
$\xi(\beta,\omega) \in v({\cal O}_C(-a))^\perp$.
\end{enumerate}
\end{lem}

\begin{proof}
  (1) is obvious.
Since $Z_{(\beta,\omega)}(v) \not \in {\Bbb R}$ for
$(\beta-\delta,\omega) \ne 0$,
$\xi(\beta,\omega) \in v({\cal O}_C(-a))^\perp$
if and only if $Z_{(\beta,\omega)}({\cal O}_C(-a))=0$.
Thus (2) holds. 
\end{proof}

\begin{NB}
\begin{rem}
If $\xi(\beta,tH) \not \in \cup_a v({\cal O}_C(-a))^\perp$, 
that is, $(\beta,C) \not \in \frac{1}{2}+{\Bbb Z}$, then
${\frak C}^\beta$ is determined by $\xi(\beta,\omega)$.
\end{rem}
\end{NB}

\begin{lem}\label{lem:fiber2}
\begin{enumerate}
\item[(1)]
For $u=H+(H,\delta)\varrho_X$ with $H \in \pi^*(\Amp(Y)_{\Bbb Q})$,
\begin{equation}
\xi^{-1}(u) \cap \overline{\cal K}(X) 
=\{(\beta,\omega) \mid 
\beta-\delta \in H^\perp, (\beta,C) \not \in \tfrac{1}{2}+{\Bbb Z},
\omega \in {\Bbb R}_{>0}H
\}.
\end{equation}
In particular, $\xi^{-1}(u) \cap \overline{\cal K}(X)$
is not connected.
\begin{NB}
By the disconnectivity,
the category ${\frak C}^\beta$ depends on the choice of $\beta$.
\end{NB}
\item[(2)]
Assume that $\rk u = -1$.
\begin{enumerate}
\item
If $u \not \in \cup_{a \in {\Bbb Z}} v({\cal O}_C(-a))^\perp$, then
\begin{equation}
\xi^{-1}(u) \cap \overline{\cal K}(X)=
\{(\beta,\omega) \in \xi^{-1}(u) \mid \omega \in \Amp(X)_{\Bbb R} \cup
\pi^*(\Amp(Y)_{\Bbb R}) \}.
\end{equation}
\item
If $u \in \cup_{a \in {\Bbb Z}} v({\cal O}_C(-a))^\perp$, then
\begin{equation}
\xi^{-1}(u) \cap \overline{\cal K}(X)=
\{(\beta,\omega) \in \xi^{-1}(u) \mid \omega \in \Amp(X)_{\Bbb R} \}.
\end{equation}
\end{enumerate}
In particular, $\xi^{-1}(u) \cap \overline{\cal K}(X)$
is connected.
\end{enumerate}
\end{lem}

We fix $H \in \pi^*(\Amp(Y)_{\Bbb Q})$.
\begin{NB}
For $u=H+(H,\delta)\varrho_X$,  
$$
\xi^{-1}(u) \cap \Stab(X)=\{(\beta,\omega) \mid 
\beta-\delta \in H^\perp, (\beta,C) \not \in \tfrac{1}{2}+{\Bbb Z}
\},
$$
since $Z_{(\beta,\omega)}({\cal O}_C(-a))=(C,\beta)-\tfrac{1}{2}+a
\ne 0$ for all $a \in {\Bbb Z}$.

For $u=e^\delta(xH+pC+(1+\frac{(v^2)}{2r^2}\varrho_X))$,
$u \in \ch {\cal O}_C(-a)^\perp$ if and only if 
$p=-\frac{1}{2}+\lambda)$.
These are lines on the $x,p$-plane.
If $\omega=tH$, then $(\beta,\omega)$ define a stability
condition on the complement of 
$$
\bigcup_a \xi^{-1}(v({\cal O}_X(-a))^\perp)=
\bigcup_a \xi^{-1}(\{u \mid p=-\tfrac{1}{2}+a \}).
$$

In $\NS(X)_{\Bbb R} \times {\Bbb R}_{>0}H$,
$\xi^{-1}(\{u \mid p=-\tfrac{1}{2}+a\})$ is a family
of semi-circles over a line.

If $\omega$ is ample, then
$$
\bigcup_a \xi^{-1}(\{u \mid p=-\tfrac{1}{2}+a\})
$$
is a candidate of a wall.
\end{NB}
We have an inclusion
$$
\iota:H^\perp \times {\Bbb R} \to \NS(X)_{\Bbb R} \times P^+(X)_{\Bbb R}
$$
such that $\iota(D,s)=(\delta+sH+D,H)$.
By \eqref{eq:xi-explicit},
\begin{equation}
\xi(\delta+sH+D,H)=
e^\delta \left\{\frac{r^2((H^2)(s^2+1)-(D^2))+(v^2)}{2r^2}H+
s(H^2)\left(D+\left(1+\frac{(v^2)}{2r^2}\varrho_X \right)\right)\right\}.
\end{equation}
\begin{NB}
Not correct:
Hence
$\xi_{|\im \iota}$ is surjective in a neighborhood of
$H+(H,\delta)\varrho_X$ in $C^+(v) \setminus 
(\varrho_X^\perp \setminus {\Bbb R}_{>0}(H+(H,\delta)\varrho_X))$.
\end{NB}
$\xi(\im \iota)$ consists of
$e^\delta(H+X+y(1+\tfrac{(v^2)}{2r^2}\varrho_X))$
satisfying
\begin{equation}
(H^2) \geq -(X^2)+y^2((H^2)+\frac{(v^2)}{r^2}), y \ne 0
\end{equation}
or
$(X,y)=(0,0)$.
Moreover if $s<0$, then $y<0$.
\begin{NB}
We set 
\begin{equation}
\xi(\delta+sH+D,H)=e^\delta
\left(\frac{(H^2)(s^2+1)-(D^2)}{2}+\frac{(v^2)}{2r^2} \right)
\left(H+X+y\left(1+\frac{(v^2)}{2r^2}\varrho_X \right)\right).
\end{equation}
Then $X=yD$ and 
$y\left(\frac{(H^2)(s^2+1)-(D^2)}{2}+\frac{(v^2)}{2r^2} \right)=s(H^2)$.
Hence we see that
\begin{equation}
\left(y-\frac{s}{(s^2+1)+\frac{(v^2)}{r^2 (H^2)}} \right)^2+
\frac{-(X^2)}{(H^2)((s^2+1)+\frac{(v^2)}{r^2 (H^2)})}=
\frac{s^2}{(s^2+1+\frac{(v^2)}{r^2 (H^2)})^2}.
\end{equation}
It is a quadric of elliptic type and having tangent hyperplane
$y=0$ at $(X,y)=(0,0)$.
Since it is a quadratic equation for $s$ for $y \ne 0$,
the condition for the existence of $s$ is
\begin{equation}
y^2(H^2)^2-y^2(H^2)(y^2((H^2)+\frac{(v^2)}{r^2})-X^2) \geq 0
\end{equation}
if $y \ne 0$.
Hence the condition is $(X,y)=(0,0)$ or 
$y \ne 0$ and
\begin{equation}
(H^2) \geq -(X^2)+y^2((H^2)+\frac{(v^2)}{r^2}).
\end{equation}
\begin{rem}
For $s^2 \gg 0$ and $D$ is sufficiently large, 
$\xi(\delta+sH+D,H)$ is close to $y=0$. 
\end{rem}

We next consider a restricted parameter space
$\xi(\delta+sH+D_0+pC,H)$.
In this case, we set $X=yD_0+x C$ and $X=yD$, where
$D=D_0+(x/y) C=D_0+pC$.
Then 
\begin{equation}
\xi(\delta+sH+D_0+pC,H)=e^\delta
\left(\frac{(H^2)(s^2+1)-(D^2)}{2}+\frac{(v^2)}{2r^2} \right)
\left(H+xC+y D_0+y\left(1+\frac{(v^2)}{2r^2}\varrho_X \right)\right).
\end{equation}
Hence
\begin{equation}
\left(y-\frac{s}{A} 
\right)^2+
\frac{(x-y(D_0,C))^2}{(H^2)A}=
\frac{s^2}{A^2}
\end{equation}
where
$$
A:=(s^2+1)+\frac{(v^2)}{r^2 (H^2)}+(-(D_0^2)-(D_0,C)^2)
\geq (s^2+1)+\frac{(v^2)}{r^2 (H^2)}.
$$
Thus
\begin{equation}
\xi(\delta+sH+D_0+pC,H)=e^\delta
\left(H+(x-y(D_0,C))C+y(D_0+(D_0,C)C)+
y\left(1+\frac{(v^2)}{2r^2}\varrho_X \right)\right) \in C^+(v).
\end{equation}
\end{NB}
For a fixed $D$, $\xi \circ \iota(D,s)=\xi(\delta+sH+D,H)$ is a line
passing $H+(H,\delta) \varrho_X$ and 
\begin{equation}
\iota^{-1} (\xi^{-1}(H+(H,\delta)\varrho_X))=
\{ (D,0) \mid D \in H^\perp \}.
\end{equation}
\begin{NB}
If $D \ne D'$, then
$\xi(\delta+sH+D,H)=\xi(\delta+s' H+D',H)$
if and only if $s=s'=0$.
\end{NB}
We also have
$$
\im \iota \cap \overline{\cal K}(X)=
\{(\delta+sH+D,H) \mid (\delta+D,C) \not \in \tfrac{1}{2}+{\Bbb Z} \}.
$$
\begin{NB}
Let
$\iota':{\Bbb R}C \times {\Bbb R} \to \NS(X)_{\Bbb R} \times P^+(X)_{\Bbb R}$
be the restriction of $\iota$ to ${\Bbb R}C \times {\Bbb R}$.
Then 
$$
\im \iota' \cap 
\overline{\cal K}(X)= 
\{ \delta+sH+pC \mid p \not \in \frac{1}{2}+{\Bbb Z} \}.
$$
For a Mukai vector $v_1$ with $v_1 \perp H+(H,\delta)\varrho_X$,
we write $v_1=\frac{r_1}{r}v+e^\delta(qC+D+b \varrho_X)$,
where $D \in H^\perp \cap C^\perp$.
Then $(\xi \circ \iota')^{-1}(v_1^\perp)=\{(p,s) \mid s(pq+b)=0 \}$. 
\end{NB}

\begin{lem}\label{lem:isom2}
Assume that
$(\beta,\omega),(\beta',\omega') \in \overline{\cal K}(X)$
satisfy
$\xi(\beta,\omega)=\xi(\beta',\omega') \not \in 
\cup_{a \in {\Bbb Z}} v({\cal O}_C(-a))^{\perp}$.
Then
${\cal M}_{(\beta,\omega)}(v)={\cal M}_{(\beta',\omega')}(v)$.
\end{lem}

\begin{proof}
By Lemma \ref{lem:fiber2} (2) and Lemma \ref{lem:xi-wall}, 
the claim follows.
\end{proof}

\begin{rem}\label{rem:beta'}
Assume that $H \in \pi^*(\Amp(Y))$.
Then the category ${\frak C}^{\beta'}$ is determined by
the integer $l$ satisfying
$l-\frac{1}{2}<(\beta',C)<l+\frac{1}{2}$, where
\begin{equation}
(\beta',C)=
\frac{r^2((\omega^2)-((\beta-\delta)^2))+(v^2)}
{2r^2 (\beta-\delta,\omega)}
(\omega,C)
+(\beta,C).
\end{equation}
Indeed
by \eqref{eq:def-xi} and the proof of Lemma \ref{lem:isom},
$$
\beta' \equiv \beta+\frac{r^2((\omega^2)-((\beta-\delta)^2))+(v^2)}
{2r^2 (\beta-\delta,\omega)}\omega \mod {\Bbb R}H.
$$
\end{rem}

\begin{rem}
For
the fiber $\xi^{-1}(H+(H,\delta)\varrho_X)$,
Lemma \ref{lem:isom2} does not
hold by Lemma \ref{lem:lambda} (1). 
Moreover if $(\omega,C)<0$, then Lemma \ref{lem:isom2} does not
hold. 
Thus the structure of $\Stab(X)$ seems to be complicated if $(\omega,C)<0$.
\end{rem}

\begin{NB}
Assume that $\xi(\beta,\omega)=H+(H,\delta)\varrho_X$. Then
the $S$-equivalence class of $E \in {\cal M}_{(\beta,\omega)}(v)$ 
is represented by $\oplus_i E_i$
such that each $E_i$ is a $\mu$-stable  
objects of ${\frak C}^\beta$ with 
$(c_1(E_i(-\beta)),H)=0$ and $E_i(lC)$ is the pull-back of
a locally free sheaf on $Y$
or irreducible objects
${\cal O}_C(l)[-1],{\cal O}_C(l-1), k_x[-1]$ ($x \in X \setminus C$)
of ${\frak C}^\beta[-1]$
by Lemma \ref{lem:reflexive-hull2}.
By Lemma \ref{lem:perverse-Bogomolov},
$(v(E_i)^2) \geq 0$ if $\rk E_i>0$.
\end{NB}

By Proposition \ref{prop:large}, we get the following claim.
\begin{prop}\label{prop:unbdd}
Assume that $(\beta,tH)$ belongs to an adjacent chamber 
of $\varrho_X^\perp$ and $(\beta-\delta,H)<0$.
Then
${\cal M}_{(\beta,tH)}(v)={\cal M}_H^{\beta-\frac{1}{2}K_X}(v)$.
\end{prop}

\begin{proof}
We can take a sufficiently large $t'$ such that 
$\xi(\beta,t' H)$ belongs to the adjacent chamber ${\cal C}$.
Since $(\beta,tH)$ also belongs to ${\cal C}$, 
we have ${\cal M}_{(\beta,tH)}(v) \cong {\cal M}_{(\beta,t'H)}(v)$.
By Proposition \ref{prop:large},
${\cal M}_{(\beta,t'H)}(v) \cong {\cal M}_H^{\beta-\frac{1}{2}K_X}(v)^{ss}$.
Thus the claim holds.
\end{proof}

The following claim gives an explicit example
of $(\beta,tH)$ in Proposition \ref{prop:unbdd}
\begin{prop}\label{prop:min}
We set $\mu:=\min \{(D,H)>0 \mid D \in \NS(X) \}$.
Assume that $v$ satisfies 
$\gcd(r,(r\delta,H)/\mu)=1$.
We take $r_0 \in {\Bbb Z}$ and $\xi_0 \in \pi^*(\NS(Y))$ such
that $(r\xi_0-r_0 (r\delta),H)=-\mu$.
Then 
$$
{\cal M}_{(\delta+sH+D,tH)}(v)={\cal M}_H^{\delta+D-\frac{1}{2}K_X}(v)
$$
if 
\begin{equation}\label{eq:min-s}
-\frac{\mu}{rr_0 (H^2)} \leq s<0.
\end{equation}
\begin{NB}
\begin{equation}
\left(s(H^2)-(\beta_0-\delta,H) \right)
\left(s(\beta_0-\delta,H)-\left(\frac{(v^2)}{r^2}-(D^2)\right) \right)
\geq 0.
\end{equation}
\end{NB}
\end{prop}

\begin{proof}
Assume that $s$ satisfies \eqref{eq:min-s}.
We set 
\begin{equation}
\begin{split}
\gamma:=& \delta+D+\frac{(\xi_0-r_0 \delta,H)}{r_0 (H^2)}H\\
=& \delta+D-\frac{\mu}{r r_0 (H^2)}H.
\end{split}
\end{equation}
Then
$r_0 (\gamma,H)=(\xi_0,H)$ and
$(c_1(E),H)-\rk E (\gamma,H) \in \frac{\mu}{r_0}{\Bbb Z}$
for all $E \in {\bf D}(X)$.
Since $d_\gamma(v)(H^2)=\frac{\mu}{r_0}$,
there is no wall intersecting 
$\{(sH+D,tH) \mid t>0 \}$.
Hence the claim holds.
\begin{NB}
For $s=(\beta_0,H)/(H^2)$,
$d_{\beta_0}(v)(H^2)=\frac{\mu}{r_0}$.
Hence there is no wall intersecting 
$\{(sH+D,tH) \mid t>0 \}$.
Under the assumption,
we have
\begin{equation}
\frac{(H^2)s^2+\frac{(v^2)}{r^2}-(D^2)}
{2s(H^2)}
\geq 
\frac{\frac{(\beta_0-\delta,H)^2}{(H^2)}+\frac{(v^2)}{r^2}-(D^2)}
{2(\beta_0-\delta,H)}.
\end{equation}
Then there is $t'>0$ such that
$\xi(\delta+sH+D,tH)=\xi(\beta_0,t' H)$.
Hence the claim holds.
\end{NB}
\end{proof}

\begin{NB}
It seems that the following is not needed (Oct.28):
Hence $E \in {\cal M}_{(\beta,tH)}(v)$ is $S$-equivalent to
$\oplus_i E_i$ such that $E_i$ are $\mu$-stable torsion free objects
of ${\frak C}^\beta$.
In particular,
$(v(E_i)^2) \geq 0$ for all $i$.
Then we see that $(v(E_i)^2) \leq r_i (v^2)/r^2$, where
$r_i=\rk E_i$.
Therefore the choice of $v(E_i)$ is finite.
\end{NB}

\begin{NB}
For $\omega=H+pC$ with $\sqrt{(H^2)}>p>0$,
\end{NB}

Toda \cite{Toda}, \cite{Toda2} constructed a stability condition
$\sigma_{(0,\pi^*(\omega')-qC)}=
(Z_{(0,\pi^*(\omega')-qC)},{\cal A}_{\omega'} (X/Y))$, where
$q<0$,
$Z_{\pi^*(\omega')-qC}(E)=(e^{\pi^*(\omega')-qC},v(E))$,
${\cal A}_{\omega'}(X/Y)=
\langle {\bf L}\pi^* {\cal A}_{(0,\omega')} ,
{\cal C}_{X/Y}^0 \rangle$
and ${\cal C}_{X/Y}^0$ is spanned by
${\cal O}_C[-1]$.
By the same proof in \cite{Toda2}, 
we shall generalize ${\cal A}_{\omega'}(X/Y)$ 
as ${\cal A}_{(\beta',\omega')}(X/Y):=
\langle {\bf L}\pi^* {\cal A}_{(\beta',\omega')}, {\cal O}_C[-1] \rangle$,
where $(\beta',\omega') \in \NS(Y)_{\Bbb Q} \times 
\NS(Y)_{\Bbb Q}$.
Then 
$$
\sigma_{(\pi^*(\beta'),\pi^*(\omega')-qC)}
:=(Z_{(\pi^*(\beta'),\pi^*(\omega')-qC)},{\cal A}_{(\beta',\omega')}(X/Y))
$$
is a stability condition.
Assume that $v=\pi^*(v')$ with $v' \in H^*(Y,{\Bbb Q})$.
\begin{NB}
We note that $-\frac{1}{2}<(\beta,C)<\frac{1}{2}$.
If $\beta$ is sufficiently
close to $\delta$, then
${\cal M}_{(\beta',H')}(v') \cong {\cal M}_H^{\beta'}(v')$.
\end{NB}
Then
${\cal M}_{(\pi^*(\beta'),H-qC)}(v) \cong {\cal M}_{(\beta',H')}(v')$,
where $H=\pi^*(H')$.

\subsection{A classification of walls}\label{subsect:classification-wall}

\begin{NB}
\begin{lem}\label{lem:O_C}
Let $L$ be an ample divisor on $X$.
Let $E$ be a $\beta$-twisted stable sheaf such that $\Supp(E)=C$.
Then $E \cong {\cal O}_C(a)$. 
\end{lem}

\begin{proof}
We may assume that $E$ is an ${\cal O}_{nC}$-module
and the multiplication of $E$ by $(n-1)C$ is not zero.
Let $F$ be the torsion free quotient of $E_{|C}$
as an ${\cal O}_C$-module.
Then 
we have an injective morphism
$F(-(n-1)C) \to E$.
By the stability of $E$,
$$
\frac{\chi(F(n-1-\beta))}{(c_1(F),L)} \leq 
\frac{\chi(E(-\beta))}{(c_1(E),L)} \leq
\frac{\chi(F(-\beta))}{(c_1(F),L)}.
$$
Therefore $n=1$.
Then the classification of locally free sheaves
on $C$, we get the claim.
\end{proof}

Let $(\beta',\omega')$, $\omega' \in \Amp(X)_{\Bbb Q}$ be sufficiently close to
$(\beta,\omega)$ with $\omega \in \pi^*(\Amp(Y)_{\Bbb Q})$.

\begin{lem}\label{lem:perverse-wall}
Let $E$ be a $\sigma_{(\beta',\omega')}$-stable object
such that $E[1]$ is a 0-dimensional object of ${\frak C}^\beta$.
Then $E \cong {\cal O}_C(a)$, $a \leq l-1$.
\end{lem}

\begin{proof}
We note that $E[1]$ is generated by 
${\cal O}_C(l), {\cal O}_C(l-1)[1]$.
Since $H^i(E[1])=0$ for $i \ne -1,-2$,
(i) $\Hom(E[1],{\cal O}_C(l))=0$ and
(ii) $\Hom(E[1],{\cal O}_C(l-1)[1]) \ne 0$.
Since $E[1] \in {\frak C}^\beta$,
$H^{-2}(E[1])=0$.
Thus $E$ is a torsion sheaf on $X$.
Then $\sigma_{(\beta',\omega')}$-stability
of $E$ implies that $E$ is a stable 1-dimensional sheaf. 
By Lemma \ref{lem:O_C} and
(ii), we get the claim.      
\end{proof}

\begin{NB2}
The ampleness of $\omega'$ is important. For the classification of walls,
it is not sufficient to study this case.
\end{NB2}

\begin{rem}
$\beta$ satisfies
$\xi(\beta',\omega')=\xi(\beta,tH)$.
$\sigma_{(\beta',\omega')}$-stability of $E$ is the same as
$\sigma_{(\beta,tH)}$-stability of $E$ by Lemma \ref{lem:xi-wall}.
By our assumption,
along a curve connecting 
$\sigma_0:=\sigma_{(\beta,tH)}$ and 
$\sigma_1:=\sigma_{(\beta-\frac{(\beta-\delta,H)}{(H^2)}H,t'H)}$,
semi-stability does not change. 
Hence $E$ is a $\sigma_1$-semi-stable.
If $\rk E=(c_1(E),H)=0$, then 
$E[1]$ is a 0-dimensional object of 
${\frak C}^\beta$ and $\phi_{\sigma_1}(E)=\phi_{\sigma_1}(v)=0$.
\end{rem}
\end{NB}

\begin{NB}
If $(\beta',\omega')$ is close to $(0,tH)$, 
then 
${\cal O}_C$, $a \geq 0$ does not define a wall.
 
\end{NB}

\begin{NB}
Let $U$ be an open neighborhood of $(\beta,\omega)$ such that
$\overline{U} \subset \NS(X)_{\Bbb R} \times P^+(X)_{\Bbb R}$ 
is compact.
We consider the set $S$ of Mukai vectors 
$v(E_1)$ such that 
$E_1$ is $\sigma_{(\beta',\omega')}$-stable subobject
of $E \in {\cal M}_{(\beta',\omega')}(v)$,
$\phi_{(\beta',\omega')}(E_1)=\phi_{(\beta',\omega')}(E)$
and $(\beta',\omega') \in U$.
Then $S$ is a finite set.
Replacing $U$ by a small neighborhood,
we may assume that $E_1$ is $\sigma_{(\beta',\omega')}$-stable
in a chamber for $v(E_1)$ whose closure 
contains $\xi^{-1}(H+(H,\delta)\varrho_X)$.
Then by Lemma \ref{lem:perverse-wall},
$E_1={\cal O}_C(a)$ with $a \leq l-1$.
\end{NB}

\begin{NB}

\begin{lem}
We take a sufficiently small neighborhood $U$ of
$H+(H,\delta)\varrho_X$, where
$H \in \pi^*(\Amp(Y))$.
If $(\beta,\omega) \in \xi^{-1}(U)$ and
$\xi(\beta,\omega) \not \in \cup_a v({\cal O}_C(a))^\perp$,
then $(\beta,\omega)$ is general with respect to $v$.
\end{lem}

\begin{proof}
We take $(\beta',tH)$ with $\xi(\beta,\omega)=\xi(\beta',tH)$.
If there is a $\sigma_{(\beta',tH)}$-stable subobject
$E_1$ of $E \in {\cal M}_{(\beta',tH)}(v)$ with
$\phi_{(\beta',tH)}(E_1)=\phi_{(\beta',tH)}(E)$,
then $(v(E_1),H+(H,\delta)\varrho_X)=0$
by the choice of $U$.
If $\rk E_1=0$, then
$(c_1(E_1),H)=0$ and $E_1$ is a 0-dimensional object 
of ${\frak C}^{\beta'}$.
Then we have
$Z_{(\beta',tH)}(E_1) \in {\Bbb R}_{<0}$ and
$Z_{(\beta',tH)}(E) \not \in {\Bbb R}$, which is a contradiction.
\end{proof}

\end{NB}

\begin{NB}
We shall study walls in a neighborhood of
$(\delta-\epsilon H,tH)$.

We may assume that $(\omega,C)<0$.
For a $\sigma_{\omega+tC}$-stable object $E$,
we have an exact sequence
\begin{equation}
0 \to F \to E \to {\bf L}\pi^* M \to 0
\end{equation}
where $F \in \langle {\cal O}_C[-1] \rangle$.
Assume that $E[1]$ is a 0-dimensional object of
${^{-1}\Per}(X/Y)$.
We note that
$\Hom({\cal O}_C[-1],{\cal O}_C(-1))=0$.
Since $\Hom({\bf L}\pi^* M,{\cal O}_C(-1))=0$,
$\Hom(E,{\cal O}_C(-1))=0$.
By the stability of $E$,
$E=F$ or $E={\bf L}\pi^* M$.
Since $\Ext^1({\cal O}_C,{\cal O}_C)=0$,
$E=F$ means $E={\cal O}_C[-1]$.
Since
$\Hom({\bf L}\pi^* M,{\cal O}_C[-1])=0$, the latter case does not
occur. 

Since $\Hom({\cal O}_C[-1],{\bf L}\pi^* M)=0$,
${\cal O}_C[-1]$ does not define a wall for $v$.
Thus there is a neighborhood $U$ of $(\delta,H)$ such that
if $Z_\sigma(v) \in {\Bbb H}$, then
${\cal M}_\sigma(v)$ is constant 
in a connected component of $U \setminus \cup_{a>0} W_a$,
where $W_a=\{\sigma \mid 
\phi_\sigma({\cal O}_C(-a))=\phi_\sigma(E) \}$.  
\begin{lem}
In a neighborhood of $(\delta-\epsilon H,tH)$,
the candidates of the Mukai vectors defining
walls are ${\cal O}_C(a)$, $a<0$.
\end{lem}
\end{NB}

\begin{NB}
We set $\beta' \in \NS(X)_{\Bbb R}$ and $\beta=\pi^*(\beta')$.
${\cal A}_{(\beta,\omega)}(X/Y)=\langle {\cal O}_C[-1],
{\bf L}\pi^* {\cal A}_{(\beta',\omega')} \rangle$, 
where $\omega=\pi^*(\omega')+tC$, $t>0$.
Then $\sigma_{(\beta,\omega)}=(Z_{(\beta,\omega)},{\cal A}_{(\beta,\omega)})$
is a stability condition.
\end{NB}

\begin{NB}
For $(\beta,\omega)=(\delta-\epsilon H,H+t_0 C)$,
$(\beta',tH) \in \xi^{-1}(\xi(\beta,\omega))$ satisfies
\begin{equation}
(\beta',C)=
t_0 \frac{r^2((1-\epsilon^2)(H^2)-t_0^2)+(v^2)}{2\epsilon (H^2)}.
\end{equation}
Hence if $t_0>0$ and $0<\epsilon \ll 1$, then
$l$ is sufficiently large.
On the other hand, $\sigma_{\delta-\epsilon H,H+t_0 C}$
is constant.
Then ${\cal M}_{(\delta-\epsilon H,H+t_0 C)}(v)=
{\cal M}_H^{-\frac{1}{2}K_X}(v)$ and
${\cal M}_{(\beta',tH)}(v)=\emptyset$.
This means that semi-stability is not invariant
on $\xi^{-1}(\xi(\beta,\omega))$ if
$\omega$ is not ample. 

\begin{NB2}
The following is not correct:
For $(\beta,\omega)=(\delta-\epsilon H,H+t_0 C)$,
$\xi(\beta,\omega)=\xi(\beta_t,\omega_t)$,
where 
\begin{equation}
\beta_t:=\delta+x(t)(H+tC)-
\left((t_0-t)\frac{(1-\epsilon^2)(H^2)-t_0^2}{2\epsilon (H^2)}-
\epsilon\right)C,\;
\omega_t=y(t)(H+tC)
\end{equation}
and $(x(t),y(t))$ is determined by $H+tC$ and $\xi(\beta,\omega)$.
We require $(\beta_{t_0},\omega_{t_0})=(\beta,\omega)$.
Then we have $(x(t_0),y(t_0))=(-\epsilon,1)$.
Since $(\beta_t,C)$ is increasing,
${\cal A}_{\sigma_{(\beta_t,\omega_t)}}$ will be non-constant.
Thus ${\cal O}_C(a)$ will gives a wall.
Thus stability will change in $\xi^{-1}(\xi(\beta,\omega))$. 
\end{NB2}
\end{NB}

\begin{NB}
Moduli spaces for
$\xi^{-1} \cap \{(\beta,\omega) \mid (\omega,C)<0 \}$
are quite different from those for
$\xi^{-1} \cap \{(\beta,\omega) \mid (\omega,C)>0 \}$.
\end{NB}

\begin{NB}

For $\beta$ with $(\beta,C)=-1/2$, we consider an exact
sequence in ${\cal A}_{(\beta-\epsilon C,tH)}$.
\begin{equation}
0 \to E_0 \to E \to E_1 \to 0
\end{equation}
$E_0={\cal O}_C(-1)^{\oplus n}$
and $\Hom({\cal O}_C(-1),E_1)=0$.
We set $v(E)=r+dH+D+a \varrho_X$, $r>0$, $D \in H^\perp$.
Since $v({\cal O}_C(-1))=C-\frac{1}{2}\varrho_X$,
we have
\begin{equation}
0 \leq (v(E_1)^2)-((D-nC)^2)=d^2(H^2)-r(2a+n).
\end{equation}
Hence the choice of $n$ is finite.
We define the 
$\sigma_{(\beta,tH)}$-semi-stability
if $\phi_{(\beta,tH)}(E_1) \leq \phi_{(\beta,tH)}(E)$ for all subobject
$E_1 \subset E$ in ${\cal A}_{(\beta-\epsilon C,tH)}$,
where $\phi_{(\beta,tH)}({\cal O}_C(-1)):=\phi_{(\beta,tH)}(v)$.

We generalize $\sigma_{(\beta,\omega)}$-semi-stability
for $(\beta,C)=-m-\frac{1}{2}$.

Assume that 
$\gcd((c_1(v),H)/d_0,r)=1$.
We can take $\beta_0$ such that $d_{\beta_0}(v)=d_{\min,\beta_0}$.
We note that $\beta_0+nC$ also satisfies the same property. 
We write
$\beta_0 \equiv \delta-d_{\min,\beta_0}H+D_0 \mod {\Bbb R}C$,
where $D_0 \in H^\perp$.
We may assume that $D_0 \in C^\perp$.
Then $\xi(\beta_0+sH+pC,tH)$ is contained 
in the plane
$e^\delta(H+{\Bbb R}C+{\Bbb R}(1+D_0+\frac{(v^2)}{2r^2}\varrho_X))$.
Assume that $u \in C^+(v)$ is contained in this
plane and $-\rk u=(u,\varrho_X)>0$.
From now on, we assume that
$s=-d_{\min,\beta_0}$.
Then the image ${\cal D}$ of $\xi$ with $e^\delta H$ is
is a quadric and its interior.
There is no wall except $v({\cal O}_C(n))^\perp$ and
$M_{(\beta,\omega)}(v)$ consists of 
stable perverse coherent sheaves if $\xi(\beta,\omega) \in {\cal D}$
and $\xi(\beta,\omega) \not \in \cup_n v({\cal O}_C(n))^\perp$.
 
\begin{lem}\label{lem:deg-H}
For $(\beta,\omega) \in \xi^{-1}(u)$, $u \in v^\perp$,
let $E$ be a $\sigma_{(\beta,\omega)}$-semi-stable object with
$(v(E),u)=0$.
Then 
$(c_1(E_1(-\beta_0)),H) \geq 0$ for all quotient object 
$E_1$ of $H^0(E)$ and
$(c_1(E_1(-\beta_0)),H) \leq 0$ for all subobject 
$E_1$ of $H^{-1}(E)$, where
$(\beta_0,tH) \in \xi^{-1}(u)$.
\end{lem}

\begin{proof}
Since $E$ is $\sigma_{(\beta,\omega)}$-semi-stable
for any $(\beta,\omega) \in \xi^{-1}(u) \cap {\cal K}(X)$
and $(u,\varrho_X) \ne 0$,
we may assume that $\phi_{(\beta,\omega)}(E) \in (0,1)$
for all $(\beta,\omega)  \in \xi^{-1}(u) \cap {\cal K}(X)$.
In particular $E \in {\cal A}_{(\beta,\omega)}$.
If $(c_1(E_1(-\beta_0)),H) < 0$ for a quotient object 
$E_1$ of $H^0(E)$, then
$(c_1(E_1(-\beta')),\omega') < 0$ for   
a pair $(\beta',\omega')$
in a neighborhood of $(\beta,\omega)$ in $\xi^{-1}(u)$.
Since $E$ is a $\sigma_{(\beta',\omega')}$-semi-stable
object, $H^0(E) \in {\cal T}_{\beta'}$, which is a contradiction.
\end{proof}

\begin{NB2}  
\begin{lem}
Assume that $v(E)=v({\cal O}_C(a))$.
Then $E={\cal O}_C(a)$.
\end{lem}
\begin{proof}
We consider the exact sequence
\begin{equation}
0 \to H^{-1}(E)[1] \to E \to H^0(E) \to 0
\end{equation}
in ${\cal A}_{(\beta,\omega)}$.
Then $H^{-1}(E)$ is torsion free.
Since $(c_1(E(-\beta_0)),H)=0$,
$H^{-1}(E)$ and the torsion free quotient
of $H^0(E)$ are $\mu$-semi-stable sheaves
such that
\begin{equation} 
(c_1(H^{-1}(E)(-\beta_0)),H)=(c_1(H^0(E)(-\beta_0)),H)=0.
\end{equation}
If $H^{-1}(E) \ne 0$, then the Bogomolov inequality implies
$Z_{(\beta_0,tH)}(H^{-1}(E)[1]) \in {\Bbb R}_{<0}$,
which also implies
$Z_{(\beta_0,tH)}(H^0(E)) \in {\Bbb R}_{>0}$.
Since $E$ is $\sigma_{(\beta',\omega')}$-semi-stable for all
$(\beta',\omega')$ and
$\phi_{(\beta',\omega')}(H^{-1}(E)[1])>\phi_{(\beta',\omega')}(H^0(E))$
in a neighborhood of $(\beta_0,tH)$,
$H^{-1}(E)[1]=0$ or $H^0(E)=0$.
By our assumption, $H^0(E)=0$. 
Then $\rk E=\rk H^{-1}(E)<0$, which is a contradiction.
Therefore $H^{-1}(E)=0$.
If $H^0(E)$ contains a 0-dimensional subsheaf $T$, then
$T$ is a subobject of $E$ in ${\cal A}_{(\beta,\omega)}$.
Since $\mathrm{Im}(Z_{(\beta',\omega')}({\cal O}_C(a)))>0$,
$E$ is not $\sigma_{(\beta',\omega')}$-semi-stable.
Therefore $E$ is a purely 1-dimensional sheaf.
Then we have $E={\cal O}_C(a)$.
\end{proof}

\end{NB2}

\begin{lem}\label{lem:d_min}
We set $v:=(r,\xi,a)$, $r>0$.
Assume that $\omega \in \pi^*(\Amp(Y))$,
$d_\beta(v)=d_{\min}$ and $E$ is a
$\sigma_{(\beta',\omega')}$-semi-stable object
with $v(E)=v$
for $(\beta',\omega') \in \xi^{-1}(u) \cap {\cal K}(X)$.
Then $E$ is a coherent sheaf such that the torsion free quotient
of $E$ is $\mu$-semi-stable
with respect to $\omega$ and the torsion part is
supported on $C$. 
\end{lem}

\begin{proof}
We note that
 $d_\beta(H^{-1}(E)[1]) \geq 0$,
$d_\beta(H^0(E)) \geq 0$ implies that
$d_\beta(H^{-1}(E)[1])=0$ or
$d_\beta(H^0(E))=0$.
For the first case,
$\phi_{(\beta',\omega')}(H^{-1}(E)[1])>\phi_{(\beta',\omega')}(E)$
in a neighborhood of $(\beta,\omega)$.
Hence $H^{-1}(E)=0$.
For the second case,
$\rk H^0(E)=0$, which implies that $\rk E<0$.
Therefore $E$ is a sheaf.
Let $T$ be the torsion subsheaf of $E$.
If $(c_1(T),\omega)>0$, then
$(c_1((E/T)(-\beta)),\omega) \leq 0$.
Then $(c_1((E/T)(-\beta')),\omega') <(c_1(E(-\beta')),\omega')$
in a neighborhood of $(\beta,\omega)$,
which is a contradiction.
Therefore $T$ is supported on $C$.
Moreover we see that the torsion free quotient
of $E$ is $\mu$-semi-stable
with respect to $\omega$.
\end{proof}

\begin{lem}
Let $E$ be an object of ${\bf D}(X)$ in Lemma \ref{lem:d_min}.
\begin{enumerate}
\item[(1)]
Assume that $l-\frac{1}{2}<(\beta,C)<l+\frac{1}{2}$.
Then $\Hom({\cal O}_C(l),E)=\Hom(E,{\cal O}_C(l-1))=0$.
In particular $E$ is a $\mu$-stable object of ${\frak C}^\beta$.
\item[(2)]
$\Hom({\cal O}_C(l),E)=0$ and
$\Hom(E,{\cal O}_C(l-2))=0$, if
$\xi(\beta,\omega) \in v({\cal O}_C(l-1))^\perp$.
\end{enumerate}
\end{lem}

\begin{proof}
(1)
Assume that $l-\frac{1}{2}<(\beta,C)<l+\frac{1}{2}$.
Then ${\cal O}_C(l)$ and ${\cal O}_C(l-1)[1]$ are 
$\sigma_{(\beta',\omega')}$-stable in a neighborhood
of $(\beta,\omega)$.
Since $\phi_{(\beta,\omega)}({\cal O}_C(l))=1$ and
$\phi_{(\beta,\omega)}({\cal O}_C(l-1))=0$, we get
$\Hom({\cal O}_C(l),E)=\Hom(E,{\cal O}_C(l-1))=0$. 
Moreover we see that
$\Hom({\cal O}_C(l),E)=0$ and
$\Hom(E,{\cal O}_C(l-2))=0$, if
$\xi(\beta,\omega) \in v({\cal O}_C(l-1))^\perp$.

\begin{NB2}
If $E$ is $\sigma_{(\beta',\omega')}$-stable,
then $\Hom(E,{\cal O}_C(l-1))=\Hom({\cal O}_C(l-1),E)=0$.
Hence $E$ is $\mu$-stable in ${\frak C}^\beta$.
In particular, $(v(E)^2) \geq 0$.
\end{NB2}

(2)
Since we do not know the stability of ${\cal O}_C(l)$ and
${\cal O}_C(l-1)$,
the proof is complicated.
For a quotient sheaf $F$ of $E$,
$E \in T_{\beta,\omega}$ implies 
$(c_1(F(-r \beta)),\omega) \geq 0$.

Let $(\beta_t,\omega_t)$ $(t \in [0,1])$ 
be a continuous family in $\xi^{-1}(u)$
such that
$\beta_t:=a_t H-b_t C+D_t$, $D_t \in H^\perp \cap C^\perp$,
$\omega_t \in {\Bbb R}_t(H-\lambda_t C)$,
$\lambda_0=0$.
We may assume that $|a_t|,|b_t|<B$.
We set
$c_1(F):=xH-yC+D$, $D \in H^\perp \cap C^\perp$.
Then
\begin{equation}
(c_1(F(-\beta_t)),H-\lambda_t C)=
(x-ra_t)(H^2)-(y-rb_t)\lambda_t \geq 0
\end{equation}
for all $t$.
Thus
\begin{equation}\label{eq:y-bound}
y \leq rb_t+\frac{(x-r a_t)}{\lambda_t}(H^2).
\end{equation}  

For a non-zero morphism
$\phi:E \to {\cal O}_C(l-2)$,
let $F$ be the image of $\phi$ in ${\cal A}_{(\beta_s,\omega_s)}$.
Then $F$ is a quotient sheaf of $E$ and
we have an exact sequence
\begin{equation}
0 \to H^{-1}(\coker \phi) \to F \to {\cal O}_C(k) \to 0.
\end{equation}
Since $H^{-1}(\coker \phi) \in {\cal F}_{\beta_s}$,
$(c_1(F(-\beta_s))-C,\omega_s) \leq 0$.
Thus
\begin{equation}
(x-ra_s)(H^2) \leq \lambda_s(1+y-r b_s) 
\leq \lambda_s(1+y+rB).
\end{equation}
Since $F$ is a quotient sheaf of $E$,
applying \eqref{eq:y-bound}, we see that
\begin{equation}
\lambda_s \geq \frac{(x-ra_s)(H^2)}
{1+rb_s+rB+\frac{x-r a_t}{\lambda_t}(H^2)}
\geq
\frac{(x-ra_s)(H^2)}
{1+2rB+\frac{x-r a_s+2rB}{\lambda_t}(H^2)}.
\end{equation}
For a sufficiently small $s$,
$|r(a_s-a_0)|<d_{\min}/2$.
Hence $x-ra_s \geq x-ra_0-d_{\min}/2$.
If $x-ra_0> 0$, then
$x-ra_0 \geq d_{\min}$,
which gives a lower bound of $\lambda_s$:
\begin{equation}
\lambda_s \geq \frac{(H^2)d_{\min}/2}
{1+2rB+\frac{d_{\min}/2+2rB}{\lambda_t}(H^2)}.
\end{equation}
On the other hand, $\lambda_0=0$ implies
$F$ must satisfy $x-ra_0=0$ 
for a sufficiently small $s$.
Hence $\rk F=0$,
$F={\cal O}_C(k)$ and $H^{-1}(\coker \phi)=0$
for a small $s$.
Then $\phi_{(\beta_s,\omega_s)}(F) \leq 
\phi_{(\beta_s,\omega_s)}({\cal O}_C(l-2))$
for any $s$ and 
$\phi_{(\beta_s,\omega_s)}({\cal O}_C(l-2))<
\phi_{(\beta_s,\omega_s)}(E)$ for a small $s$.   
Hence $E$ is not semi-stable.
Therefore $\Hom(E,{\cal O}_C(l-2))=0$.

For a morphism
$\phi:{\cal O}_C(l) \to E$,
let $F$ be the image of $\phi$ in ${\cal A}_{(\beta_s,\omega_s)}$.
Then $F \in \Coh(X)$ and ${\cal O}_C(l) \to F$ is surjective in
$\Coh(X)$.
Since the torsion part of $E$ is purely 1-dimensional,
$F={\cal O}_C(l)$.
Then $\phi_{(\beta_s,\omega_s)}(F)>
\phi_{(\beta_s,\omega_s)}(E)$ for a small $s$.
Therefore $\Hom({\cal O}_C(l),E)=0$. 
\end{proof}

\begin{lem}\label{lem:Z=0}
Let $E$ be a $\sigma_{(\beta,\omega)}$-semi-stable object
with $(\beta,\omega) \in \xi^{-1}(u) \cap {\cal K}(X)$
such that $Z_{(\beta_0,tH)}(E)=0$ and
$(u,v(E))=0$.
Then $E \cong {\cal O}_C(l-1)^{\oplus n}$.
Moreover ${\cal O}_C(l-1)$ is $\sigma_{(\beta,\omega)}$-stable. 
\end{lem}

\begin{proof}
By Lemma \ref{lem:deg-H} and $Z_{(\beta_0,tH)}(E)=0$,
we have $(c_1(H^{-1}(E)(-\beta_0)),H)=
(c_1(H^0(E)(-\beta_0)),H)=0$ and
they are $\mu$-semi-stable with respect to $H$.
If $H^{-1}(E) \ne 0$, then
$Z_{(\beta_0,tH)}(H^{-1}(E)[1]) \in {\Bbb R}_{<0}$,
which implies $Z_{(\beta_0,tH)}(H^0(E)) \in {\Bbb R}_{>0}$.
Then $E$ is not $\sigma_{(\beta,\omega)}$-semi-stable
in a neighborhood of $(\beta_0,tH)$.
Hence $H^{-1}(E)=0$.
For $H^0(E)$, let $T$ be the torsion subsheaf of
$H^0(E)$ in $\Coh(X)$ and $F=H^0(E)/T$.
If $F \ne 0$, then
$Z_{(\beta_0,tH)}(F) \in {\Bbb R}_{>0}$, which implies
$Z_{(\beta_0,tH)}(T) \in {\Bbb R}_{<0}$.
Since $T$ is a subobject of $E$ in ${\cal A}_{(\beta,\omega)}$
for all $(\beta,\omega) \in {\cal K}(X)$,
$E$ is not $\sigma_{(\beta,\omega)}$-semi-stable
in a neighborhood of $(\beta_0,tH)$.
Hence $F=0$.
$\sigma_{(\beta,\omega)}$-semi-stability
of $E$ implies $E$ is a $(\beta_0-\frac{1}{2}K_X)$-twisted 
semi-stable sheaf with respect to $\omega$.
Since $Z_{(\beta_0,tH)}(E)=0$,
$c_1(E)=nC$, $n \in {\Bbb Z}_{>0}$ and
$\ch_2(E)=(\beta_0,c_1(E))$.
By Lemma \ref{lem:O_C},
$E$ is a successive extension of ${\cal O}_C(a)$.
Since $\Ext^1({\cal O}_C,{\cal O}_C)=0$,
$E \cong {\cal O}_C(a)^{\oplus n}$.
Since $\ch_2(E)=(\beta_0,c_1(E))=n(l-\frac{1}{2})$,
we have $E={\cal O}_C(l-1)^{\oplus n}$.  

If ${\cal O}_C(l-1)$ is not $\sigma_{(\beta,\omega)}$-stable,
then we have a non-trivial subobject $F_1$ of ${\cal O}_C(l-1)$
with $\phi_{(\beta,\omega)}(F_1)=\phi_{(\beta,\omega)}({\cal O}_C(l-1))$. 
Since $1>Z_{(\beta,\omega)}(F_1)/Z_{(\beta,\omega)}({\cal O}_{C}(l-1))>0$
for all $(\beta,\omega) \in \xi^{-1}(u) \cap {\cal K}(X)$,
we have $Z_{(\beta_0,tH)}(F_1)=0$.
Applying the first part of our lemma,
$F_1={\cal O}_C(l-1)^{\oplus m}$.
Therefore ${\cal O}_C(l-1)$ is 
$\sigma_{(\beta,\omega)}$-stable.
\end{proof}

\begin{prop}\label{prop:wall-H}
Assume that $u \in v({\cal O}_C(l-1))^\perp$.
If $E$ is a properly $\sigma_{(\beta,\omega)}$-semi-stable
object with $v(E)=v$ for $(\beta,\omega) \in \xi^{-1}(u)$, then
${\cal O}_C(l-1)$ is $\sigma_{(\beta,\omega)}$-stable
and $E$ is $S$-equivalent to
${\cal O}_C(l-1) \oplus E'$. 
\end{prop}

\begin{proof}
Assume that there is an exact sequence
\begin{equation}
0 \to E_1 \to E \to E_2 \to 0
\end{equation}
such that
$E_1$ and $E_2$ are $\sigma_{(\beta,\omega)}$-semi-stable
with $u \in v(E_1)^\perp$
and $v(E)=v$.

Then the same property hold for any
$(\beta,\omega) \in \xi^{-1}(u) \cap {\cal K}(X)$.
Then $(c_1(E_i(-\beta)),\omega) \geq 0$ for $i=1,2$ and
any $(\beta,\omega) \in \xi^{-1}(u)$.
In particular,
$(c_1(E_i(-\beta_0)),H) \geq 0$ for $i=1,2$.
Since $(c_1(E(-\beta_0)),H)=d_{\min} (H^2)$,
$(c_1(E_1(-\beta_0)),H)=0$ or
$(c_1(E_2(-\beta_0)),H)=0$.
By $(v(E_i),u)=0$,
we have
$Z_{(\beta_0,tH)}(E_1)=0$ or  
$Z_{(\beta_0,tH)}(E_2)=0$.
By Lemma \ref{lem:Z=0},
$E_1={\cal O}_C(l-1)^{\oplus n}$ or
$E_2={\cal O}_C(l-1)^{\oplus n}$.
\end{proof}

\begin{cor}
If $E$ is $\sigma_{(\beta,\omega)}$-semi-stable
for $(\beta,\omega) \in \xi^{-1}(u)$
and $\Hom(E,{\cal O}_C(l-1))=\Hom({\cal O}_C(l-1),E)=0$,
then $E$ is $\sigma_{(\beta,\omega)}$-stable.
\end{cor}

Proposition \ref{prop:wall-H} says that walls in ${\cal K}(X)$ gives
a wall for semi-stability of perverse coherent sheaves.
Conversely if there is a wall for semi-stability of perverse 
coherent sheaves, then it gives a wall in ${\cal K}(X)$. 
We set
$\beta_\pm :=\beta_0 \pm \epsilon C$.
If
there is a $(\beta_--\frac{1}{2}K_X)$-twisted semi-stable
object $E'$ with $v(E')=v-v({\cal O}_C(l-1))$ and
$\chi(E',{\cal O}_C(l-1))<0$,
there is a non-trivial extension
\begin{equation}
0 \to {\cal O}_C(l-1) \to E \to E' \to 0
\end{equation}
which gives a $(\beta_--\frac{1}{2}K_X)$-twisted semi-stable
object with $v(E)=v$.
Obviously $E$ is not $(\beta_+-\frac{1}{2}K_X)$-semi-stable.
Then $E$ is properly $\sigma_{(\beta,\omega)}$-semi-stable.

\end{NB}

For $(s,q) \in {\Bbb R}^2$ with $|s|,|q| \ll 1$, we consider
\begin{equation}\label{eq:sq-plane}
\begin{split}
\xi(\delta+sH+pC,H-qC)=&
e^\delta \left(
\frac{(1+s^2)(H^2)+p^2-q^2+2psq+\frac{(v^2)}{r^2}}{2}H \right.\\
& \left.
 +\frac{q^3+q(s^2-1)(H^2)+p^2 q+2ps (H^2)-\frac{(v^2)}{r^2}q}{2}C
+(s(H^2)+pq)(1+\frac{(v^2)}{2r^2}\varrho_X)
\right).
\end{split}
\end{equation}
If$(s,q)=(0,0)$, then
$\xi(\delta+sH+pC,H-qC)=e^\delta H \in C^+(v)$.
In this subsection,
we shall classify walls in a neighborhood of $(s,q)=(0,0)$.
We set 
\begin{equation}
L:=\{(\delta+sH+pC,H-qC) \mid (s,q) \in {\Bbb R}^2 \}.
\end{equation}
\begin{NB}
By Lemma \ref{lem:perv2:non-empty},
we may assume that
$\beta=\delta+sH+pC$ satisfies
$0<(\beta-\frac{1}{2}K_X,C)<1$ and
$-(\delta,C) \leq 0$, if $q=0$.
Then $p-\frac{1}{2}<(\delta,C)<p+\frac{1}{2}$
and $(\delta,C) \geq 0$. 
We may normalize $0 \leq (\delta,C) < 1$.
If $0<(\delta,C) < 1$, then
for $p=\frac{1}{2}$, 
$p-\frac{1}{2}<(\delta,C)<p+\frac{1}{2}$ holds.
If $(\delta,C)=0$, then
we may take $p=0$.
\end{NB}
\begin{NB}
Since $p^2+2psq=(p+sq)^2-(sq)^2$,
if $q^2(1+s^2)<(H^2)+\frac{(v^2)}{r^2}$, then
the coefficient of $H$ is positive.
\end{NB}
\begin{NB}
We need to assume
$(\delta+sH+pC,C)=(\delta,C)-p \not \in \frac{1}{2}+{\Bbb Z}$, if $q=0$.
\end{NB}
We set $\epsilon:=s(H^2)+pq$ and assume that
$\epsilon \leq 0$.
We have
\begin{equation}
\xi(\delta+sH+pC,H-qC)
=
e^\delta(H+xC+y(1+\frac{(v^2)}{2r^2}\varrho_X)),
\end{equation}
where 
\begin{equation}\label{eq:def-xy}
\begin{split}
x:=& \frac{q^3+q(s^2-1)(H^2)+p^2 q+2ps (H^2)-\frac{(v^2)}{r^2}q}
{(1+s^2)(H^2)+p^2-q^2+2psq+\frac{(v^2)}{r^2}},\\
y:=& 
\frac{2\epsilon}
{(1+s^2)(H^2)+p^2-q^2+2psq+\frac{(v^2)}{r^2}}.
\end{split}
\end{equation}
Then we have expansions
\begin{equation}\label{eq:expansion1}
\begin{split}
x=& \frac{-((H^2)+p^2+\frac{(v^2)}{r^2})q+2p\epsilon}
{(H^2)+p^2+\frac{(v^2)}{r^2}}+O_2(\epsilon,q),\\
y=& 
\epsilon \left(
\frac{2}
{(H^2)+p^2+\frac{(v^2)}{r^2}}+O_1(\epsilon,q) \right),
\end{split}
\end{equation}
where $O_n(\epsilon,q)$ is a power series of 
$\epsilon$ and $q$ contained in the ideal $(\epsilon, q)^n$.  
If $q=0$, then
\begin{equation}
\begin{split}
x=& \frac{2ps (H^2)}
{(1+s^2)(H^2)+p^2+\frac{(v^2)}{r^2}}=\frac{2s(H^2)}{p}+O_2(p^{-1}),\\
y=& 
\frac{2\epsilon}
{(1+s^2)(H^2)+p^2+\frac{(v^2)}{r^2}}=\frac{2\epsilon}{p^2}+O_3(p^{-1}),\\
x=&py.
\end{split}
\end{equation}
If $p \gg 0$, then $(x,y)$ is close to $y=0$.

We set
$\beta_0:=\delta+p_0 C$, and assume that
$-\frac{1}{2}+l<(\beta_0,C) <\frac{1}{2}+l$ $(l \in {\Bbb Z})$.
\begin{NB}
If $l=0$ and $0 \leq (\delta,C)<1$, then 
\begin{equation}
-(\beta_0-\delta,C)=
\begin{cases}
0, & 0 \leq (\delta,C)<\frac{1}{2}\\
\frac{1}{2}, & \frac{1}{2} \leq (\delta,C)<1
\end{cases}
\end{equation}
\end{NB}
By \eqref{eq:expansion1},
we can take a neighborhood $U$ of $(\beta_0,H)$ in $L$ such that
$\overline{U}$ is compact and $\xi:U \to \xi(U)$ is isomorphic,
where $p:=p_0$. 
In particular $\xi(U)$ is a neighborhood of $\xi(\beta_0,H)=e^\delta H$.
\begin{NB}
If $\beta_0=\delta+p_0 C$,
then $x=-q+px+(\epsilon,q)^2$ implies
$x-p_0 y \leq 0$ for $q \geq 0$.
If $p \to \infty$, then
stability is close to Gieseker stability with respect to $H-qC$.
So it is better to consider small $p_0$.
If $p_0$ is sufficiently small, then
${\cal M}_H^{\beta_0-\frac{1}{2}K_X}(v)=\emptyset$.
\end{NB}
By shrinking $U$, we may assume that 
there are finitely many Mukai vectors defining walls
 in $U$ and all walls passes
$(\beta_0,H)$.
For each Mukai vector $v_1$ 
defining a wall,
we may also assume that all walls in $U$ with respect to
$v_1$ passes $(\beta_0,H)$.
\begin{NB}
All walls are of the form
$\{(\beta,\omega) \in U \mid \phi_{\beta,\omega}(E)=
\phi_{\beta,\omega}(E_1) \}$.
\end{NB}
We set
\begin{equation}
\begin{split}
U^{\leq 0}:=& U \cap \{(\delta+sH+pC,H-qC) \mid  \epsilon \leq 0 \},\\
U^{< 0}:=& U \cap \{(\delta+sH+pC,H-qC) \mid  \epsilon < 0 \}.\\
\end{split}
\end{equation}
Let $E_1$ be a $\sigma_{(\beta,\omega)}$-stable object defining a wall
in $U^{\leq 0}$. Since there is no wall with respect to $v(E_1)$ 
between $(\beta,\omega)$ and $(\beta_0,H)$,
$E_1$ is $\sigma_{(\beta_0,H)}$-semi-stable.
Since $(\beta_0-\delta,H)=0$,
$E_1$ is generated by $\mu$-semi-stable objects
and objects ${\cal O}_C(l)[-1],
{\cal O}_C(l-1),k_x \in {\frak C}^l[-1]$,
$(x \in X \setminus C)$.
\begin{NB}
If $\rk E_1>0$, then
$E_1$ is twisted stable.
\end{NB}

Since $\varrho_X^\perp \cap e^\delta({\Bbb R}H+{\Bbb R}C
+{\Bbb R}(1+\frac{(v^2)}{2r^2}\varrho_X))$ is $y=0$, 
if $|\epsilon| \ll q$, then $\sigma_{(\beta,\omega)}$-twisted
stability coincides with Gieseker semi-stability with respect to
$H-qC$, where 
$(\beta,\omega)=(\delta+sH+p_0 C,H-qC)$.

\begin{lem}\label{lem:mu-ss:wall}
If $\rk E_1 \ne 0$, then
$E_1$ is a $\mu$-semi-stable object of ${\frak C}^l$. 
\end{lem}

\begin{proof}
For $(\beta,\omega) \in U^{<0}$,
$E_1$ is $\sigma_{(\beta,\omega)}$-stable with
$\phi_{(\beta,\omega)}(E_1) \in (0,1]$.
Hence $H^i(E_1)=0$ for $i \ne -1,0$.
Since $E_1$ is $\sigma_{(\beta_0,H)}$-semi-stable with
$\phi_{(\beta_0,H)}(E_1)=0$,
$H^i(E_1)=0$ for $i \ne 0,1$.
Therefore $E_1 \in {\frak C}^l$.
Thus $E_1$ is a $\mu$-semi-stable object of
${\frak C}^l$.
\begin{NB}
Assume that $H^{-1}(E_1) \ne 0$.
$\phi_{(\beta,\omega)}(H^{-1}(E_1)[1]) \in (0,1]$
and $\phi_{(\beta,\omega)}(H^{-1}(E_1)[1]) <
\phi_{(\beta,\omega)}(E_1)$
imply $\phi_{(\beta_0,H)}(H^{-1}(E_1)[1])=0$.
Since $H^{-1}(E_1)$ is a $\mu$-semi-stable object
of twisted degree 0,
$\phi_{(\beta_0,H)}(H^{-1}(E_1)[1])=1$, which is a contradiction.
\end{NB}
\end{proof}
All walls in $\xi^{-1}(\xi(U)) \cap {\cal K}(X)$ are defined by an object 
$E_1$ which defines a wall in $U$.
\begin{NB}
\begin{lem}
For any ${\cal O}_C(a)$,
there is a neighborhood $V_a$ of $e^\delta H$ such
that ${\cal O}_C(a)$ is $\sigma_{(\beta,\omega)}$-stable
for $(\beta,\omega) \in \xi^{-1}(V_a)$ with $\omega \in \Amp(X)_{\Bbb R}$.
\end{lem}
\begin{proof}
We take $(\beta,H)$ such that 
$(\beta-\delta,H)=0$ and $-\frac{1}{2}+a<(\beta,C)<a+\frac{1}{2}$.
Since ${\cal O}_C(a)$ is an irreducible object of ${\frak C}^a$,
${\cal O}_C(a)$ is a $\sigma_{(\beta,\omega)}$-stable object.
We take an open neighborhood $U$ of $(\beta,\omega)$ such
that ${\cal O}_C(a)$ is stable on $U$.
Let $V_a$ be an open neighborhood of $e^\delta H$ such tha
$V_a \subset \xi(U)$.
Then ${\cal O}_C(a)$ is $\sigma_{(\beta,\omega)}$-stable
for $(\beta,\omega) \in \xi^{-1}(V_a)$ with $\omega \in \Amp(X)_{\Bbb R}$. 
\end{proof}
\end{NB}
\begin{NB}
We set $(y, x)=(\lambda \cos \theta,\lambda \sin \theta)$.
Then $\epsilon=\frac{(H^2)+p^2+\frac{(v^2)}{r^2}}{2}
r \cos \theta(1+r f(r,\theta))$,
where $f$ is $C^\infty$ in a neighborhood of $r=0$.
So $|f(r, \theta)|<N_1$ for some $N_1$. 
If $r<1/N_1$, then 
$\epsilon  \leq 0$.
Assume that $\tan \varphi=-p$.
Then $q=py-x=-\frac{r}{\cos \varphi}(\sin (\theta+\varphi)
+rg(r,\theta))$. 
Similarly if $\sin (\theta+\varphi)<-\frac{1}{N_2}$, then
there is a large number $N_3$ such that 
$q=py-x>0$ for $r<1/N_3$. 
\end{NB}
\begin{NB}
We take $p \in {\Bbb R}$ such that
$a+(\delta,C)-\tfrac{3}{2}<p<a+(\delta,C)-\tfrac{1}{2}$.
Then ${\cal O}_C(-a)[1]$ is an irreducible object 
of ${\frak C}^{-a+1}$.
Let $V_a$ be the open set defined by
$x<(a+(\delta,C)-\tfrac{3}{2})y$ and $y<0$.
Then there is $p_0 <p$ such that
$a+(\delta,C)-\tfrac{3}{2}<p_0 < p<a+(\delta,C)-\tfrac{1}{2}$.
We set $\beta_0:=\delta+p_0 C$.
In a neighborhood $U$ of 
$(\beta_0,H)$,
we may assume that $\xi$ induces an isomorphism
$\xi:U \to \xi(U)$, where $\xi(U)$ is an open neighborhood
of $\xi(\beta_0,H)=e^\delta H$.
By shrinking $U$, we may assume that
$$
\{\xi(\beta_0+sH,H-qC) \mid \epsilon \leq 0, q>0\}
\supset \{e^\delta(H+xC+y(1+\tfrac{(v^2)}{2r^2}\varrho_X))
\mid x< p y, y \leq 0\}.
$$
Since ${\cal O}_C(-a)[1]$ is $\sigma_{(\beta_0,H)}$-sable,
there is an open neighborhood $U_p \subset U$ of $(\beta_0,H)$
such that ${\cal O}_C(-a)$ is
a $\sigma_{(\beta_0+sH,H-qC)}$-stable object
for $(\beta_0+sH,H-qC) \in U_p$.
\end{NB}

\subsubsection{Stability of ${\cal O}_C(a)$}

For $p_1 > p_0=(\delta-\beta_0,C)$, by shrinking $U$, we may assume that
\begin{equation}
\{\xi(\beta_0+sH,H-qC) \mid \epsilon \leq 0, q>0\}
\supset \{e^\delta(H+xC+y(1+\tfrac{(v^2)}{2r^2}\varrho_X)) \in\xi(U)
\mid x< p_1 y, y \leq 0 \},
\end{equation}
where $\epsilon=s(H^2)+p_0 q$ and
we also use
$$
x=-q+p_0 y+O_2(\epsilon,q).
$$
For $a \in {\Bbb Z}$ with
$p_1<a+(\delta,C)-\tfrac{1}{2}$,
we take
$p:=p(a) \in {\Bbb R}$ such that
$a+(\delta,C)-\tfrac{3}{2}<p(a)<a+(\delta,C)-\tfrac{1}{2}$.
We set $\beta_p:=\delta+p(a) C$.
Let $V_a$ be the open set defined by
$x<(a+(\delta,C)-\tfrac{3}{2})y$ and $y<0$.
\begin{NB}
In a neighborhood $U$ of 
$(\beta_0,H)$,
we may assume that $\xi$ induces an isomorphism
$\xi:U \to \xi(U)$, where $\xi(U)$ is an open neighborhood
of $\xi(\beta_0,H)=e^\delta H$.
\end{NB}
Since $(-a+1)-\frac{1}{2}<(\beta_p,C)<(-a+1)+\frac{1}{2}$ and
${\cal O}_C(-a)[1]$ is an irreducible object 
of ${\frak C}^{-a+1}$,
${\cal O}_C(-a)[1]$ is $\sigma_{(\beta_p,H)}$-stable.
Hence there is an open neighborhood $U_p$ of $(\beta_p,H)$
such that ${\cal O}_C(-a)$ is
a $\sigma_{(\beta_p+sH,H-qC)}$-stable object
for $(\beta_p+sH,H-qC) \in U_p$.
By shrinking $U_p$, we may assume that
$$
\ell_a:=\{e^\delta(H+xC+y(1+\tfrac{(v^2)}{2r^2}\varrho_X)) \in \xi(U_p)
\mid x=(a+(\delta,C)-\tfrac{1}{2})y, y <0 \}
$$ 
is a subset of 
$$
\{\xi(\beta_p+sH,H-qC) \mid 0 > \epsilon_p:=s(H^2)+p(a)q, q>0\}.
$$
We note that $v({\cal O}_C(-a))^\perp$ is the line
$x=(a+(\delta,C)-\frac{1}{2})y$.
By Remark \ref{rem:fiber-property2},
$\xi_v^{-1}(\ell_a)=\xi_{v({\cal O}_C(-a))}^{-1}(\ell_a)$.
Applying Lemma \ref{lem:xi-wall}
to $\xi_{v({\cal O}_C(-a))}^{-1}(\ell_a)$,
${\cal O}_C(-a)$ is $\sigma_{(\beta_0+sH,H-qC)}$-stable,
if $(\beta_0+sH,H-qC) \in U$ $(q>0)$ belongs to $\xi^{-1}(\ell_a)$. 
Thus we obtain the following.
\begin{lem}\label{lem:C-stable}
For any $a$ with
$p_1<a+(\delta,C)-\tfrac{1}{2}$,
there is a neighborhood $U_a$ 
of $\xi^{-1}(\ell_a) \cap U$ such that  
${\cal O}_C(-a)$ is $\sigma_{(\beta_0+sH,H-qC)}$-stable.
Moreover 
$$
\phi_{(\beta_0+sH,H-qC)}({\cal O}_C(-a))=
\phi_{(\beta_0+sH,H-qC)}(E),\;\;E \in {\cal M}_{(\beta_0+sH,H-qC)}(v)
$$
if $(\beta_0+sH,H-qC) \in \xi^{-1}(\ell_a) \cap U$.
\end{lem}

\begin{NB}
If $\omega \in \pi^*(\Amp(Y))_{\Bbb R}$, then
$\xi^{-1}(\xi(\beta,\omega))$ is not connected.
\end{NB}

\begin{NB}
If $\omega$ is not ample, then
${\cal O}_C(a)$ is not stable in general.
For example,
${\cal O}_C(1)$ is not semi-stable if $\omega=H+qC$ ($p>0$).
$$
0 \to {\cal O}_C^{\oplus 2} \to {\cal O}_C(1) \to {\cal O}_C(-1)[1]
\to 0
$$
and
$$
0 \to {\cal O}_C[-1] \to {\cal O}_C(-1)[1] \to 
{\bf L}\pi^*({\cal O}_y) \to 0, 
$$
$y \in Y$.
Hence ${\cal O}_C(-1)[1] \in {\cal A}_{(\delta+sH+pC,H+qC)}$
and ${\cal O}_C \in {\cal A}_{(\delta+sH+pC,H+qC)}[1]$.
\end{NB}

\subsection{A description of ${\cal M}_{(\delta+sH+pC,H-qC)}(v)$.}
For $(\beta,\omega) \in U^{< 0}$ with $q \geq 0$,
we shall describe ${\cal M}_{(\beta,\omega)}(v)$
in terms of the moduli spaces of semi-stable perverse coherent
sheaves ${\cal M}_H^{\gamma}(v)$.
We set 
$$
u:=e^\delta(H+xC+y(1+\tfrac{(v^2)}{2r^2}\varrho_X)).
$$
We first assume that $u \in \xi(U)$ satisfies
 $(u,v({\cal O}_C(-a)))
=-x-y(\frac{1}{2}-a-(\delta,C)) \ne 0$ for all $a$.
Under this condition, we shall describe ${\cal M}_{(\beta,\omega)}(v)$
with $\xi(\beta,\omega)=u$, where
$(\beta,\omega)=(\delta+sH+pC,H-qC)$.
There are $s'<0, t'>0$ such that
$\xi(\delta+s' H+p' C,t' H)=u$,
where
\begin{equation}
p':=\frac{x}{y}=
\frac{q^3+q(s^2-1)(H^2)+p^2 q+2ps (H^2)-\frac{(v^2)}{r^2}q}{2\epsilon}.
\end{equation}
Then $\{\xi(\delta+s H+p' C,t H) \mid s \leq 0, t >0 \}$
is a segment.  
\begin{NB}
Let $u_s$ ($ 0 \leq s \leq 1$) be a segment such that
$u_0=e^\delta(H+xC+y(1+\frac{(v^2)}{2r^2}\varrho_X))$
and $u_1=e^\delta H$.
We can take a family of stability conditions
$\sigma_{(\beta_s,\omega_s)}$ such that
$\xi(\beta_s,\omega_s)=u_s$ and $\omega_s \in {\Bbb R}_{>0}H$.
\end{NB}
Since there is no wall between $u$ and $e^\delta H$,
$$
{\cal M}_{(\beta,\omega)}(v)={\cal M}_{(\delta+s' H+p' C,t_\infty H)}(v)=
{\cal M}_H^{\delta+p' C-\frac{1}{2}K_X}(v),
$$
 where 
$t_\infty$ is sufficiently large.
If $(\beta,\omega)$ is not general,
then the wall is defined by
\begin{equation}
\frac{\chi(E(-(\delta+\frac{x}{y}C-\frac{1}{2}K_X)))}{\rk E}
=\frac{\chi(E_1(-(\delta+\frac{x}{y}C-\frac{1}{2}K_X)))}{\rk E_1},
\end{equation} 
where $E_1$ is a $(\delta+p' C-\frac{1}{2}K_X)$-twisted stable object
with $(c_1(E_1)-\rk E_1 \delta,H)=0$.

\begin{NB}
Then there is $l \in {\Bbb Z}$ 
with $-\frac{1}{2}+l<(\beta',C)<\frac{1}{2}+l$.
$u_s$ belongs to a wall if and only if 
there is a 
$(\beta'-\frac{1}{2}K_X)$-twisted stable object $E_1$ of ${\frak C}^l$
such that $\chi(E(-\beta'))/\rk E=\chi(E_1(-\beta'))/\rk E_1$
and $0<\rk E_1 <\rk E$.
\end{NB}

\begin{NB}
If $\xi(\delta+sH+pC,tH)=
e^\delta(H+xC+y(1+\frac{(v^2)}{2r^2}\varrho_X))$, then
$x/y=p$ and
$\beta'=\delta+pC$. 
\end{NB}

We next consider the case where $(\beta,C)=l-\frac{1}{2}$.
We set $\beta_\pm:=\beta \mp \eta C$ for a sufficiently small $\eta>0$.

\begin{defn}\label{defn:exceptional-case2}
For 
$E \in {\cal A}_{(\beta_-,tH)}$ with 
$Z_{(\beta,tH)}(E) \in {\Bbb H}$,
$E$ is  
$\sigma_{(\beta,tH)}$-semi-stable,
if 
$$
Z_{(\beta,tH)}(E_1)/Z_{(\beta,tH)}(E) \in -{\Bbb H} \cup {\Bbb R}_{\geq 0}
$$
for all non-zero subobject $E_1$ of $E$ in ${\cal A}_{(\beta_-,tH)}$.
Let ${\cal M}_{(\beta,tH)}(v)$ be the moduli stack of
$\sigma_{(\beta,tH)}$-semi-stable objects $E$ with $v(E)=v$.
\end{defn}
If $v=nv({\cal O}_C(l-1))$, then we also
set ${\cal M}_{(\beta,tH)}(v):=
{\cal M}_{(\beta_-,tH)}(v)=\{{\cal O}_C(l-1)^{\oplus n} \}$.
Then we have ${\cal M}_{(\beta_+,tH)}(v)={\cal M}_{(\beta,tH)}(v)$.
\begin{NB}
${\cal M}_H^{\beta_+-\frac{1}{2}K_X}(v)=\{{\cal O}_C(l-1)^{\oplus n} \}
 \subset {\frak C}^{\beta_+}[-1]$.
\end{NB}

\begin{prop}\label{prop:exceptional-case}
${\cal M}_{(\gamma,\omega)}(v)={\cal M}_{(\beta,tH)}(v)=
{\cal M}_H^{\beta-\frac{1}{2}K_X}(v)$
for $(\gamma,\omega) \in \xi^{-1}(\xi(\beta,tH))$.
\end{prop}

\begin{proof}
For $E \in {\cal M}_{(\beta,tH)}(v)$,
we take a filtration 
\begin{equation}\label{eq:JHF-H}
0 \subset F_1 \subset F_2 \subset \cdots \subset F_s=E
\end{equation}
such that $E_i:=F_i/F_{i-1}$ are $\sigma_{(\beta,tH)}$-stable
and $Z_{(\beta,tH)}(E_i)=\lambda_i Z_{(\beta,tH)}(E)$
with $0 \leq \lambda_i \leq 1$.
If $\lambda_i=0$, then $E_i={\cal O}_C(l-1)$.
\begin{NB}
Since $d_\beta(E)>0$ and $d_\beta(\bullet)$ is a discrete invariant 
for $\beta \in \NS(X)_{\Bbb Q}$,
there is a filtration \eqref{eq:JHF-H} such that
$E_i:=F_i/F_{i-1}$ does not have a subobject 
$E'$ with $Z_{(\beta,tH)}(E')=\lambda Z_{(\beta,tH)}(E_i)$
and $0<\lambda<1$.
Then applying Corollary \ref{cor:irreducible},
we get a desired decomposition.
\end{NB}
We set $v_i:=v(E_i)$.
For the proof of our claim, 
it is sufficient to prove that
\begin{equation}\label{eq:wall-stability}
{\cal M}_{(\gamma,\omega)}(v_i)^s={\cal M}_{(\beta,tH)}(v_i)^s=
{\cal M}_H^{\beta-\frac{1}{2}K_X}(v_i)^s
\end{equation}
for $(\gamma,\omega) \in \xi^{-1}(\xi(\beta,tH))$.

We first assume that $\lambda_i \ne 0$.
As in Remark \ref{rem:fiber-property2},
we set $\xi_{v_i}(\gamma,\omega):=
\mathrm{Im}(Z_{(\gamma,\omega)}(v_i)^{-1}e^{\gamma+\sqrt{-1}\omega})$.
Since $\mathrm{Im}(Z_{(\beta,tH)}(v_i)^{-1}e^{\beta+\sqrt{-1}tH})
=\lambda_i^{-1}
\mathrm{Im}(Z_{(\beta,tH)}(v)^{-1}e^{\beta+\sqrt{-1}tH})$,
 we get
$\xi_{v_i}(\beta,tH) \in {\Bbb R} \xi(\beta,tH)$.
\begin{NB}
By Remark \ref{rem:fiber-property},
$Z_{(\gamma,\omega)}(x) \in {\Bbb R}Z_{(\gamma,\omega)}(v_i)$
if and only if 
$x \in \xi_{v_i}(\gamma,\omega)^\perp$.
\end{NB}
By Remark \ref{rem:fiber-property2},
$\xi_{v_i}^{-1}(\xi_{v_i}(\beta,tH))=
\xi^{-1}(\xi(\beta,tH))$.
We take $(\gamma_\pm,\omega_\pm) \in 
\xi_{v_i}^{-1}(\xi_{v_i}(\beta_\pm,tH))$
which are in a neighborhood of $(\gamma,\omega)$.
Then
${\cal M}_{(\beta_\pm,tH)}(v_i)^s=
{\cal M}_{(\gamma_\pm,\omega_\pm)}(v_i)^s$.
Since 
$$
{\cal M}_{(\gamma_-,\omega_-)}(v_i)^s 
\cap {\cal M}_{(\gamma_+,\omega_+)}(v_i)^s=
{\cal M}_{(\gamma,\omega)}(v_i)^s,
$$
$E_i$ is $\sigma_{(\gamma,\omega)}$-stable
with $Z_{(\gamma,\omega)}(E_i) \in {\Bbb R}_{>0}Z_{(\gamma,\omega)}(v)$.
If $E_i={\cal O}_C(l-1)$, then
it is also $\sigma_{(\gamma,\omega)}$-stable
with $Z_{(\gamma,\omega)}(E_i) \in {\Bbb R}_{>0}Z_{(\gamma,\omega)}(v)$.
Indeed we also have ${\cal M}_{(\gamma_\pm,\omega_\pm)}(v_i)
={\cal M}_{(\beta_\pm,tH)}(v_i)=
\{ {\cal O}_C(l-1) \}$
and
${\cal M}_{(\gamma,\omega)}(v_i)
={\cal M}_{(\beta,tH)}(v_i)=
\{ {\cal O}_C(l-1) \}$.
Hence $E$ is $\sigma_{(\gamma,\omega)}$-semi-stable with a
Jordan-H\"{o}lder filtration \eqref{eq:JHF-H}.
Since $\omega$ is ample, by our choice of $U$,
$e^\beta H \in v_i^\perp$ and all
walls for $v_i$ passes $e^\delta H$. 
Hence
$$
v_i=e^\beta(r_i+d_i H+D_i +a_i \varrho_X),\;r_i=r \frac{d_i}{d},
a_i=a \frac{d_i}{d}
$$
and
$$
{\cal M}_H^{\beta_\pm-\frac{1}{2}K_X}(v_i)^s
={\cal M}_{(\beta_\pm,tH)}(v_i)^s=
{\cal M}_{(\gamma_\pm,\omega_\pm)}(v_i)^s.
$$
By Lemma \ref{lem:stable-on-wall},
we have 
$$
{\cal M}_H^{\beta-\frac{1}{2}K_X}(v_i)^s=
{\cal M}_H^{\beta_--\frac{1}{2}K_X}(v_i)^s 
\cap {\cal M}_H^{\beta_+-\frac{1}{2}K_X}(v_i)^s.
$$
Hence \eqref{eq:wall-stability} holds.

\begin{NB}
Let $v=\sum_i v_i$ be a decomposition of $v$ such that
$$
v_i=e^\beta(r_i+d_i H+D_i +a_i \varrho_X),\;d_i=r_i \frac{d}{r},
a_i=r_i \frac{a}{r}.
$$
We note that $v_i=e^\beta D_i$, if $r_i=0$.
We first prove that
\begin{equation}\label{eq:wall-stability}
{\cal M}_{(\gamma,\omega)}(v_i)^s={\cal M}_{(\beta,tH)}(v_i)^s=
{\cal M}_H^{\beta-\frac{1}{2}K_X}(v_i)^s
\end{equation}
for $(\gamma,\omega) \in \xi^{-1}(\xi(\beta,tH))$.

As in Remark \ref{rem:fiber-property2},
we set $\xi_{v_i}(\gamma,\omega):=
\mathrm{Im}(Z_{(\gamma,\omega)}(v_i)^{-1}e^{\gamma+\sqrt{-1}\omega})$.
By our assumption, we see that
$\xi_{v_i}(\beta,tH) \in {\Bbb R} \xi(\beta,tH)$.
\begin{NB2}
By Remark \ref{rem:fiber-property},
$Z_{(\gamma,\omega)}(x) \in {\Bbb R}Z_{(\gamma,\omega)}(v_i)$
if and only if 
$x \in \xi_{v_i}(\gamma,\omega)^\perp$.
\end{NB2}
By Remark \ref{rem:fiber-property2},
$\xi_{v_i}^{-1}(\xi_{v_i}(\beta,tH))=
\xi^{-1}(\xi(\beta,tH))$.
We first assume that $\rk v_i \ne 0$.
By Lemma \ref{lem:stable-on-wall},
we have 
$$
{\cal M}_H^{\beta-\frac{1}{2}K_X}(v_i)^s=
{\cal M}_H^{\beta_--\frac{1}{2}K_X}(v_i)^s 
\cap {\cal M}_H^{\beta_+-\frac{1}{2}K_X}(v_i)^s.
$$
We take $(\gamma_\pm,\omega_\pm) \in 
\xi_{v_i}^{-1}(\xi_{v_i}(\beta_\pm,tH))$
which are in a neighborhood of $(\gamma,\omega)$.
Then
${\cal M}_H^{\beta_\pm-\frac{1}{2}K_X}(v_i)^s
={\cal M}_{(\beta_\pm,tH)}(v_i)^s=
{\cal M}_{(\gamma_\pm,\omega_\pm)}(v_i)^s$, where $t$ is sufficiently large.
Since 
$$
{\cal M}_{(\gamma_-,\omega_-)}(v_i)^s 
\cap {\cal M}_{(\gamma_+,\omega_+)}(v_i)^s=
{\cal M}_{(\gamma,\omega)}(v_i)^s,
$$
we have ${\cal M}_H^{\beta-\frac{1}{2}K_X}(v_i)^s=
{\cal M}_{(\gamma,\omega)}(v_i)^s$.

Let $E \in {\bf D}(X)$ be an object which is a successive 
extension of $E_i \in {\bf D}(X)$ with
$v(E_i)=v_i$.   
Then 
$E$ is $\beta$-twisted semi-stable if
and only if $E$ is $\sigma_{(\gamma,\beta)}$-semi-stable.  

If $\rk v_i=0$, then 
we also have ${\cal M}_{(\gamma_\pm,\omega_\pm)}(v_i)
={\cal M}_{(\beta_\pm,tH)}(v_i)=
\{ {\cal O}_C(l-1)^{\oplus n} \}$
where $n$ is deermined by $v_i=nv({\cal O}_C(l-1))$.
Hence 
${\cal M}_{(\gamma,\omega)}(v_i)
={\cal M}_{(\beta,tH)}(v_i)=
\{ {\cal O}_C(l-1)^{\oplus n} \}$, which implies that 
\eqref{eq:wall-stability} also holds for this case.
\end{NB}
\begin{NB}
If $r_i=0$, then we also have $M_{(\beta,H)}(v_i)=\{{\cal O}_C(l-1) \}$
by Lemma \ref{lem:C-stable}. 
\end{NB}
\end{proof}

\section{Wall crossing for Gieseker semi-stability}\label{sect:wall-crossing}

\subsection{Ample line bundles on the moduli spaces}\label{subsect:ample}

Assume that there is a moduli scheme 
$M_{(\beta,\omega)}(v)$ of $S$-equivalence classes of
semi-stable objects and consisting of stable objects.
For simplicity, we assume that there is a universal
family ${\cal E}_v$ 
on $M_{(\beta,\omega)}(v) \times X$.
For $\alpha \in K(X)_{\Bbb Q}$, we set
\begin{equation}
{\cal L}(\alpha):=
\det p_{M_{(\beta,\omega)}(v)!}(\alpha^{\vee} \otimes {\cal E}_v).
\end{equation}

We set
\begin{equation}
K(X)_v:=\{\alpha \in K(X)_{\Bbb Q} \mid  \chi(\alpha^{\vee} \otimes v)=0 \}. 
\end{equation}
We have a morphism
\begin{equation}
\begin{matrix}
\kappa:& K(X)_v & \to & \Pic(M_{(\beta,\omega)}(v))_{\Bbb Q}\\
& \alpha & \mapsto & {\cal L}(\alpha).
\end{matrix}
\end{equation}

As in \cite{univ}, we have a morphism
${\frak a}:M_{(\beta,\omega)}(v) \to \Alb(X) \times \Pic^0(X)$.
$\kappa$ induces a homomorphism
\begin{equation}
\theta_v:v^\perp \to \NS(M_{(\beta,\omega)}(v))_{\Bbb Q}
/{\frak a}^*(\NS(\Alb(X) \times \Pic^0(X))_{\Bbb Q}),
\end{equation}
where 
\begin{equation}
\begin{split}
v^\perp:= & \{ u \in H^*(X,{\Bbb Q})_{\alg} \mid 
\chi(u(\td_X^{-1})^{\vee},v)=0 \}\\
= & \{ u \in H^*(X,{\Bbb Q})_{\alg} \mid 
(u,v)=0 \}.
\end{split}
\end{equation}
\begin{NB}
$\chi(u(\td_X^{-1})^{\vee},v)=-(u,v)$.
\end{NB}
By works of Fogarty \cite{Fo} and 
Li \cite{Li}, \cite{Li3} (see also \cite{univ}), 
if $\rk v=1,2$ and $M_{(\beta,\omega)}(v)$ is the moduli of
Gieseker semi-stable sheaves $M_\omega^{\beta-\frac{1}{2}K_X}(v)$, then
${\frak a}$ is the Albanese morphism and 
$\theta_v$ is isomorphic for all sufficiently large $(v^2)$
depending on $\omega, \rk v,c_1(v)$. 

\begin{rem}
If there is a universal family as a twisted object, then 
$\kappa$ is well-defined.
\end{rem}

\begin{NB}

\begin{equation}
\begin{split}
\chi(e^\beta,v)=& -(e^\beta,v\td_X)\\
=&
-r(e^\beta,1+(\delta-\frac{1}{2}K_X)+(\frac{a}{r}-\frac{1}{2}(K_X,\delta)
+\chi({\cal O}_X))\varrho_X)\\
=& -r((\beta,\delta-\frac{1}{2}K_X)-
(\frac{a}{r}-\frac{1}{2}(K_X,\delta)
+\chi({\cal O}_X))-\frac{(\beta^2)}{2}).
\end{split}
\end{equation}
Hence 
\begin{equation}
e^\beta-\frac{\chi(e^\beta,v)}{r}\varrho_X=
1+\beta+(\beta,\delta-\frac{1}{2}K_X)-
(\frac{a}{r}-\frac{1}{2}(K_X,\delta)
+\chi({\cal O}_X))
\end{equation}

\begin{equation}
e^{\beta-\frac{1}{2}K_X}
-\frac{\chi(e^{\beta-\frac{1}{2}K_X},v)}{r}\varrho_X=
1+\left(\beta-\frac{1}{2}K_X \right)+
(\beta,\delta-\frac{1}{2}K_X)+\frac{1}{4}(K_X^2)
-(\frac{a}{r}
+\chi({\cal O}_X))
\end{equation}

\begin{equation}
e^\delta(1+\beta-\delta+\frac{(v^2)}{2r^2}\varrho_X)
=1+\beta+\left((\beta,\delta)-\frac{a}{r}\right)\varrho_X.
\end{equation}

\begin{equation}
\begin{split}
& \left(
1+\beta+\left((\beta,\delta)-\frac{a}{r}\right)\varrho_X
\right)
(\td_X^{-1})^{\vee}\\
=& 
1+\left(\beta-\frac{1}{2}K_X \right)+
\left((\beta,\delta)-\frac{a}{r}-(\beta,\frac{1}{2}K_X)+
\frac{1}{4}(K_X^2)-\chi({\cal O}_X) \right)\varrho_X
\end{split}
\end{equation}

Thus
\begin{equation}
\begin{split}
-\left(\left(
1+\beta+\left((\beta,\delta)-\frac{a}{r}\right)\varrho_X
\right),x \right)=&
\chi \left(\left(
1+\beta+\left((\beta,\delta)-\frac{a}{r}\right)\varrho_X
\right)
(\td_X^{-1})^{\vee},x \right)\\
=&
\chi \left(
e^{\beta-\frac{1}{2}K_X}
-\frac{\chi(e^{\beta-\frac{1}{2}K_X},v)}{r}\varrho_X,x \right).
\end{split}
\end{equation}

\end{NB}

\begin{NB}
\begin{lem}
Let $L_1$ be an ample divisor on $X$ and $L_2$ be a nef divisor
on $X$.
Then $tL_1+(1-t)L_2$ is ample for $0<t \leq 1$.
\end{lem}

\begin{proof}
There is an open neighborhood $U$ of $L_1$ in $\Amp(X)_{\Bbb R}$.
Then the cone spanned by $U$ and $L_2$ is contained in the nef cone.
Since every $tL_1+(1-t)L_2$ ($0<t \leq 1$) is in the interior
of this cone,
Kleiman's criterion implies the ampleness. 
\end{proof}
\end{NB}

We take $\xi_1,\xi_2 \in K(X)_{\Bbb Q}$ 
such that
\begin{equation}
\begin{split}
\ch(\xi_1)=& (H+(H,\delta)\varrho_X)(\td_X^{-1})^{\vee}
=H+(H,\delta-\tfrac{1}{2}K_X)\varrho_X,\\
\ch(\xi_2)=& -\left(e^{\beta-\frac{1}{2}K_X}
-\frac{\chi(e^{\beta-\frac{1}{2}K_X},v)}{r}\varrho_X \right).
\end{split}
\end{equation} 
By the construction of the moduli scheme, we have the following.
\begin{lem}[cf. {\cite[Lem. A.4]{MYY:2011:2}}]\label{lem:ample}
${\cal L}(n \xi_1+\xi_2)$ is ample for $n \gg 0$.
\end{lem}

We set
\begin{equation}
u:=e^\delta(\zeta-
(1+\tfrac{(v^2)}{2r^2}\varrho_X)).
\end{equation}
Assume that there is no wall between
$H+(H,\delta)\varrho_X$ and $u$.
Then there is a semi-circle in $(s,t)$-plane such that
$\xi(\beta,tH)=u$, where
$\beta=\delta+sH-D$,
$D=\zeta-\frac{(\zeta,H)}{(H^2)}H$.
Then Lemma \ref{lem:isom} implies 
$$
M_{(\delta+sH-D,tH)}(v)=\baM_H^{\beta-\frac{1}{2}K_X}(v).
$$
Since $\beta=\delta-\zeta+\left(s-\frac{(\zeta,H)}{(H^2)}\right) H$,
we also have 
$$
M_{(\delta+sH-D,tH)}(v)=\baM_H^{\delta-\zeta-\frac{1}{2}K_X}(v).
$$

\begin{prop}
Assume that there is no wall between
$H+(H,\delta)\varrho_X$ and $u:=\zeta+(\zeta,\delta)\varrho_X-
(e^\delta+\frac{(v^2)}{2r^2}\varrho_X)$.
Then 
for $\alpha \in K(X)_{\Bbb Q}$ with 
$\ch(\alpha)=u(\td_X^{-1})^{\vee}$,
${\cal L}(\alpha)$ is a ${\Bbb Q}$-ample divisor.
In particular, for an adjacent chamber ${\cal C}$ of $\varrho_X^\perp$,
$\theta_v({\cal C}) \subset
\Amp(\baM_H^\gamma(v))_{\Bbb Q}$,
where $\gamma=\delta-\zeta-\frac{1}{2}K_X$. 
\end{prop}

\begin{proof}
Let $(\beta,tH)$ satisfy
$\xi(\beta,tH)=u$ as above, where
$\beta \equiv \delta-\zeta \mod {\Bbb R}H$.
Then we see that 
\begin{equation}\label{eq:u2}
\begin{split}
u=& -e^\delta(1-\zeta+\tfrac{(v^2)}{2r^2}\varrho_X)\\
\equiv & -(1+\beta+((\beta,\delta)-\tfrac{a}{r})\varrho_X ) 
\mod
{\Bbb R}(H+(H,\delta)\varrho_X)\\
=& -\left(e^\beta+\frac{(e^\beta,v)}{r}\varrho_X \right) 
\mod {\Bbb R}(H+(H,\delta)\varrho_X)
\end{split}
\end{equation}
by \eqref{eq:chi}, where
$v=r(1+\delta+\frac{a}{r}\varrho_X)$.
Hence 
\begin{equation}
\begin{split}
u(\td_X^{-1})^{\vee} \equiv &
-\left(e^{\beta-\frac{1}{2}K_X}
-\frac{\chi(e^{\beta-\frac{1}{2}K_X},v)}{r}\varrho_X \right)
\mod
{\Bbb R} \ch(\xi_1)\\
\equiv & \ch(\xi_2) \mod {\Bbb R}\ch(\xi_1).
\end{split}
\end{equation}
We take $u'$ such that $u'=u-\epsilon (H+(H,\delta)\varrho_X)$ 
($\epsilon>0$) belongs to the same chamber.
We take $\alpha'$ with $\ch(\alpha')=u' (\td_X^{-1})^{\vee}$.  
Then
$$
c_1({\cal L}(\alpha'))=
c_1(p_{M_H^\gamma(v)*}({u'}^{\vee}\ch{\cal E}_v)).
$$
By \cite{BM2}, ${\cal L}(\alpha')$ is nef.
We take an ample ${\Bbb Q}$-divisor 
${\cal L}(\xi_1+\frac{1}{n}\xi_2)$, $n \gg 0$.
Since ${\cal L}(\alpha)={\cal L}((1-t)\alpha'+ t(\xi_1+\frac{1}{n}\xi_2))$
($1>t>0$),
Lemma \ref{lem:ample} implies 
${\cal L}(\alpha)$ is ample. 
\end{proof}

\begin{rem}
If $\xi(\beta,\omega)$ belongs to a wall
$v_1^\perp$ of ${\cal C}$
and there are $\sigma_{(\beta,\omega)}$-stable
objects $E_1$ and $E_2$ with
$v(E_1)=v_1$ and $v(E_2)=v-v_1$
such that
$\dim \Ext^1(E_2,E_1) \geq 2$ and
$\phi_{(\beta',\omega')}(E_1)< \phi_{(\beta',\omega')}(E_2)$
for $\xi(\beta',\omega') \in {\cal C}$.
Then $\theta_v(\xi(\beta,\omega))$ is not ample. 
\end{rem}

Let 
$N_H(v)$
 be the Uhlenbeck moduli space of
$\mu$-semi-stable sheaves constructed by Li \cite{Li2}
(see also Remark \ref{rem:relative}).
Thus a point of $N_H(v)$ corresponds to a pair $(F,Z)$ 
of a poly-stable vector
bundle $F$ with respect to $H$ 
and a 0-cycle $Z=\sum_{i=1}^p n_i x_i \in S^n X$
such that $v=v(F)-n\varrho_X$, where $n \geq 0$.
For $(F,\sum_{i=1}^p n_i x_i) \in N_H(v)$, we have
$F \oplus (\bigoplus_{i=1}^p k_{x_i}[-1] ^{\oplus n_i})
\in {\cal M}_{(\delta,H)}(v)$.
Hence $N_H(v)$ paramerizes $S$-equivalence
classes of 
$\sigma_{(\delta,\omega)}$-semi-stable objects.
By the construction of the moduli space,
we have a morphism
$f_H:\baM_H^{\beta-\frac{1}{2}K_X}(v) \to N_H(v)$ and
${\cal L}(\alpha) \in f_H^*(\Pic(N_H(v)))$ 
if $\ch(\alpha)=(H+(H,\delta)\varrho_X)(\td_X^{-1})^{\vee}$. 
\begin{rem}\label{rem:uniqueness-N}
By the construction of $N_H(v)$ in \cite{Li2}, 
$N_H(v)$ may depend on the choice of $H$. 
Assume that ${\cal M}_H(v)^{\mu\text{-}ss}$ is irreducible,
normal and ${\cal M}_H(v)^{\mu\text{-}s}$ is an open
dense substack.
Then we can construct $N_H(v)$ as a normal projective variety
and $f_H:\baM_H^\gamma(v) \to N_H(v)$ is a birational
map.
If $H'$ and $H$ belongs to the same chamber, then
we have a birational map $f_{H'}\circ f_H^{-1}:N_H(v) \to N_{H'}(v)$.
Since $f_{H'}\circ f_H^{-1}(z)$ is a point for all
$z \in N_H(v)$, it is an isomorphism.
In particular, if $(v^2) \gg 0$, then
$N_H(v)$ depends only on the chamber.
By \cite[1.4]{Ktheory}, ${\cal L}(\alpha) \in f_H^*(\Pic(N_H(v)))$
if $\rk \alpha=0$.

By the same argument, we have a contraction $N_H(v) \to N_{H'}(v)$
if $H$ belongs to a chamber and $H'$ is on the boundary of the chamber. 
\end{rem}

\begin{rem}\label{rem:relative}
We have a morphism
$f:\baM_H^\gamma(v) \to \Pic^0(X)$ 
by sending $E$ to $\det (E-E_0)$, where $E_0$ is a fixed element
of $\baM_H^\gamma(v)$.
In the construction of
the Uhlenbeck moduli space \cite{Li2},
the determinant $\det E$ is fixed. 
In order to construct the Uhlenbeck moduli space for 
$\baM_H^\gamma(v)$, we need to recall its construction.  

We first note that $f$ is \'{e}tale locally trivial.
Indeed we have an \'{e}tale morphism
$f^{-1}(0) \times \Pic^0(X) \to M_H(v)$ by
sending $(E,L) \in f^{-1}(0) \times \Pic^0(X)$
to $E \otimes L$, and hence the composed morphism 
$\{E_0 \} \times \Pic^0(X) \to M_H(v) \to \Pic^0(X)$ is
the multiplication by $\rk E_0$.
Hence we have an identification $i_x:f^{-1}(x) \cong f^{-1}(0)$.
Assume that $\ch (\alpha)=(H+(H,\delta)\varrho_X)(\td_X^{-1})^{\vee}$. 
Then
${\cal L}(n\alpha)_{|f^{-1}(x)}$ is independent of $x \in \Pic^0(X)$
under the identification $i_x$.
By the base change theorem,
$f_*({\cal L}(n\alpha))$ $(n \gg 0)$
is locally free and 
$f_*({\cal L}(n\alpha))_s \cong f_*({\cal L}(n\alpha)_s)$.
Since ${\cal L}(n\alpha)_s$ is base point free,
we have a surjective homomorphism
$f^*(f_*({\cal L}(n\alpha))) \to {\cal L}(n\alpha)$, which implies
we have a morphism
$\baM_H^\gamma(v) \to {\Bbb P}(f_*({\cal L}(n\alpha)))$ over
$\Pic^0(X)$ for $n \gg 0$. Then the image is the Uhlenbeck moduli space.
\end{rem}

\begin{NB}
For an extension
$0 \to E_1 \to E \to E_2 \to 0$,
$E$ deforms to $E_1 \oplus E_2$ by the parameter space
$\Ext^1(E_2,E_1)$.
So $E$ deforms to $\oplus_i E_i$, where $E_i$ is a $\mu$-stable sheaves.
So assume that $E$ is a poly-stable sheaf.
Since the fiber of ${\cal M}_H(v)^{\mu\text{-}ss} \to N_H(v)$
over $(E^{\vee \vee},\sum_i n_i P_i)$
contains a product of punctual quot-schemes 
$\prod_i \Quot_{{\cal O}_{P_i}^r/X}^{n_i}$
which is irreducible, 
the fiber is connected.
\end{NB}

\subsection{Wall crossing for Gieseker and Uhlenbeck moduli spaces}
\label{subsect:MW}

Assume that $H \in \Amp(X)$.
Since the set of walls is locally finite 
(Lemma \ref{lem:locally-finite}),
in a small neighborhood $U$ of $H+(H,\delta)\varrho_X$,
we may assume that all walls pass through 
$H+(H,\delta)\varrho_X$. 
Then ${\cal M}_{(\beta,tH)}(v)={\cal M}_H^{\beta-\frac{1}{2}K_X}(v)$
if $\xi(\beta,tH) \in U$ and $(\beta,H)<(\delta,H)$.
We classify such walls.
For $H+(H,\delta)\varrho_X$,
$v_1:=r_1+\xi_1+a_1 \varrho_X \in (H+(H,\delta)\varrho_X)^\perp$
if and only if $(\xi_1/r_1,H)=(\delta,H)$.
We set $v_2:=v-v_1$.
Assume that $\xi(\beta,tH) \in v_1^\perp$.
We may assume that ${\cal M}_{(\beta,tH)}(v_1) \ne 
\emptyset$ and
${\cal M}_{(\beta,tH)}(v_2) \ne \emptyset$.
By shrinking a neighborhood $U$, we may also assume that 
there is no wall for $v_1, v_2$ between
$\xi(\beta,tH)$ and $H+(H,\delta)\varrho_X$.
Then ${\cal M}_{(\beta,tH)}(v_i)={\cal M}_H^{\beta-\frac{1}{2}K_X}(v_i)$
for $i=1,2$. 
In particular, $(v_1^2),(v_2^2) \geq 0$. 
$u=\xi(\beta,tH) \in v_1^\perp$ 
means that 
\begin{equation}
\begin{split}
& \chi \left(e^{\beta-\frac{1}{2}K_X}
-\frac{\chi(e^{\beta-\frac{1}{2}K_X},v)}{r}\varrho_X,v_1 \right)\\
=& -(e^\beta,v_1)+\frac{r_1}{r}(e^\beta,v)=0
\end{split}
\end{equation}
by \eqref{eq:chi}, \eqref{eq:u2} and
$(H+(H,\delta)\varrho_X,v_1)=0$.
Thus 
$$
\frac{\chi(e^{\beta-\frac{1}{2}K_X},v_1)}{r_1}=
\frac{\chi(e^{\beta-\frac{1}{2}K_X},v)}{r}.
$$
Hence if $u$ crosses a wall $v_1^\perp$, then
twisted semi-stability changes. 

Therefore the wall crossing of Gieseker semi-stability
is naturally understood by the wall crossing of Bridgeland's stability:
For a family of vectors
$$
u_H:=n(H+(H,\delta)\varrho_X)-
\left(e^\delta+\frac{(v^2)}{2r^2}\varrho_X \right),
$$ 
$H \in U \subset \Amp(X)$,
we have a family of moduli spaces, where $U$ is a compact subset and
$n$ is a sufficiently large integer depending on $U$.
\begin{NB}
Let $v_1^\perp,v_2^\perp,...,v_k^\perp$ be the set of walls passing through
$H_0+(H_0,\delta)\varrho_X$.
Then there are chambers ${\cal C}_1,{\cal C}_2,...,{\cal C}_l$ 
and $M_{{\cal C}_i}(v)=M_{(\gamma_i,tH)}(v)
=M_H^{\gamma_i-\frac{1}{2}K_X}(v)$.
The relation of these moduli spaces are described as a sequence of
Mumford-Thaddeus type flips.
\end{NB}
If $H$ crosses a wall for $\mu$-semi-stability, then
$u_H$ crosses all walls for twisted semi-stability.
If $H_0$ is on a wall and let $H_\pm$ be ample divisors in
adjacent chambers, then
we have morphisms $N_{H_\pm}(v) \to N_{H_0}(v)$
constructed by Hu and Li \cite{HL} (see also Remark \ref{rem:uniqueness-N}).
On the other hand, for the Gieseker moduli spaces,
we have a sequence of flops as was constructed by 
Ellingsrud-G\"{o}ttsche \cite{EG}, 
Friedman-Qin \cite{FQ}, Matsuki-Wentworth \cite{MW}.

\subsection{The case of a blow-up}\label{subsect:wall-crossing:blow-up}
Assume that $k:=(c_1(v),C)$ is normalized as $0 \leq k <r$. 
We take $\beta_0:=\delta+p_0 C$ 
such that $(\beta_0,C)=0$, that is, $\beta_0=\pi^*(\beta')$
($\beta' \in \NS(Y)_{\Bbb Q}$).
Let $U$ be the open neighborhood of $(\beta_0,H)$ 
in subsection \ref{subsect:classification-wall}.
\begin{NB}
In a neighborhood of $H+(H,\delta)\varrho_X \in \xi(U)$, 
we shall study the wall.
We assume that $v \in \pi^*(H^*(Y,{\Bbb Q}))$
and $(r,(c_1(v),H))=1$. 
\end{NB}
We take $p_1$ with
$p_0+\frac{1}{2} \gg p_1 > p_0$.
Then we may assume that
\begin{equation}\label{eq:image=q<0}
\{\xi(\beta_0+sH,H-qC) \mid \epsilon \leq 0, q>0\}
\supset \{e^\delta(H+xC+y(1+\tfrac{(v^2)}{2r^2}\varrho_X)) \in\xi(U)
\mid x< p_1 y, y \leq 0\},
\end{equation}
where $\epsilon:=s(H^2)+p_0 q$.
Then
${\cal M}_{(\beta_0+sH,H-qC)}(v)={\cal M}_{(\beta,tH)}(v)$
for $q \geq 0$ and $\epsilon \leq 0$,
where $\xi(\beta,tH)=\xi(\beta_0+sH,H-qC)$.
In paricular the wall crossing in $U$ covers 
the wall crossing for the moduli of perverse coherent sheaves
in \cite{perv2}.

We first look at the wall crossing along $\varrho_X^\perp$.
We note that 
$\sigma_{(\beta_0,H)}$ is on the wall defined by $\varrho_X$ and
also on walls defined by $v({\cal O}_C(-n))$.
We define $v' \in H^*(Y,{\Bbb Q})$,
$\delta' \in \NS(Y)_{\Bbb Q}$ and $H'$ by
\begin{equation}
\pi^*(v')=v+kv({\cal O}_C),\;\;\delta':=\frac{c_1(v')}{r},\;\;H=\pi^*(H').
\end{equation}
\begin{lem}
\begin{enumerate}
\item[(1)]
${\cal M}_{(\beta_0,H)}(v)$ consists of
$E$ which is $S$-equivalence to 
$\pi^*(F) \oplus A$
where $F$ is a locally free poly-stable sheaf $F$ on $Y$ and 
\begin{equation}
A=\bigoplus_{i=1}^m   k_{x_i}[-1] \oplus {\cal O}_C[-1]^{\oplus n_1}
\oplus {\cal O}_C(-1)^{\oplus n_2},\;(x_i \in X \setminus C).
\end{equation}
\item[(2)]
The $S$-equivalence class of $E$ in (1) is uniquely 
determined by the $S$-equivalence class of 
${\bf R}\pi_*(E) \in {\cal M}_{(\delta',H')}(v')$, that is,
it is determined by $F$ and
\begin{equation}
A':={\bf R}\pi_*(A)=
\bigoplus_{i=1}^m   k_{x_i}[-1] \oplus k_p[-1]^{\oplus n_1},
\;(x_i \in X \setminus C). 
\end{equation}
\end{enumerate}
\end{lem}

\begin{proof}

(2) 
By $A'$, $\{x_1,...,x_m\}$ and $n_1$ are determined.
Since $c_1(E)=c_1(\pi^*(F))+(n_2-n_1)C$,
$n_2$ is also determined.
Hence the $S$-equivalence class of $E$ is uniquely
 detemined by the $S$-equivalence class of 
${\bf R}\pi_*(E)$.
\end{proof}

   
For $\sigma_{(\beta_0+sH,H-q C)}$ with $s(H^2)+p_0 q=0$ and
$\sigma_{(\beta_0,H)}$ on $\varrho_X^\perp$,
\begin{NB}
Then
$N_{H-q C}(v)$ is regarded as the set of 
$\sigma_{(\delta,H-q C)}$-semi-stable objects.
${\cal M}_{(\delta,H-q C)}(v)$ is the Uhlenbeck space
with respect to $H-q C$
and ${\cal M}_{(\delta,H)}(v)$ is the Uhlenbeck space 
on $Y$.
\end{NB}
we have a contraction $N_{H-q C}(v) \to N_{H'}(v')$.

\begin{lem}
For $E \in {\frak C}^0$ with
$\Hom({\cal O}_C,E)=0$,
there is an exact sequence
\begin{equation}\label{eq:univ-ext}
0 \to E \to {\bf L}\pi^*(F) \to \Ext^1({\cal O}_C,E) \otimes {\cal O}_C
\to 0.
\end{equation}
\end{lem}

\begin{proof}
We note that $\Ext^2({\cal O}_C,E)=\Hom(E,{\cal O}_C(-1))^{\vee}=0$.
Hence 
$$
\dim \Ext^1({\cal O}_C,E)=-\chi({\cal O}_C,E)=(c_1(E),C).
$$
Then the universal extension $E'$ is 
an object of ${\frak C}^0$ such that 
$\Hom({\cal O}_C,E')=0$ and $(c_1(E'),C)=0$.
Since $F:=\pi_*(E')$ is torsion free,
${\bf L}\pi^*(F)=\pi^*(F)$ and $\phi:\pi^*(F) \to E'$ is surjective
in ${\frak C}^0$.
By comparing the first Chern classes,
we see that $\phi$ is isomorphic. 
\end{proof}

\begin{rem}
Since ${\bf R}\pi_*({\cal O}_X(C))={\cal O}_Y$,
$F={\bf R}\pi_*({\bf L}\pi^*(F)(C))=\pi_*(E(C))$.
\end{rem}

In ${\cal A}_{(\beta'+sH',H')}(X/Y)$, \eqref{eq:univ-ext}
gives an exact sequence 
\begin{equation}\label{eq:HNF:q<0}
0 \to \Hom({\cal O}_C[-1],E) \otimes {\cal O}_C
\to E \to {\bf L}\pi^*(F) \to 0.
\end{equation}

Therefore we get the following.
\begin{prop}\label{prop:Uhlenbeck:blou-up}
\begin{enumerate}
\item[(1)]
If $k \ne 0$, then
${\cal M}_{(\beta_0,H-qC)}(v) =\emptyset$
for $q<0$.
\item[(2)]
If $k = 0$, then
${\cal M}_{(\beta_0,H-qC)}(v) =
{\cal M}_{(\delta',H')}(v')$
for $q<0$.
\end{enumerate}
\end{prop}

We next consider the wall crossing for the Gieseker semi-stability.
For simplicity, we assume that
$\gcd(r,(c_1(v),H))=1$.
In this case, the hyperplanes
$v({\cal O}_C(-n))^\perp$ $(n>0)$ in $C^+(v)$
are the candidates of walls. 
Let ${\cal C}_n$ be a chamber
between $v({\cal O}_C(-n))^\perp$ and
$v({\cal O}_C(-n-1))^\perp$.  
Then $n-\frac{1}{2}<-(\beta,C)<n+\frac{1}{2}$
for $\xi(\beta,\omega) \in {\cal C}_n$ with 
$\omega \in {\Bbb R}_{>0}H$
and ${\cal M}_{{\cal C}_n}(v) \cong
{\cal M}_H^{\beta-\frac{1}{2}K_X}(v)$.
\begin{NB}
For $\beta_s:=\beta+sH$,
$\sigma_{(\beta_s,H)}$ is a family of stability conditions
such that $\xi(\beta_s,H)$ is a line passing through
$H+(H,\delta)\varrho_X$ and
$\xi(\beta_s,H) \in {\cal C}_n$
for $s=(\delta-\beta,H)/(H^2)$.
We have ${\frak C}^{\beta_s}={\frak C}^\beta$.
\end{NB}
By Lemma \ref{lem:l=infty} or the finiteness of walls in $U$,
there is an integer $N(v)$ such that
${\cal M}_{{\cal C}_n}(v)={\cal M}_{H-q C}^{\beta-\frac{1}{2}K_X}(v)$
for $n \geq N(v)$ where $q>0$ is sufficiently small.
If $k=(c_1(v),C)=0$, then ${\cal M}_{{\cal C}_0}(v)$ is the moduli space
of semi-stable sheaves on $Y$.
Thus the wall-crossing in \cite{perv2} is the wall crossing in 
the space of stability conditions.
If $0<k<r$, then by \eqref{eq:HNF:q<0},
we have a morphism 
$$
{\cal M}_{{\cal C}_0}(v) \to
{\cal M}_{(\beta'+sH',H')}(v')
$$ 
whose general fibers are $Gr(r,k)$-bundles.
If $\xi(\beta_0+sH,H-qC) \in {\cal C}_0$ 
crosses the wall $v({\cal O}_C)^\perp$, then
${\cal M}_{(\beta_0+sH,H-qC)}(v)=\emptyset$ for $q<0$.

The relation between 
$\baM_{{\cal C}_0}(v)$ and 
$\baM_{{\cal C}_{N(v)}}(v)$
is described as Mumford-Thaddeus type flips:
\begin{equation}
\begin{matrix}
          \baM_{{\cal C}_0}(v) &&&& \baM_{{\cal C}_1}(v)&&&&&
    \baM_{{\cal C}_{N(v)}}(v),\cr
          &\searrow &&\swarrow&&\searrow&&\cdots & \swarrow &\cr
          &&\baM_{{\cal C}_{0,1}}(v)&&&&
\baM_{{\cal C}_{1,2}}(v)&&\cr    
\end{matrix}
\end{equation}
\begin{NB}
If $n \gg 0$, then
${\cal O}_C(-a)$ does not define a wall
in ${\cal K}(X)$.
Thus $\xi^{-1}({\cal C}_n)$ and
$\xi^{-1}({\cal C}_{n+1})$ belong o the same
chamber.
\end{NB}
where
$\baM_{{\cal C}_{n,n+1}}(v)$ parameterizes $S$-equivalence classes
of $\sigma_{(\beta,\omega)}$-twisted semi-stable objects
with $\xi(\beta,\omega) \in v({\cal O}_C(-n-1))^\perp$.

\begin{rem}
Although we considered the wall-crossing behavior
in a neighborhood of $(\beta_0,H)$,
we may move the parameter $p$ as in \cite{perv2}.
In ${\cal C}_n$, $n \geq 0$,
the wall-crossing behavour is the same.
However for the wall $v({\cal O}_C)^\perp$,
the behavior is different.
For $E \in {\cal M}_{{\cal C}_0}(v)$,
we have an exact sequence
\begin{equation}
0 \to E' \to 
E \to \Hom(E,{\cal O}_C)^{\vee} \otimes {\cal O}_C \to 0
\end{equation}
with $E' \in {\frak C}^1 \cap {\frak C}^0$.
If $k>0$, then
we have a morphism
$\baM_{{\cal C}_0}(v) \to M_{(\beta',H')}(w')$
by sending $E$ to
$\pi_*(E(C))$,
where $\pi^*(w')=v-(r-k)v({\cal O}_C)$.   
\end{rem}

\begin{NB}
Let us briefly explain the structure of
$\baM_{{\cal C}_{n,n+1}}(v)$.
 
We take $G_n:=\pi^*(E_0) \otimes {\cal O}_X(nC)$.
Then ${\bf R}\pi_*(G_n^{\vee} \otimes {\cal O}_C(-n-1))=0$.
Let $E$ be an object of ${\frak C}^\beta$ such that
$E$ does not contain a 0-dimensional subobject $T \ne 0$. 
Since $\chi(G_n,E(m))=\chi(E_0,\pi_*(E(-nC))(m))$
and $E(-nC)$ fits in an exact sequence
\begin{equation}
0 \to {\cal O}_C(-1)^p \to \pi^*(\pi_*(E(-nC))) \to E(-nC) \to 0.
\end{equation}

\begin{equation}
\frac{\chi(G_n,E_1(m))}{\rk E_1} \leq 
\frac{\chi(G_n,E(m))}{\rk E}
\end{equation}
for any subobject $E_1$ of $E$ in ${\frak C}^\beta$ if and only if
$\pi_*(E(-nC))$ is $E_0$-twisted semi-stable.
Hence we have a morphism
$\baM_{{\cal C}_n}(v) \to \baM_{H'}^{\beta'}(v')$
by sending $E$ to $\pi_*(E(-nC))$, where
$\beta'=\frac{c_1(E_0)}{\rk E_0}$ and $v'=v(\pi_*(E(-nC)))$.
As a Brill-Noether locus, we can introduce a scheme structure
on the image (\cite[Prop. 3.31]{perv2}). We denote this scheme by 
$\baM_{{\cal C}_{n,n+1}}(v)$. 
Set-theoretically, we have the following description:
\begin{equation}
\baM_{{\cal C}_{n,n+1}}(v)
=\left\{ F \oplus {\cal O}_C(-n-1)^{\oplus p} \left|
\begin{aligned}
F \in \baM_{{\cal C}_n}(v') \cap
\baM_{{\cal C}_{n+1}}(v')\\
p \in {\Bbb Z}_{\geq 0},\;
v'+pv({\cal O}_C(-n-1))=v 
\end{aligned}
\right. \right\}. 
\end{equation}
\end{NB}

\begin{NB}
We don't need the following anymore (Oct. 2014)
\begin{rem}
Since ${\cal M}_{C_N(v)}(v)={\cal M}_{H-q C}(v)^{\mu\text{-}ss}$,
it seems there is no wall between $C_N(v)$ and
$e^\delta(H-q C)$.
Unfortunately we can not exclude the possibility of the existence
of properly semi-stable objects 
at this moment. 
\end{rem}
\end{NB}

\begin{NB}
We take an ample divisor $L:=H-q C$, 
where $q>0$ is sufficiently
small.
There is $t_0>0$ such that
$(\beta,t_0 L) \in \xi^{-1}$.
Along the line $\{(\beta,tL) \mid t>t_0 \}$,
there are finitely many walls by Lemma \ref{lem:mini-wall}.
Hence almost all ${\cal O}_C(-n)$ do not define walls for $v$, and
we have a sequence of birational maps
$$
M_{H'}^{\beta'-\frac{1}{2}K_Y}(v')=M_{(\beta,t_0 L)}(v)
\leftarrow \cdots \rightarrow
M_{(\beta,t_1 L)}(v)
\leftarrow \cdots \rightarrow
\cdots 
\leftarrow \cdots \rightarrow
M_{(\beta,tL)}(v)=M_L^{\beta-\frac{1}{2}K_X}(v).
$$
For $u=e^\delta(\zeta+(1+\frac{(v^2)}{2r^2}\varrho_X))$,
we set 
$$
D_q:=
\zeta-\frac{(\zeta,H-q C)}{((H-q C)^2)}(H-q C).
$$
Then
$$
D_q-D_0=q \left(
\frac{(\zeta,C)(H^2)-(\zeta,H)q}{(H^2)((H^2)-q^2)}H
+\frac{(\zeta,H-q C)}{(H^2)-q^2}q C \right).
$$

\end{NB}

\begin{NB}
For $E \in {\frak C}^{\beta+C}$,
we have an exact sequence
\begin{equation}
0 \to E'(-nC) \to E(-nC) \to {\cal O}_C(-1)^{\oplus p} \to 0
\end{equation}
where $E' \in {\frak C}^\beta$.
\end{NB}

\begin{NB}
${\bf R}\pi_*:{\bf D}(X) \to {\bf D}(Y)$
induces a morphsim
${\cal M}_{(\delta,H)}(v) \to {\cal M}_{(\delta',H')}(v')$,
where $(\pi^*(\delta'),\pi^*(H'),\pi^*(v'))=
(\delta,H,v)$.
If $E$ is a torsion free sheaf, then
we have
$$
0 \to \pi_*(E) \to {\bf R}\pi_*(E) \to R^1 \pi_*(E)[-1] \to 0.
$$
$\pi_*(E)$ is a torsion free sheaf of $d_\delta(\pi_*(E))=0$ and
and $R^1 \pi_*(E)$ is 0-dimensional with $\phi=0$.
\end{NB}

\begin{NB}
Assume that ${\cal M}_H(v)^{\mu\text{-}ss}$ and
${\cal M}_{H'}(v')^{\mu\text{-}ss}$ are irreducible normal stack.
Then we have a morphism $\baM_{H-q C}(v) \to
N_{H'}(v')$, by sending $E$ to the $S$-equivalence
class of 
$\pi_*(E)^{\vee \vee} \oplus 
\pi_*(E)^{\vee \vee}/\pi_*(E) \oplus R^1 \pi_*(E)$.
\end{NB}

\begin{NB}
\subsection{Bogomolov inequality}

Toda proved that 
$(v^2) \geq C_\omega d_\beta(v)^2 (H^2)$,
where $d_\beta(v)(H^2)=(c_1-r \beta,H)$.

If $\phi_{(\beta,\omega)}(v_1)=\phi_{(\beta,\omega)}(v_2)$, then
\begin{equation}
(v_1,v_2) \geq \frac{d_2}{2d_1}(v_1^2)+\frac{d_1}{2d_2}(v_2^2)
\geq c_\omega (H^2)d_1 d_2.
\end{equation}
Then $((v_1+v_2)^2) \geq C_\omega (H^2)(d_1+d_2)^2$.

\end{NB}

\begin{NB}

\section{The blow-up case}

We set $\beta=pC$ and $\omega=H+tC$.
Then 
\begin{equation}
Z_{(\beta,\omega)}(E)=
\left(p(c_1(E),C)+\rk E\frac{(H^2)+p^2-t^2}{2}-\ch_2(E) \right)
+\sqrt{-1}
\left((c_1(E),H+tC)+rpt \right).
\end{equation}

For ${\cal O}_C(-a-\frac{1}{2})$,
$Z_{(\beta,\omega)}({\cal O}_C(-a-\frac{1}{2}))=
( e^{\beta+\sqrt{-1}\omega},(0,C,-a))
=-(p+\lambda)-t \sqrt{-1}.
$

The equation of the wall defined by ${\cal O}_C(-a-\frac{1}{2})$.
\begin{equation}
t\left(p(c_1,C)+r\frac{(H^2)+p^2-t^2}{2}-\ch_2 \right)-
(p+\lambda)\left((c_1,H+tC)+rpt \right)=0.
\end{equation} 

\begin{equation}
t\left( r\frac{(H^2)+\lambda^2-(p+\lambda)^2-t^2}{2}-\ch_2-\lambda(c_1,C)
\right)-(p+\lambda)(c_1,H)=0.
\end{equation}

\begin{equation}
p+\lambda=\frac{(c_1,H)\pm\sqrt{
(c_1,H)^2+rt(rt((H^2)-t^2)-2t\ch_2-2\lambda t(c_1,C))}}{-rt}
\end{equation}
Since $(c_1,H)>0$, 
if $-1 \ll t<0$, then
$p+\lambda$ is sufficiently large, or
\begin{equation}
\begin{split}
p+\lambda=&\frac{(c_1,H)-\sqrt{
(c_1,H)^2+rt(rt((H^2)-t^2)-2t\ch_2-2\lambda t(c_1,C))}}{-rt}\\
\sim & t\frac{r((H^2)-t^2)-2\ch_2-2\lambda (c_1,C)}
{(c_1,H)(1+
\sqrt{
1+\frac{t}{(c_1,H)^2}(rt((H^2)-t^2)-2t\ch_2-2\lambda t(c_1,C))})}
\end{split}
\end{equation}

We set 
$v=e^\beta(r+a_\beta \varrho_X+d_\beta H+D_\beta)$, 
where $D_\beta \in H^\perp$.
\begin{lem}
Assume that
$(H^2) >\frac{d_\beta-d_{\min}}{r\delta}((v^2)-(D_\beta^2))$ and $t=0$, then
$E$ is $\sigma_{(\beta,\omega)}$-semi-stable if and only if
$E \in {\frak C}$ and $E$ is $\beta-\frac{1}{2}K_X$-twisted semi-stable.
\end{lem}

If $\beta=pC$, then
$d_\beta(E)(H^2)=(c_1(E)-\rk E pC,H)=(c_1(E),H)$.
More generally,
$d_\beta(E)(H^2)=(c_1(E)-r \pi_*(\beta),H)$. 

For $\omega=H$,
$(Z_{(\beta,\omega)},{\frak A}_{(\beta,\omega)})$ is a stability 
condition.

\begin{rem}
Assume that $X$ is a Del Pezzo surface.
Let $D$ be an irreducible and reduced curve.
Then $(D,D+K_X) \geq -2$.
Since $(D,-K_X)>0$, 
if $(D^2)<0$, then
$(D^2)=(D,K_X)=-1$.
Thus $D$ is a $(-1)$-curve.
   
\end{rem}

We set
\begin{equation}
f(t):=-\frac{r}{2}t^3+
\left(\frac{r}{2}(-(p+\lambda)^2+\lambda^2+(H^2))-\ch_2-\lambda(c_1,C) \right)t
-(p+\lambda)(c_1,H).
\end{equation}
Then $f'(t)=0$ at
$$
t_\pm =\pm \sqrt{\frac{1}{3}\left(-(p+\lambda)^2+\lambda^2+(H^2)-
\frac{\ch_2}{r}-
\frac{\lambda}{r}(c_1,C) \right)}.
$$

Then 
$f(t_-)=r t_+^3-(p+\lambda)(c_1,H)$.
If $(H^2)$ is sufficiently larger than $|(p+\lambda)|$,
$r,|\ch_2|,|(c_1,C)|,(c_1,H)$, then
$t_- \sim -\sqrt{\frac{(H^2)}{3}}$ and
$f(t_+) \sim r\frac{(H^2)}{3}\sqrt{\frac{(H^2)}{3}}$.

We set $H=s H_0$. Then 
$H+tC=sH_0+tC=s(H_0+\frac{t}{s}C)$.
If $0>\frac{t}{s} \gg -1$ and $s \gg |\tfrac{t}{s}|$, then 
$\sigma_{(\beta,\omega)}$-semi-stability is the same as
the Gieseker semi-stability.

If $s \gg 0$ and $0>t(s)>-1$ satisfies
$f(t(s))=0$, then 
$t(s) \sim \frac{2(p+\lambda)(c_1,H)}{r(H^2)}
=\frac{2(p+\lambda)(c_1,H_0)}{r(H_0^2)s}$
and
$|t(s)|$ is sufficiently small, that is
$\lim_{s \to \infty} t(s)=0$. 
Thus in order to compare $\sigma_{(pC,H)}$-semi-stabilty
with Gieseker semi-stability with respect to 
$H_0+\frac{t}{s}C$, $0>\frac{t}{s} \gg-1$,
we need to cross all walls defined by
${\cal O}_C(\lambda-\frac{1}{2})$.

\begin{equation}
\begin{split}
& \left(r(c_1',H+tC)-r'(c_1,H+tC)\right)\frac{(H^2)+p^2-t^2}{2}\\
+& \left((p(c_1,C)-\ch_2)((c_1',H+tC)+r' pt)-
(p(c_1',C)-\ch_2')((c_1,H+tC)+rpt)\right)=0
\end{split}
\end{equation}

Assume that $|\frac{t}{s}|$ is bounded.
\begin{equation}
\begin{split}
& s^2 \left(r(c_1',H_0+\tfrac{t}{s}C)-r'(c_1,H_0+\tfrac{t}{s}C)\right)
\frac{(H_0^2)+\tfrac{p^2}{s^2}-\tfrac{t^2}{s^2}}{2}\\
+& \left((p(c_1,C)-\ch_2)((c_1',H_0+\tfrac{t}{s}C)+r' p \tfrac{t}{s})-
(p(c_1',C)-\ch_2')((c_1,H_0+\tfrac{t}{s}C)+rp\tfrac{t}{s})\right)=0
\end{split}
\end{equation}

\begin{lem}
There is a positive constant $C$ such that
$|E| \leq C|Z_{(\beta,\omega)}(E)|$
for all $\sigma_{(\beta,\omega)}$-stable object $E$.  
\end{lem}

By \cite[Prop. 3.3]{BM},
the set of walls is locally finite.
Thus there are finitely many vectors $v'$ defining walls for any
compact subset.

\subsection{ }

We set $\beta:=uH$.
Then 
$$
v=e^\beta(r+(c_1-ruH)+(\ch_2-u(c_1,H)+\tfrac{r}{2}u^2 (H^2))).
$$

For $\beta=uH+pC$ and $\omega=sH+tC$,
\begin{equation}
\begin{split}
Z_{(\beta,\omega)}(E)=&
\left(p(c_1(E),C)+\rk E\frac{s^2(H^2)+p^2-t^2}{2}-\ch_2(E)
+u(c_1(E),H)-\frac{\rk E}{2}u^2 (H^2) \right)\\
& +\sqrt{-1}
\left((c_1(E)-\rk E uH,sH+tC)+rpt \right).
\end{split}
\end{equation}

For ${\cal O}_C(\lambda-\frac{1}{2})$,
$Z_{(\beta,\omega)}({\cal O}_C(\lambda-\frac{1}{2}))=
( e^{\beta+\sqrt{-1}\omega},(0,C,\lambda))
=-(p+\lambda)-t \sqrt{-1}.
$

The equation of the wall defined by ${\cal O}_C(\lambda-\frac{1}{2})$.
\begin{equation}
t\left(p(c_1,C)+r\frac{s^2(H^2)+p^2-t^2}{2}-\ch_2 
+u(c_1,H)-\frac{r}{2}u^2 (H^2) \right)-
(p+\lambda)\left((c_1-r uH,sH+tC)+rpt \right)=0.
\end{equation} 

\begin{equation}
t\left( r\frac{s^2(H^2)+\lambda^2-(p+\lambda)^2-t^2}{2}
-\ch_2+u(c_1,H)-\frac{r}{2}u^2 (H^2) -\lambda(c_1,C)
\right)-(p+\lambda)(c_1-ruH,sH)=0.
\end{equation}

\begin{NB2}
For $(\beta,\omega)=(u(H+tC),s(H+tC))$,
the equation is a family of circles:
\begin{equation}
\frac{rt}{2}((H+tC)^2)(s^2+u^2)=\lambda(c_1-ru(H+tC),H+tC).
\end{equation}
\end{NB2}

Assume that $s=1$.
For $u(t):=\frac{(c_1,H+tC)}{r(H^2)}$,
$d_\beta(v)=(c_1-r uH,H+tC)=0$.
\begin{NB2}
If $(c_1,C)=0$, then $u(t)$ is constant.
\end{NB2}
Hence $(\beta,\omega)=(u(t)H+pC,H+tC)$ intersect with
the wall at $t=0$.
Assume that $p=0$.
$(\beta_0,\omega_0)=((u(0)-\epsilon)H,H)$, $0<\epsilon \ll 1$,
$\sigma_{(\beta_0,\omega_0)}$-semi-stability 
is the same as the $(\beta_0-\tfrac{K_X}{2})$-twisted 
semi-stability in ${\frak C}$.

If $(\beta,\omega)=(u(t)H+pC,H+tC)$, 
then $\phi_{(\beta,\omega)}(E)=0$.
If $t<0$, then 
$\phi_{(\beta,\omega)}({\cal O}_C(\lambda-\tfrac{1}{2}))>0$
and 
$\phi_{(\beta,\omega)}({\cal O}_C(\lambda-\tfrac{1}{2}))=0$
for $t=0$.
\begin{NB2}
If ${\cal O}_C(\lambda-\tfrac{1}{2})[1] \in {\frak C}$, then
$\phi_{(\beta,\omega)}({\cal O}_C(\lambda-\tfrac{1}{2})[1])=1$.
Hence
$\phi_{(\beta,\omega)}({\cal O}_C(\lambda-\tfrac{1}{2}))=0$. 
\end{NB2}

For $t<0$ depending on $q$,
$\sigma_{(\beta-q H,\omega)}$-semi-stability is the same as
the $\beta$-twisted semi-stability with respect to
$H+tC$.

If $\phi_{(\beta,\omega)}(E')=0$ at $(\beta_0,\omega_0)$, 
then
$r(c_1(E'),H)=r'(c_1(E),H)$, where
$\rk E'=r'$.
Moreover $E'$ is generated by 
$\mu$-semi-stable objects of ${\frak C}$ and
0-dimensional objects of ${\frak C}[-1]$.
In particular, if $\gcd(r,(c_1,H))=1$, then
$E'$ is generated by 
${\cal O}_C(\lambda-\frac{1}{2})$ with
$\lambda-\frac{1}{2}<0$.

We regard $t$ as a parameter.
Then the wall is 
\begin{equation}
\begin{split}
&\frac{rt}{2}(p+\lambda)^2
-rs(H^2)(p+\lambda)\left(u-\frac{(c_1,H)}{r(H^2)} \right)
+\frac{rt(H^2)}{2}\left(u-\frac{(c_1,H)}{r(H^2)} \right)^2
\\
=&t \left(
r\frac{s^2(H^2)+\lambda^2-t^2}{2}-\ch_2+\frac{(c_1,H)^2}{r(H^2)}
-\lambda(c_1,C) \right).
\end{split}
\end{equation}
So if $-1 \ll t <0$, then
it is a hyperbola.

\begin{NB2}
In order to understand the wall-crossing behavior of Gieseker
semi-stability, it is better to study the case where
$(c_1-uH,H)=0$.  
\end{NB2}

\end{NB}

\section{Appendix}\label{sect:appendix}

\subsection{Some properties of perverse coherent sheaves}

Let $\pi:X \to Y$ be a resolution of a rational singularity.
Let ${\frak C}$ be a category of perverse coherent sheaves and $G$ be a
local projective generator of ${\frak C}$ 
which is a locally free sheaf on $X$
(\cite[Defn. 1.1.3]{PerverseI}).

\begin{lem}\label{lem:reflexive-hull}
For a torsion free object $E$ of ${\frak C}$,
there is an exact sequence 
\begin{equation}
0 \to E \to E' \to T \to 0
\end{equation}
such that $T$ is 0-dimensional, $E'$ is torsion free and
$\Ext^1(A,E')=0$ for all 0-dimensional objects $A$ of
${\frak C}$.
\end{lem}

\begin{proof}
If $\Ext^1(A,E) \ne 0$ for an irreducible object $A$ of ${\frak C}$,
then a non-trivial extension is torsion free.
So we inductively construct torsion free objects $E_n$ 
$$
0 \subset E=E_0 \subset E_1 \subset E_2 \subset \cdots \subset E_n
$$
such that $E_i/E_{i-1}$ are irreducible objects of ${\frak C}^\beta$.
We set $T_n:=E_n/E_0$. 
\begin{NB}
Old argument(different proof):
Since $T_n$ is a 0-dimensional object, we have 
$v(T_n)=D_n+(D_n,\beta)\varrho_X+a_n \varrho_X$, $D_n \in H^\perp$.
Since $0<\chi(G,T_i/T_{i-1})=-(e^\beta,v(T_i/T_{i-1}))=a_i-a_{i-1}$,
we have $a_n \geq n$.
We set $v(E)=e^\beta(r+a \varrho_X+dH+D)$, $D \in H^\perp$.
Since
\begin{equation}
(v(E_n)^2)=(v(E)^2)-2\rk E a_n+((D+D_n)^2)-(D^2)
 \leq (v(E)^2)-(D^2)-2\rk E a_n,
\end{equation}
Bogomolov inequality implies that $n$ is bounded above.
Hence there is $n$ such that $E_n$ satisfies the desired property.
\end{NB}
We have an exact sequence
\begin{equation}
0 \to \pi_*(G^{\vee} \otimes E) \to \pi_*(G^{\vee} \otimes E_n) 
\to {\bf R}\pi_*(G^{\vee} \otimes T_n) \to 0.
\end{equation}
By the torsion freeness of $E,E_n$,
$\pi_*(G^{\vee} \otimes E)$ and $\pi_*(G^{\vee} \otimes E_n)$
are torsion free sheaves on $Y$ by
\begin{equation}
\Hom({\Bbb C}_y,\pi_*(G^{\vee} \otimes E_n))=
\Hom(G \otimes {\cal O}_{\pi^{-1}(y)},E_n)=0,\;y \in Y.
\end{equation}
Hence 
$\pi_*(G^{\vee} \otimes E_n)$ is regarded as a subsheaf of 
$\pi_*(G^{\vee} \otimes E)^{\vee \vee}$ 
and 
$$
\chi(\pi_*(G^{\vee} \otimes E)^{\vee \vee}
/\pi_*(G^{\vee} \otimes E)) \geq 
\chi({\bf R}\pi_*(G^{\vee} \otimes T_n))=\chi(G,T_n). 
$$ 
On the other hand, 
$\chi(G,T_i/T_{i-1})>0$ for all $i$ imply that 
$\chi(G,T_n) \geq n$.
Therefore $n$ is bounded above.
Hence there is $n$ such that $E_n$ satisfies the desired property.
\begin{NB}
Assume that $l-\frac{3}{2}<(\beta,C)<l-\frac{1}{2}$. Then 
${\frak C}^\beta(lC)={^{-1}\Per(X/Y)}$.
Hence we have an exact sequence
$$
0 \to \pi_*(E(lC)) \to \pi_*(E_n(lC)) \to {\bf R}\pi_*(T_n) \to 0
$$ 
\end{NB}
\end{proof}

\begin{lem}\label{lem:weak-Bogomolov2}
For a $\mu$-semi-stable object $E$ of ${\frak C}$, 
$d_\beta(E)^2 (H^2)-2\rk E a_\beta(E) \geq 0$.
\end{lem}

\begin{proof}
We take $E'$ in Lemma \ref{lem:reflexive-hull}.
Then $E'$ is a $\mu$-semi-stable locally free sheaf
with respect to $H$, since
$\Ext^1(k_x,E')=0$ for all $x \in X$ (cf. \cite[Lem. 1.1.31]{PerverseI}).
Since $a_\beta(T) \geq 0$, we have 
$$
d_\beta(E)^2 (H^2)-2\rk E a_\beta(E) \geq
 d_\beta(E')^2 (H^2)-2\rk E' a_\beta(E').
$$
 Then the claim 
follows from the ordinary Bogomolov-Gieseker inequality.
\end{proof}

The following weak form of Bogomolov inequality is an easy consequence
of Lemma \ref{lem:weak-Bogomolov2} and the proof of 
Lemma \ref{lem:strong-Bogomolov}.
\begin{prop}\label{prop:weak-Bogomolov2}
For a $\sigma_{(\beta,\omega)}$-semi-stable object $E$ of
${\cal A}_{(\beta,\omega)}$,
\begin{equation}\label{eq:weak-bogomolov}
d_\beta(E)^2 (H^2)-2\rk E a_\beta(E) \geq 0.
\end{equation}
\end{prop}

We shall postpone the proof of this claim, since we need
a claim in Proposition \ref{prop:large}.
So we shall prove \eqref{eq:weak-bogomolov} in the course 
of the proof of Proposition \ref{prop:large}.

\begin{NB}
If $\pi:X \to Y$ is the blow-up of a smooth surface $Y$
or $\pi$ is the minimal resolution of a rational double
point, then we have the Bogomolov inequality. 
Assume that $G$ satisfies
\eqref{eq:ch(G)}.
Then for a pair $(\beta,\omega)$,
we have a torsion pair $({\cal T}_\beta,{\cal F}_\beta)$
of ${\frak C}$ which defines a tilted category 
${\cal A}_{(\beta,\omega)}$.
By Lemma \ref{lem:weak-Bogomolov},
$\sigma_{(\beta,\omega)}:=({\cal A}_{(\beta,\omega)},Z_{(\beta,\omega)})$
is a stability condition.
\begin{lem}
If $\pi:X \to Y$ is the blow-up of a smooth surface $Y$
or $\pi$ is the minimal resolution of a rational double
point, then the support property also holds.
\end{lem}
\end{NB}

Let ${\frak C}^*={\frak C}^D(-K_X)$ be a category of 
 perverse coherent sheaves such that 
$G^{\vee}(-K_X)$ is a local projective generator
(see \cite[Lem. 1.1.14]{PerverseI}).
Since $(G^{\vee}(-K_X))^{\vee}(-K_X)=G$, we have 
$({\frak C}^*)^*={\frak C}$.
By \cite[Lem. 1.1.14 (4)]{PerverseI}, the following claim holds.
\begin{lem}\label{lem:0-dim-dual}
$A \in {\bf D}(X)$ is a 0-dimensional object
of ${\frak C}$ if and only if 
$A^{\vee}[2]$ is a 0-dimensional object of
${\frak C}^*$.
\end{lem}

For a 0-dimensional object $A$ of ${\frak C}$,
\begin{equation}\label{eq:dual*}
\begin{split}
\Hom(A,E)=& \Hom(E^{\vee}[1],(A^{\vee}[2])[-1]),\\
\Hom(E,A[-1])=& \Hom(A^{\vee}[2],E^{\vee}[1]).
\end{split}
\end{equation}

\begin{lem}\label{lem:duality*}
Let $E \in {\bf D}(X)$ satisfy $\pH^i(E)=0$
($i \ne -1,0$) and $\Hom({\Bbb C}_x,E)=0$ for all $x \in X$.
Then $H^i(E^{\vee}[1])=0$ for $i \ne -1,0$ and
$\Hom(B,E^{\vee}[1])=0$ for all 0-dimensional object
$B$ of ${\frak C}^*$.
Moreover if $\Hom(A,E)=0$ for all 0-dimensional object 
$A$ of ${\frak C}$,
then $\pH^i(E^{\vee}[1])=0$ for $i \ne -1,0$.
\end{lem}

\begin{proof}
Let $G$ be a local projective generator of ${\frak C}$.
Since $\Ext^2(G(-n),\pH^{-1}(E))=0$ for $n \gg 0$,
we have a morphism 
$G(-n)^{\oplus N} \to E$ such that
$\pH^0(G(-n)^{\oplus N}) \to \pH^0(E)$ is surjective in ${\frak C}$.
We set $V_0:=G(-n)^{\oplus N}$ and
$V_{-1}:=\mathrm{Cone}(V_0 \to E)[-1]$.
Then we see that $V_{-1} \in {\frak C}$. 
Since $V_0$ is a locally free sheaf, $\Ext^1({\Bbb C}_x,V_0)=0$
for all $x \in X$.
By $\Ext^2(E,{\Bbb C}_x)=\Hom({\Bbb C}_x,E)^{\vee}=0$,
$\Ext^1(V_{-1},{\Bbb C}_x)=0$ for all $x \in X$.
Hence $V_{-1}$ is a locally free sheaf.
Then $E^{\vee}[1]$ is represented by a two term complex
of locally free sheaves.
For a 0-dimensional object $B$ of ${\frak C}^*$,
we set $A:=B^{\vee}[2]$. Then $A^{\vee}[2]=B$ and 
Lemma \ref{lem:0-dim-dual} implies that $A$ is a 0-dimensional
object of ${\frak C}$.
Then \eqref{eq:dual*} implies
$\Hom(B,E^{\vee}[1])=\Hom(E,A[-1])=0$.

Assume that $\Hom(A,E)=0$ for all 0-dimensional objects.
Then by Lemma \ref{lem:0-dim-dual} and \eqref{eq:dual*},
we see that
$H^0(E^{\vee}[1]) \in {\frak C}^*$. 
For $F:=H^{-1}(E^{\vee}[1])$,
we have an exact triangle
\begin{equation}
\pH^0(F) \to F \to \pH^1(F)[-1] \to \pH^0(F)[1].
\end{equation}
Therefore $\pH^i(E^{\vee}[1])=0$ for $i \ne -1,0$,
$\pH^{-1}(E^{\vee}[1])=\pH^0(F)$ and
we have an exact sequence 
$$
0 \to \pH^1(F) \to
\pH^0(E^{\vee}[1]) \to H^0(E^{\vee}[1]) \to 0
$$
in ${\frak C}^*$.
\end{proof}

\begin{rem}\label{rem:C^D}
$V_i^{\vee}$ are local projective objects of ${\frak C}^D$.
However we do not know whether $V_i^{\vee} \in {\frak C}^*$
or not. 
\end{rem}

\begin{cor}\label{cor:duality*}
For a complex $E \in {\bf D}(X)$,
 $\pH^i(E)=0$
($i \ne -1,0$) and $\Hom(A,E)=0$ for all 0-dimensional
object $A$ of ${\frak C}$ if and only if 
$\pH^i(E^{\vee}[1])=0$ for $i \ne -1,0$ and
$\Hom(B,E^{\vee}[1])=0$ for all 0-dimensional
object $B$ of ${\frak C}^*$.
\end{cor}

\begin{rem}
\begin{enumerate}
\item[(1)]
Let $E$ be a $\sigma$-stable object 
with $0<\phi(E)<1$.
Then $\Hom(A,E)=0$ for all 0-dimensional object $A$ of 
${\frak C}$.
\item[(2)]
For the object $E$ in Lemma \ref{lem:duality*},
if $\rk \pH^0(E)=0$, then $E^{\vee}[1] \in \Coh(X)$.
Moreover if $\Hom(A,E)=0$ for all
0-dimensional object $A$ of ${\frak C}$,
then $E^{\vee}[1] \in {\frak C}^*$ and
$E^{\vee}[1]$ does not contain any 0-dimensional object. 
\item[(3)]
Let $V_{-1} \to V_0$ be a two term complex of locally free sheaves.
Then $\pH^i(E)=0$ for $i \ne 0,-1$ if $\Hom(E,A[-1])=0$ for all
0-dimensional object $A$ of ${\frak C}$.
\end{enumerate}
\end{rem}

\begin{NB}
\begin{lem}[Bogomolov inequality]\label{lem:Bogomolov}
For a $\mu$-semi-stable object $E$ of ${\frak C}$ with $\rk E>0$,
$(v(E)^2) \geq 0$.
\end{lem}

\begin{proof}
In the proof of Lemma \ref{lem:reflexive-hull},
we set
$v(T_n)=D_n+(D_n,\beta)\varrho_X+a_n \varrho_X$, $D_n \in H^\perp$,
since $T_n$ is a 0-dimensional object.
Since $0<\chi(G,T_i/T_{i-1})=-(e^\beta,v(T_i/T_{i-1}))=a_i-a_{i-1}$,
we have $a_n \geq n$.
We set $v(E)=e^\beta(r+a \varrho_X+dH+D)$, $D \in H^\perp$.
Since
\begin{equation}
(v(E_n)^2)=(v(E)^2)-2\rk E a_n+((D+D_n)^2)-(D^2)
 \leq (v(E)^2)-(D^2)-2\rk E a_n,
\end{equation}

\end{proof}
\end{NB}

\begin{lem}\label{lem:irreducible}
Let $\pi:X \to Y$ be the blow-up of a point of a smooth 
surface $Y$ and $C$ the exceptional
divisor. Assume that $l-\frac{1}{2}<(\beta,C)<l+\frac{1}{2}$.
\begin{enumerate}
\item[(1)]
${\cal O}_C(l)$ is an irreducible object
of ${\cal A}_{(\beta,tH)}$.
\item[(2)]
Let $E$ be a subobject of ${\cal O}_C(l)^{\oplus n}$ in
${\cal A}_{(\beta,tH)}$. Then
$E={\cal O}_C(l)^m$ with $0 \leq m \leq n$.
\end{enumerate}
\end{lem}

\begin{proof}
It is sufficient to prove (2).
We set $F:={\cal O}_C(l)^{\oplus n}/E \in {\cal A}_{(\beta,tH)}$.
Then we have an exact sequence
\begin{equation}
0 \to \pH^{-1}(F) \to \pH^0(E) \to {\cal O}_C(l)^{\oplus n}
\to \pH^0(F) \to 0
\end{equation}
in ${\frak C}^l$ and $\pH^{-1}(E)=0$.
Since $(c_1(\pH^{-1}(F)(-\beta)),H) \leq 0$ and
$(c_1(\pH^0(E)(-\beta)),H) \geq 0$,
we see that $\pH^0(E)$ is 0-dimensional and
$\pH^{-1}(F)=0$.
Since ${\cal O}_C(l)$ is an irreducible object of
${\frak C}^l$, we see that $E={\cal O}_C(l)^{\oplus m}$.
\end{proof}

\begin{cor}\label{cor:irreducible}
For $E \in {\frak C}^\beta$, there is a filtration
$$
0 \subset F_1 \subset F_2 \subset F_3=E
$$
such that $F_1={\cal O}_C(l)^{\oplus n_1}$,
$\Hom({\cal O}_C(l),F_2/F_1)=\Hom(F_2/F_1,{\cal O}_C(l))=0$
and $F_3/F_2={\cal O}_C(l)^{\oplus n_3}$.
\end{cor}

\begin{NB}
We note that $\Ext^1({\cal O}_C(l),{\cal O}_C(l))=0$.
By Lemma \ref{lem:irreducible},
$\Hom({\cal O}_C(l),E) \otimes {\cal O}_C(l) \to E$
is injective and $E \to \Hom(E,{\cal O}_C(l))^{\vee} 
\otimes{\cal O}_C(l)$ is surjective.
Hence we can define 
$F_1:=\Hom({\cal O}_C(l),E) \otimes {\cal O}_C(l)$ and
$F_3/F_2:=\Hom(E,{\cal O}_C(l))^{\vee} 
\otimes{\cal O}_C(l)$.
\end{NB}

\subsection{Large volume limit}
We consider the stability condition $\sigma_{(\beta,\omega)}$
in \ref{subsect:stab}.
This stability condition is very similar to
that for an abelian surface.
Thus almost same results in \cite{MYY:2011:1},\cite{MYY:2011:2} and 
\cite{YY2} hold.
In particular results in \cite[sect. 4.1]{YY2} hold.
\begin{NB}
For a fixed Mukai vector $v\in H^*(X,{\Bbb Q})_{\alg}$, 
let us recall the expression of $v$ 
defined in \eqref{eq:Mukai-vector}:
\begin{align}\label{eq:v}
v=r e^\beta+a \varrho_X+(d H+D+(d H+D,\beta)\varrho_X),
\quad d_\beta>0,\;
D \in H^\perp \cap \NS(X)_{\Bbb Q}.
\end{align}
In the same way, for $E_1 \in {\bf D}(X)$, we set
\begin{align}
\label{eq:v(E_1)}
v(E_1)=r_1 e^\beta+a_1 \varrho_X+(d_1 H+D_1+(d_1 H+D_1,\beta)\varrho_X),
\quad
D_1 \in H^\perp \cap \NS(X)_{\Bbb Q}.
\end{align}
We will use these notations freely in this section.
\\

We consider the following conditions for $v$:
\begin{itemize}
\item[$(\star 1)$]
(1) $r \geq 0$, $d>0$ and
(2) $dr_1-d_1 r>0$ implies 
$(dr_1-d_1 r)(\omega^2)/2-(d a_1-d_1 a) > 0$ for
all $E_1$ with $0<d_1< d$ and
$( v(E_1)^2 ) -(D_1^2) \geq 0$.
\begin{NB2}
Thus $\phi(E_1) <\phi(E)$.
\end{NB2}
\begin{NB2}
$r_1$ may be bigger than $r$.
The assumption implies that
$dr_1>d_1r \geq 0$. Hence we have $r_1>0$.
\end{NB2}

\item[$(\star 2)$]
(1) $r \geq 0$, $d<0$ and 
(2) $dr_1-d_1 r < 0$ implies 
    $(dr_1-d_1 r)(\omega^2)/2-(d a_1-d_1 a) < 0$ for
    all $E_1$ with $d < d_1 < 0$ and
    $(v(E_1)^2 ) -(D_1^2) \geq 0$.
\begin{NB2}
Thus $\phi(E_1) > \phi(E)$.
\end{NB2}

\item[$(\star 3)$]
(1) $r \geq 0$, $d<0$ and 
(2) $dr_1-d_1 r \leq 0$ implies 
    $(dr_1-d_1 r)(\omega^2)/2-(d a_1-d_1 a) \leq 0$ 
    and the inequality is strict if $dr_1-d_1 r < 0$ 
    for all $E_1$ with $d \leq d_1 \leq 0$ and
    $( v(E_1)^2 ) -(D_1^2) \geq 0$.
\begin{NB2}
Thus $\phi(E_1) > \phi(E)$.
\end{NB2}

\begin{NB2}
\item[$(\star )$]
(1) $d<0$ and (2) 
$dr_1-d_1 r \leq 0$ implies 
$(dr_1-d_1 r)(\omega^2)/2-(d a_1-d_1 a) \leq 0$ for
all $E_1$ with $d \leq d_1 \leq 0$ and
$\langle v(E_1)^2 \rangle \geq -2$.
Thus $\phi(E_1) \geq \phi(E)$.
\end{NB2}
\end{itemize}
\end{NB}

We set 
$$
v:=e^\beta(r+a_\beta \varrho_X+d_\beta H+D_\beta), D_\beta \in H^\perp.
$$
\begin{NB}
If $\chr \ne 0$, then
assume that the Bogomolov inequality holds for any semi-stable objects
in ${\frak C}$.
\end{NB}
By Lemma \ref{lem:weak-Bogomolov},
 \cite[Lem. 3.1.8]{MYY:2011:1} holds, where
\begin{equation}\label{eq:omega^2}
(\omega^2)>
\begin{cases}
\frac{d_\beta-d_{\min,\beta}}{r d_0}((v^2)-(D_\beta^2)),\;r>0\\
\frac{d_\beta-d_{\min,\beta}}{d_0}((v^2)-(D_\beta^2))+
\frac{2d_\beta|a_\beta|}{d_0},\;r=0,\\
\frac{-d_\beta+d_{\min,-\beta}}{r d_0}((v^2)-(D_\beta^2)),\;r<0,\\
\end{cases}
\end{equation}
$d_{\min,\beta}:=\min\{d_\beta(E)>0 \mid E \in K(X) \}$
and 
$d_0=\frac{1}{(H^2)}\min \{(D,H)>0 \mid D \in \Pic(X) \}$.

We set $\beta=\xi/r'$, $\xi \in \NS(X)$.
Then 
\begin{equation}
(c_1(E)-r\beta,H)=\frac{(r' c_1(E)-r \xi,H)}{r'}
\in {\Bbb Z}\frac{d_0 (H^2)}{r'}.
\end{equation}

If $\gcd(\frac{(c_1(E),H)}{d_0 (H^2)},r)=1$, that is,
$r' \frac{(c_1(E),H)}{d_0 (H^2)}-r a=1$ ($r', a \in {\Bbb Z}$),
then for $\xi \in \NS(X)$ with $(\xi,H)=a d_0 (H^2)$,
we have 
$$
(c_1(E)-r \xi/r',H)=\frac{(r' c_1(E)-r \xi,H)}{r'}
=\frac{d_0(H^2)}{r'}=d_{\min,\xi/r'} (H^2).
$$

\begin{NB}
\begin{rem}
We don't know when the Bogomolov inequality holds
for Bridgeland stable objects.
However we have a weak form of inequality
$(v^2)-(D_\beta^2)\geq 0$ for any pair $(\beta,tH)$.
Indeed for any stable object $E$ of ${\frak C}$,
$(v(E)^2)-(D_\beta(E)^2) \geq 0$.
Then by using \cite[(3.3)]{MYY:2011:1}, we can study the wall-crossing
behavior, and get the inequality.  
\end{rem}
\end{NB}
By the same proof of \cite[Prop. 3.2.1, Prop. 3.2.7]{MYY:2011:1},
we have the following results.
\begin{prop}\label{prop:large}
Assume that $\omega$ satisfies \eqref{eq:omega^2} with respect to $v$.
Then $E \in {\bf D}(X)$ is 
$\sigma_{(\beta,\omega)}$-semi-stable with $\phi(E) \in (0,1)$
if and only if 
\begin{enumerate}
\item[(i)] $\rk v \geq 0$, $E \in {\frak C}$ and
is $(\beta-\tfrac{1}{2}K_X)$-twisted semi-stable
or 
\item[(ii)]
 $\rk v<0$, $E^{\vee}[1] \in {\frak C}^*$ and is
$-(\beta+\tfrac{1}{2}K_X)$-twisted semi-stable.
\end{enumerate}
Thus we have 
\begin{equation}
{\cal M}_{(\beta,\omega)}(v) \cong
\begin{cases}
{\cal M}_\omega^{\beta-\frac{1}{2}K_X}(v),\; \rk v \geq 0,\\
{\cal M}_\omega^{-\beta-\frac{1}{2}K_X}(-v^{\vee}),\; \rk <0.
\end{cases}
\end{equation}
\end{prop}

\begin{NB}
\begin{rem}
Since $K_X$ is relatively trivial for $\pi$,
$\beta \pm \tfrac{1}{2}K_X$ also satisfies the condition.
\end{rem}
\end{NB}

\begin{proof}
Let $E$ be a $\sigma_{(\beta,\omega)}$-semi-stable object
for $(\omega^2) \gg 0$.
By using Lemma \ref{lem:weak-Bogomolov2} for 
$(\beta-\frac{1}{2}K_X)$-semi-stable objects of ${\frak C}$
and $-(\beta+\frac{1}{2}K_X)$-semi-stable object of ${\frak C}^*$,
we see that $E$ is $(\beta-\frac{1}{2}K_X)$-semi-stable
or $E^{\vee}[1]$ is $-(\beta+\frac{1}{2}K_X)$-semi-stable  
(\cite[Prop. 3.2.1, Lem. 3.2.3, Prop. 3.2.7]{MYY:2011:1}). 
In particular, $E$ satisfies Proposition \ref{prop:weak-Bogomolov2}.
For a general $\omega$, 
we use
\begin{equation}
\frac{((v_1+v_2)^2)-(D_1+D_2)^2}{(d_1+d_2)^2}=
\frac{(v_1^2)-(D_1^2)}{d_1^2}+\frac{(v_2^2)-(D_2^2)}{d_2^2}+
\left(\frac{a_1}{d_1}-\frac{a_2}{d_2} \right)
\left(\frac{r_1}{d_1}-\frac{r_2}{d_2} \right)
\end{equation}
for 
$$
v_i:= e^\beta(r_i+d_i H+D_i+a_i \varrho_X),\; D_i \in H^\perp.
$$
to study the wall crossing behavior.
Then Proposition \ref{prop:weak-Bogomolov2}
follows by the induction on $d_\beta$.

Thanks to
Proposition \ref{prop:weak-Bogomolov2}, the converse direction also holds. 
For more detail, see \cite{MYY:2011:1}.
\end{proof}

\begin{rem}
If ${\frak C}^\beta=\Coh(X)$, then
the claim is 
a refinement of results of Lo and Qin \cite{Lo-Qin} or Maciocia \cite{Ma}.
\end{rem}

\begin{NB}
\begin{lem}[{\cite[Cor. 2.2.8]{MYY:2011:1}}]
Assume that $Z_{(\beta,\omega)}(v) \in {\Bbb H}$.
If $(\omega^2)$ is sufficiently large depending only on
$\beta$ and $v$, then
${\cal M}_{(\beta,\omega)}(v)={\cal M}_\omega^{\beta-\frac{1}{2}K_X}(v)$
for $\rk v>0$ and
${\cal M}_{(\beta,\omega)}(v)=
{\cal M}_\omega^{-\beta-\frac{1}{2}K_X}(-v^{\vee})$
for $\rk v<0$. 
\end{lem}
\end{NB}

The following claim is due to Maciocia \cite[Thm. 3.13]{Ma}
if ${\frak C}^\beta=\Coh(X)$
\begin{lem}[{\cite[sect. 4]{MYY:2011:1}}]\label{lem:mini-wall}
We fix $\beta \in \NS(X)_{\Bbb Q}$.
Assume that 
$d_\beta(v)>0$ and
\begin{equation}
v_1 \in {\Bbb Q}v
\Longleftrightarrow
(r(v_1),d_\beta(v_1),a_\beta(v_1)) \in
{\Bbb Q}(r(v),d_\beta(v),a_\beta(v))
\end{equation}
for all $v_1 \in H^*(X,{\Bbb Q})_{\alg}$ with
$0<d_\beta(v_1)<d_\beta(v)$.
Then the set of walls intersecting
$\{(\beta,tH) \mid t >0 \}$
is finite.
\end{lem}

\begin{proof}
Let $B \in {\Bbb Z}_{>0}$ be the denominator of $\beta$.
Assume that we have a decomposition
$v=v_1+v_2$ such that
there are $\sigma_{(\beta,\omega)}$-semi-stable
objects $E_i$ with $v(E_i)=v_i$ $(i=1,2)$ and
$Z_{(\beta,\omega)}(v_1),Z_{(\beta,\omega)}(v_2)
\in {\Bbb R}_{>0}Z_{(\beta,\omega)}(v)$.
We set 
\begin{equation}
\begin{split}
v:= & e^\beta(r+d H+D+a \varrho_X),\; D \in H^\perp\\
v_i:= & e^\beta(r_i+d_i H+D_i+a_i \varrho_X),\; D_i \in H^\perp.
\end{split}
\end{equation}
Then
\begin{equation}
\frac{(\omega^2)}{2}(d r_i-d_i r)=
(d a_i-d_i a). 
\end{equation}
By Lemma \ref{lem:weak-Bogomolov2},
$(v_i^2)-(D_i^2)=d_i^2 (H^2)-2r_i a_i \geq 0$.
Since
\begin{equation}
\frac{(v_1,v_2)-(D_1,D_2)}{d_1 d_2}=
\frac{(v_1^2)-(D_1^2)}{2d_1^2}+
\frac{(v_2^2)-(D_2^2)}{2d_2^2}+
\left(\frac{r_1}{d_1}-\frac{r_2}{d_2} \right)
\left(\frac{a_1}{d_1}-\frac{a_2}{d_2} \right),
\end{equation}
$(v_1,v_2)-(D_1,D_2)>0$.
Hence
$0 \leq  (v_i^2)-(D_i^2) <(v^2)-(D^2)$ and
$(v_i,v-v_i)-(D_i,D-D_i)>0$.
Since the choice of $d_i$ are finite,
$r_i a_i$ are bounded.
$(v^2)-(D^2) \geq 2((v_i,v)-(D_i,D))>0$
implies $ra_i+r_i a$ is also bounded.
Therefore $r_i,a_i$ are bounded.
\begin{NB}
If $r_i=0$, then $ra_i$ is bounded.
If $a_i=0$, then $r_i a$ is bounded.
If $r=r_i=0$, then $a_i=ad_i/d$.
If $a=a_i=0$, then $r_i=r d_i/d$.
Hence if $(r,a) \ne (0,0)$, then
the choice of $r_i,a_i$ is finite.
If $r=a=0$, then $a_i=r_i (\omega^2)/2$, which implies
$r_i=0$ if and only if $a_i=0$.
Therefore the choice of $r_i,a_i$ is finite in this case. 
\end{NB}
Since $r_i \in {\Bbb Z}$ and $2 B^2 a_i \in {\Bbb Z}$,
\begin{NB}
$a_i=-(e^\beta,v_i)$
\end{NB}
the choice of $r_i$ and $a_i$ are finite.
Therefore the set of walls is finite.
\end{proof}

\subsection{Examples of Gieseker chamber}\label{subsect:ex-Gieseker}
Let $X$ be a surface with $\Pic(X)={\Bbb Z}H$.
Let $v=(r,H,a)$, $r \geq 0$.
We note that
\begin{equation}\label{eq:min}
\deg E=\min\{\deg E'>0 \mid E' \in K(X)\}
\end{equation}
for $E$ with $v(E)=v$.
\begin{lem}\label{lem:r=0}
Assume that 
\begin{equation}\label{eq:a-cond}
a \leq -\frac{1}{2}(K_X,H)-r-1.
\end{equation}
\begin{enumerate}
\item[(1)]
For all $L \in M_H(0,H,a)$,
$H^0(X,L)=0$ and $\dim \Ext^1(L,{\cal O}_X) \geq r+1$.
\item[(2)]
For $L \in M_H(0,H,a)$, 
we set $n:=\dim \Ext^1(L,{\cal O}_X)$.
Then there is a family of stable sheaves
\begin{equation}
0 \to {\cal O}_X \otimes V \to E \to L \to 0
\end{equation}
parameterized by $Gr(n,r)$, where $V$ is the universal subspace of
dimension $r$.
\end{enumerate}
\end{lem}

\begin{NB}
If $H^1(X,{\cal O}_X)=0$, then
$\dim \Ext^1(E,{\cal O}_X)=n-r$.
\end{NB}

\begin{proof}
We note that $(K_X,H) \geq (H^2)$ by $\NS(X)={\Bbb Z}H$.
(1)
Since $\deg L=a+\frac{(H^2)}{2}<0$,
$H^0(X,L)=0$.
Since $\Hom(L,{\cal O}_X)=0$ and
$\chi(L,{\cal O}_X)=\chi(L(K_X))=a+\frac{(H^2)}{2}+(K_X,H)+\chi({\cal O}_H)
=a+\frac{(K_X,H)}{2}$, we get (1).

(2) is well-known by \eqref{eq:min}.
\end{proof}

\begin{NB}
If $a \ll 0$, then there is a stable vector bundle $E$ 
fitting in
\begin{equation}
0 \to {\cal O}_X \to E \to F \to 0,
\end{equation}
where $F$ is a stable sheaf with $v(F)=(r-1,H,a)$.

\begin{proof}
Assume that $\chi(F) \leq 0$ and
$\chi(F(K_X)) \leq -2$.
Since $a$ is large, we may also assume that
$H^0(F)=0$ for a general $F \in M_H(r-1,H,a)$.
For a non-zero element of $\Ext^1(F,{\cal O}_X)$,
we have a stable sheaf $E$.
\end{proof}
\end{NB}
We note that
$$
\NS(X)_{\Bbb R} \times \Amp(X)_{\Bbb R}=
\{(sH,tH) \mid s \in {\Bbb R}, t>0 \}. 
$$
By \cite{Ma}, walls are defined by
circles containing
$(\frac{1}{r}-\sqrt{\frac{(v^2)}{r^2(H^2)}},0)$
for $r>0$ and
containing $(\tfrac{a}{(H^2)},0)$.
Since $(v^2)>(H^2)$, the center is in $s<0$.
By \eqref{eq:min},
all walls are in $s \leq 0$. 
For a Mukai vector $w=(r',d' H,a')$, we set 
$$
W_w:=\{(s,t) \mid {\Bbb R}Z_{(sH,tH)}(v)={\Bbb R}Z_{(sH,tH)}(w) \}.
$$
It is a candidate of walls for $v$.
By \cite[Lem. 5.5, Prop. 5.6]{YY2},
$\{ W_w \}$ forms a pencil of circles.
Hence there is a unique
$W_w$ passing $(0,0)$.
By \cite[Lem. 5.11]{YY2},
it is defined by $w=v({\cal O}_X)=(1,0,0)$, and the equation is
$$
t^2+s\left(s-\tfrac{2}{(H^2)}a\right)=0.
$$
In particular, it is independent of the choice of $r$. 
If $w$ defines a wall, then
$W_w$ is contained in $W_{(1,0,0)}$.
In particular
${\cal M}_{(sH,tH)}(r,H,a)={\cal M}_H(r,H,a)$
in the exterior of $W_{(1,0,0)}$.
Hence all $E \in {\cal M}_H(r-1,H,a)$ are $\sigma_{(sH,tH)}$-semi-stable
on $W_{(1,0,0)}$.
Therefore $W_{(1,0,0)}$ defines a wall for $v$ and gives
a boundary of the Gieseker chamber
if $a$ satisfies \eqref{eq:a-cond}.

\begin{NB}
Assume that $\sqrt{\frac{(v^2)}{2^2(H^2)}}>\frac{1}{2}$, that is,
$(v^2)>(H^2)$.
Then there is no wall between 
$0<s<\frac{1}{2}$.
Since $v({\cal O}_X)=1$ defines a wall $C$,
the region between $s=\frac{1}{2}$ and $C$ is the Gieseker chamber.
\end{NB}

\end{document}